\documentclass[final,hidelinks,onefignum,onetabnum,onealgnum,oneeqnum,onethmnum]{siamart250211}
\usepackage[T2A]{fontenc}
\usepackage[cp1251]{inputenc}
\usepackage[english]{babel}
\usepackage{latexml}
\usepackage{hyperref}
\hypersetup{
	pagebackref=true,
	colorlinks=true,
	linkcolor=blue,
	filecolor=magenta,      
	urlcolor=cyan,
	pdftitle={Exponentially Convergent Numerical Method for Abstract Cauchy Problem with Fractional Derivative of Caputo Type},
	pdfauthor={Dmytro Sytnyk, Barbara Wohlmuth}
	pdfpagemode=FullScreen,
}
\usepackage[capitalize,nameinlink]{cleveref}
\usepackage{latexsym,amsfonts,amssymb,amsmath}
\usepackage{graphics,graphicx,longtable,array,euscript,epsfig}
\graphicspath{{./pic/}}
\usepackage{algorithm} 
\usepackage{algorithmic}
\iflatexml

\else
\usepackage{orcidlink} 
\usepackage{overpic}
\usepackage[hyperpageref]{backref} 
\renewcommand*{\backrefalt}[4]{%
	\hypersetup{linkcolor=gray}%
	\color{gray}{%
		[%
		\ifcase #1 %
		No citations%
		\or
		Cited on p. #2%
		\else
		Cited on pp. #2%
		\fi
		]%
	}
}
\usepackage[author={Dmytro},draft]{fixme}
\fi
\usepackage{enumitem} 
\usepackage[section]{placeins} 

\iflatexml
\newtheorem{remark}{Remark} 

\renewtheorem{theorem}{Theorem}
\newtheorem{example}{Example} 
\newtheorem{lemma}{Lemma}
\newtheorem{proposition}{Proposition}
\newtheorem{corollary}{Corollary}
\crefname{example}{Example}{Examples} 
\crefname{remark}{Remark}{Remarks} 
\crefname{subsection}{Section}{Sections} 
\crefname{section}{Section}{Sections} 
\else
\newsiamremark{remark}{Remark} 
\theoremheaderfont{\normalfont\bfseries}
\renewtheorem{remark}{Remark} 
\renewtheorem{definition}{Definition}
\renewtheorem{theorem}{Theorem}
\newtheorem{example}{Example} 
\renewtheorem{lemma}{Lemma}
\renewtheorem{proposition}{Proposition}
\renewtheorem{corollary}{Corollary}
\crefname{example}{Example}{Examples} 
\crefname{remark}{Remark}{Remarks} 
\crefname{subsection}{Section}{Sections} 
\crefname{section}{Section}{Sections} 

\fi
\newcommand{\R}{\mathbb R}
\newcommand{\N}{\mathbb N}

\newcommand{\cE}{\mathcal{E}}

\newcommand{\cF}{\mathcal{F}}

\DeclareMathOperator{\Arg}{Arg}

\DeclareMathOperator*{\Res}{Res}

\newcommand{\suml}{\sum\limits}
\newcommand{\intl}{\int\limits}

\newcommand{\oC}{{\mathbb{C}}}

\newcommand{\e}{\mathrm{e}}

\newcommand{\wt}[1]{\widetilde{#1}}

\newcommand{\colorboxo}[2]{{\pgfsetfillopacity{0.7}\colorbox{#1}{\pgfsetfillopacity{1}#2}}}

\allowdisplaybreaks

\begin{document}
\title{Exponentially Convergent Numerical Method for Abstract Cauchy Problem with Fractional Derivative of Caputo Type}
\iflatexml
	\author{\href{https://orcid.org/0000-0003-3065-4921}{Dmytro Sytnyk}%
		\thanks{%
			Department of Numerical Mathematics, Institute of Mathematics, National Academy of Sciences of Ukraine, Kyiv, 01024, Ukraine; (\href{mailto:sytnik@imath.kiev.ua}{sytnik@imath.kiev.ua}). }%
		\\
		\href{https://orcid.org/0000-0001-6908-6015}{Barbara Wohlmuth}%
		\thanks{Department of Mathematics, Technical University of Munich, Garching, 85748, Germany; (\href{mailto:wohlmuth@ma.tum.de}{wohlmuth@ma.tum.de}).}
	}
\else
	\author{Dmytro Sytnyk\,\orcidlink{0000-0003-3065-4921}%
		\thanks{Department of Mathematics, Technical University of Munich, Garching, 85748, Germany;  (\email{syd@ma.tum.de});
			Department of Numerical Mathematics, Institute of Mathematics, National Academy of Sciences of Ukraine, Kyiv, 01024, Ukraine; (\email{sytnik@imath.kiev.ua}). }%
		\and
		Barbara Wohlmuth\,\orcidlink{0000-0001-6908-6015}%
		\thanks{Department of Mathematics, Technical University of Munich, Garching, 85748, Germany; (\email{wohlmuth@ma.tum.de}).}
	}
\fi
\date{\today}
\maketitle

\begin{abstract}
	We present an exponentially convergent numerical method to approximate the solution of the Cauchy problem for the inhomogeneous fractional differential equation with an unbounded operator coefficient and Caputo fractional derivative in time.
	The numerical method is based on the newly obtained solution formula that consolidates the mild solution representations of sub-parabolic, parabolic and sub-hyperbolic equations with sectorial operator  coefficient $A$ and non-zero initial data.
	The involved integral operators are approximated using the sinc-quadrature formulas that are tailored to the spectral parameters of $A$, fractional order $\alpha$ and the smoothness of the first initial condition, as well as to the properties of the equation's right-hand side $f(t)$.
	The resulting method possesses exponential convergence for positive sectorial $A$, any finite $t$, including $t = 0$ and the whole range $\alpha \in (0,2)$.
	It is suitable for a practically important case, when no knowledge of $f(t)$ is available outside the considered interval $t \in [0, T]$.
	The algorithm of the method is capable of multi-level parallelism.
	We provide numerical examples that confirm the theoretical error estimates.
\end{abstract}


\iflatexml
\else
	\begin{AMS}
		{34A08, 35R11, 34G10, 35R20, 65L05, 65J08, 65J10}
	\end{AMS}
\fi

\section{Problem Formulation and Introduction}

In this paper, we consider a Cauchy problem for the following fractional order differential equation:
\begin{equation}\label{eq:FCP_DE}
	\partial_t^\alpha u + A u = f,\quad  t  \in [0, T] .
\end{equation}
Here,
$\partial_t^\alpha $ denotes the Caputo fractional derivative of order
$\alpha$ with respect to $t$
\[
	\partial_t^\alpha u(t)= \frac{1}{\Gamma(n-\alpha)} \intl_0^t(t-s)^{n-\alpha-1}u^{(n)}(s)\, ds,
\]
where $u^{(n)}(s)$ is the usual integer order derivative,
$n = \lceil \alpha \rceil $ is the smallest integer greater or equal to $\alpha$ and
$\Gamma(\cdot)$ is Euler's Gamma function.
The operator $\partial_t^\alpha $ provides a generalization of the classical differential operator $ \tfrac{\partial}{\partial t} = \partial_t^1$.
For non-integer $\alpha$,  the action of Caputo fractional derivative is essentially nonlocal in time.
In addition to that, the memory kernel from $\partial_t^\alpha $, $\alpha < 1$ has a mild singularity at $0$.
These two facts have a profound impact on the analytical and numerical properties of solutions to fractional differential equation \eqref{eq:FCP_DE}.
If $\alpha  < 1$, this equation is called sub-parabolic.
Similarly, when $\alpha > 1$, the equation is called sub-hyperbolic.
We direct the reader to \cite{Kilbas2006} for a more concise introduction into the subject of fractional derivatives and the theory of associated ordinary differential equations.

The coefficient $A$ in \eqref{eq:FCP_DE}  is assumed to be a closed linear operator  with the domain $D(A)$ dense in a Banach space $X = X(\|\cdot\|,\Omega)$
and the spectrum $\mathrm{Sp}(A)$ contained in the following sectorial region $\Sigma(\rho_s, \varphi_s)$, that is commonly called a spectral angle:
\begin{equation}\label{eq:SpSector}
	\Sigma(\rho_s, \varphi_s) = \left\{ z=\rho_s+\rho \e^{i\theta}:\quad \rho \in [0,\infty), \ \left|\theta\right|< \varphi_s
	\right\}.
\end{equation}
The numbers $\rho_s > 0$ and $\varphi_s < \pi/2$ are called spectral parameters (characteristics) of $A$.
In addition to the assumptions on the location of spectrum, we suppose that the resolvent of $A$:
$
	R\left(z, A\right) \equiv (zI-A)^{-1}
$
satisfies the bound
\begin{equation}\label{eq:ResSector}
	\left\|(zI-A)^{-1}\right\|\leq \frac{M}{1+\left|z\right|}
\end{equation}
outside the sector $\Sigma$ and on its boundary $\Gamma_\Sigma$.
Following the established convention \cite{bGavrilyuk2011}, we will call such sectorial operators strongly positive.
We accompany equation \eqref{eq:FCP_DE} with the usual initial~condition
\begin{subequations}\label{eq:FCP_BC}
	\begin{equation}\label{eq:FCP_BC1}
		u(0) = u_0,
	\end{equation}
	for the solution and additional condition for its derivative, when $1< \alpha < 2$:
	\begin{equation}\label{eq:FCP_BC2}
		u'(0) = u_1.
	\end{equation}
\end{subequations}

The theory of fractional Cauchy problems for differential operators was developed in the works \cite{Oldham1974,Schneider1989,Dzherbashian2020}.
The abstract setting,  considered here, has been theoretically studied \mbox{in \cite{Kochubei1989,Bazhlekova1998,Keyantuo2013}} for $\alpha \in (0, 1)$ then in \cite{Bazhlekova2001} for $\alpha \in [1, 2)$ and, most recently, in \cite{SytnykWohlmuth2023}.
In the current work, we focus on the numerical evaluation of the mild solution to problem \eqref{eq:FCP_DE},
\eqref{eq:FCP_BC} that is given by the following result.
\begin{theorem}[\cite{SytnykWohlmuth2023}]\label{thm:FCP_sol_rep}
	Let $\alpha \in (0,2)$ and $A$ be a sectorial operator with the domain $D(A)$ and
	the spectral parameters $\rho_s > 0$, $\varphi_s < \pi\min{\left\{\frac{1}{2}, \left(1 - \tfrac{\alpha}{2}\right)\right\}}$.
	If  $f \in W^{1,1}([0, T], X)$ and $u_0, u_1 \in D(A)$, then there exists a mild solution $u(t)$ of problem \eqref{eq:FCP_DE}, \eqref{eq:FCP_BC}  that can be represented as follows:
	\begin{equation}\label{eq:FCP_InhomSol_rep}
		\begin{aligned}
			u(t) = S_\alpha(t) u_0 + S_{\alpha,2}(t) u_1 + J_\alpha S_{\alpha}(t)f(0) \!+\!\! \intl_0^t S_{\alpha}(t - s)   J_\alpha f'(s) d s.
		\end{aligned}
	\end{equation}
	Here, $J_\alpha$ stands for the Riemann--Liouville integral
	\begin{equation}\label{eq:FCP_RLInt}
		J_\alpha v(t) = \frac{1}{\Gamma(\alpha)} \intl_0^t(t-s)^{\alpha-1}v(s)\, ds ,
	\end{equation}
	the initial vector $u_1 \equiv 0$ for $\alpha \in (0, 1]$  and $S_{\alpha,\beta}(t)$ is defined by
	\begin{equation}\label{eq:FCP_SO_cont_repr}
		S_{\alpha,\beta}(t) x  = \frac{1}{2\pi i} \intl_\Gamma  e^{zt}z^{\alpha-\beta} (z^\alpha I + A)^{-1} x dz, \quad \beta \geq 1.
	\end{equation}
	with $S_\alpha(t) \equiv S_{\alpha,1}(t)$ for short.
	The contour $\Gamma$ is chosen in such a way that the integral in \eqref{eq:FCP_SO_cont_repr}  is convergent and the curve $z^\alpha$, $z \in \Gamma$ is positively oriented with respect to $-\Sigma(\rho_s, \varphi_s) \cup \{0\}$.
\end{theorem}
The bulk of the existing research is devoted to the particular cases of \eqref{eq:FCP_DE} when $A$ is specified as a strongly elliptic linear partial differential or, more generally, pseudo-differential operator with the domain $D(A)$ that is dense in $X$ \cite{Schneider1989, Eidelman2004,Umarov2019}.
These cases also include the fractional powers of elliptic operators are encompassed by the class of strongly positive operators \cite{fujita} considered in \cref{thm:FCP_sol_rep} and below.
In this regard, the shape of $\mathrm{Sp}(A)$ justifies the choice of the range $(0,2)$ for $\alpha$, as a maximally possible for the considered class of $A$ (see \cite{SytnykWohlmuth2023} for a more detailed discussion).

There exists a considerable body of work devoted to numerical methods for evolution fractional differential equations (see \cite{Garrappa2018,Diethelm2019,Diethelm2020,Diethelm2022} and the references therein).
Philosophically, it can be subdivided into methods that directly approximate the components of \eqref{eq:FCP_DE}, or its integral analogue, and those that make use of more elaborate solution approximations.
The methods from the first class are sequential in nature and have algebraic convergence order that typically does not exceed $2$, even for the multi-step methods \cite{Garrappa2015a}, due to the intrinsic fractional-kernel singularity \cite{Stynes_2019}.
In addition, at each time-step, these methods need to query the entire solution history in order to evaluate $\partial_t^\alpha$ or $J_\alpha $, numerically.
In the consequence of that, they are computationally costly and memory constrained.
Nonetheless, the methods from this class are popular due to their simplicity \cite{Stynes2021}, numerical stability \cite{Garrappa2015a} and the ability to handle non-smooth initial data \cite{jin2019numerical}.
The second class of numerical methods is represented by the works \cite{Cuesta2006,Baffet2017,Fischer2019,Guo2019,Khristenko2021}, to name a few.
These methods are based on the clever solution approximations
that result in a time-stepping scheme requiring only a small number of previous solution states for the next state evaluation.
With some exceptions \mbox{(e.g., \cite{Guo2019}),} these methods are also $\mathcal{O}(h^p)$.

Spectral methods from \cite{McLean2010a,McLean2010,Colbrook2022a,Vasylyk2022} deserve a separate mention.
Although formally belonging to the second class, they make use of the exponentially convergent contour-based propagator approximation, which permits to evaluate the transient component of the solution to the linear problem without time-stepping.
The authors of these works, however, do not apply it to~\eqref{eq:FCP_DE}, \eqref{eq:FCP_BC} directly.
Instead, they consider a special proxy problem $\partial_t u + I^{1-\alpha}A u = g$ where $I^\alpha$ is a nonlocal operator equal to $\partial_t^{\alpha}$, if $\alpha<1$, or to $J^\alpha$, otherwise.
It was shown in~\cite{McLean2010}, that the existing methodology for parabolic problems \cite{gm5,Weideman2006,WeidTref1} can be transferred to the mild solution of such proxy problem with all important numerical features of the solution algorithms preserved, including uniform exponential convergence for $t \in [0,T]$ and the capacity for multi-level parallelism.
Despite being simple and efficient, the proxy-problem idea has certain ramifications when applied to \eqref{eq:FCP_DE}, \eqref{eq:FCP_BC}.
Firstly, there is no easy way to incorporate the initial condition from \eqref{eq:FCP_BC2} into the proxy problem formulation, so all existing works consider $u'(0) = 0$.
Secondly, the methods from \cite{McLean2010a,McLean2010,Pang2016,Colbrook2022a} operating on the Laplace transform image of the right-hand side $g$ are prone to errors when the original $f$ from \eqref{eq:FCP_DE} is not given in the closed form.
Hence, they are unsuitable for many applications.
Meanwhile, formula \eqref{eq:FCP_InhomSol_rep}, which serves as a base for our numerical method, does not require any extra knowledge about the right-hand side $f \in W^{1,1}([0, T], X)$ beside the values $f(0)$ and $f'(t)$, $t \in [0, T]$.
In addition to that, the rigorous analysis from \cite{McLean2010,Vasylyk2022} addresses a version of the proxy problem where $\partial_\alpha$ is a Riemann--Liouville (RL) fractional derivative.
Cauchy problems with RL derivative are simpler in the sense of propagator representation \cite{SytnykWohlmuth2023}, but they are compatible with \eqref{eq:FCP_DE}, \eqref{eq:FCP_BC} only under some additional assumptions.

It is fair to point out that the majority of the mentioned methods are designed to handle the nonlinear fractional differential equation, more general than \eqref{eq:FCP_DE}.
With the view of similar nonlinear extensions in mind, in this work we would like to prioritize those properties of the solution method for \eqref{eq:FCP_DE}, \eqref{eq:FCP_BC}, which will make such extensions possible.
Let us for the moment assume that $f = f(t,u)$.
Then, representation \eqref{eq:FCP_InhomSol_rep} can be used as a base for the sequential time-stepping scheme \cite{Hochbruck2005,LopezFernandez2005} or as the fine propagator in a more parallelization-friendly ParaExp-type scheme \cite{Gander2013}.
In both cases, the method will be free of the issues with approximating $\partial_t^\alpha u$ in the vicinity of $t=0$, provided that the proposed \textit{approximation of~\eqref{eq:FCP_InhomSol_rep} converges uniformly}.
Such application scheme also justifies the use of \textit{a moderate in size final time $T$}.
If, more generally, we assume that $Au=A(t,u)$, then the problem in question can be reduced to \eqref{eq:FCP_DE}, \eqref{eq:FCP_BC} using collocation \cite{Bohonova_2008,bGavrilyuk2011} or a similar in nature time-stepping scheme, inspired by \cite{Gonzalez2016}.
In such scenario, $A(t,u)$ is approximated  by $A(t_k, u_k)$ having spectral characteristics that may vary drastically with $k$ (see Cahn--Hilliard equation from \cite{Fritz2022}, for instance), and the right-hand side in the form $A(t_k,u_k)-A(t,u)$, which makes sense only locally.
Thus, the solution method should be able to reliably handle \textit{operators with arbitrary spectral parameters} and \textit{right-hand sides that are unknown a priori}.

Taking the aforementioned properties into account, below we devise an exponentially convergent approximation for \eqref{eq:FCP_InhomSol_rep} by building upon a well-established technique \cite{McLTh, thomee1,Weideman2010,gm5} that involves the application of a trapezoidal quadrature rule to the parametrized contour integral from \eqref{eq:FCP_SO_cont_repr}.

In \cref{sec:FCP_hyp_contour} we study a question regarding the choice of the suitable integration contour for such parametrization.
The proposed time-independent hyperbolic contour $\Gamma = \Gamma_I$ is valid for the wide class of sectorial operators $A$ with fixed $\varphi_s < \pi\min{\left\{\frac{1}{2}, \left(1 - \tfrac{\alpha}{2}\right)\right\}}$ and arbitrary $\rho_s > 0$.
The parameters of $\Gamma_I$ are derived using the set of constraints that utilizes all available analyticity of the propagator, and therefore, maximize theoretical convergence speed of the sinc-quadrature applied to $S_\alpha(t)$.
\Cref{sec:FCP_num_method} is devoted to the development and justification of the numerical method.
Using the moderate smoothness assumption $u_0 \in D(A^\gamma)$, $\gamma > 0$,  in \cref{sec:FCP_Prop_Approx} we propose an exponentially convergent approximation of $S_\alpha(t)u_0$, that does not degrade for small $t$ like the similar methods from \cite{Fischer2019,Rieder2023}.
Additionally, the approximation is numerically stable for sectorial operators with the spectrum arbitrary close to the origin.
This new result is made possible by extending the idea of resolvent correction, originally introduced for $S_1(t)$ in \cite{gm5},
to the class of abstract integrands with a scalar-part singularity; see \cref{lem:FCP_Pro_parametrized}, below.
In \cref{sec:FCP_hom_sol_appr,sec:FCP_inhom_sol_appr}, we apply the developed approximation of $S_\alpha(t)$ to turn solution representation \eqref{eq:FCP_InhomSol_rep} into the exponentially convergent numerical method.
A priori error estimates given by \cref{thm:FCP_prop_appr,thm:FCP_inhom_sol_err_est} characterize the method's convergence in terms of the smoothness of $u_0$, $f'(t)$, values $\alpha$, $\varphi_s$ and the size of $T$.

The implementation details are provided by \cref{alg:FCP_hom_sol_appr,alg:FCP_inhom_sol_appr} which are capable of multilevel parallelism: at the level of solution evaluation for each of the desired $t$'s; at the level of evaluating resolvents for the set of different quadrature points $z_m$ and at the level of solving stationary problem that pertains to the resolvent evaluation for the fixed $z_m$.

The mentioned numerical properties are experimentally verified in \cref{ex:FCP_ex1_hom_R_eigenfunction}  and \cref{ex:FCP_ex2_inhom_R_eigenfunction}, for the homogeneous and inhomogeneous part of the solution, respectively.
Both examples consider the negative Laplacian with tunable spectral characteristics in place of $A$ and a conventional eigenfunction-based initial data.
Such restriction on the form of initial data permits us to evaluate the space component of solution explicitly, thus removing its contribution to the overall error.
The restriction is relaxed in \cref{ex:FCP_ex3_inhom_R_FD}, which is devoted to the experimental analysis of a fully discretized numerical scheme based on the combination of the developed method with a finite-difference stationary solver.
In all three examples, a stable numerical behavior of the approximated solution is observed for $\alpha \in [0.1, 1.9]$  \mbox{and $T \leq 5$.}

\section{Contour of Integration}\label{sec:FCP_hyp_contour}
It is well known that the choice of integration contour $\Gamma$ in \eqref{eq:FCP_SO_cont_repr} is critical to the performance of the numerical evaluation of operator function based on the contour integral representation
\cite{WeidTref1,Weideman2010,gm5,lopez-fernandez1}.
Judicious contour selection involves the analysis of the interplay between the shape of the integration contour, analytical properties of the parametrized integrand and their impact on the performance of a quadrature rule, that is used to evaluate the resulting integral numerically.
The authors of \cite{bGavrilyuk2011} showed that the hyperbolic contour is the most convenient choice for the quadrature-based numerical evaluation of abstract functions with sectorial operator argument.
Below, we extend their analysis to the case of fractional propagator $S_\alpha(t)$.

Let us consider the following hyperbolic contour:
\begin{equation}\label{eq:int_cont_hyp_crit_sec}
	\Gamma_I: z(\xi) = a_0 - a_I \cosh(\xi) + i b_I \sinh(\xi), \quad \xi \in	(-\infty, \infty ),
\end{equation}
with the parameters $a_0$, $a_I$, $b_I$ that are called shift, first and second semi-axes, respectively.
Admissible range of values for these parameters is determined from  \cref{thm:FCP_sol_rep} that enforces the integration contour $\Gamma = \Gamma_I$  to encircle the singularities of the integrand in \eqref{eq:FCP_SO_cont_repr} for $\beta =1,2$.
The integral is convergent for $t \geq 0$ if $\Re{z(\xi)}\rightarrow -\infty$ ($\xi \rightarrow \infty$), because in such case the norm of integrand on $\Gamma$ will decay faster than the exponential.
This observation transforms into the condition $a_I>0$ for the first semi-axis of hyperbola from \eqref{eq:int_cont_hyp_crit_sec}.
The condition $b_I>0$ for the second semi-axis is induced by the orientation of $z^\alpha (\xi)$.
We also have to make sure that this curve does not intersect the spectrum of $-A$.
It is worth noting that, for any $\varphi \in [0,\pi]$, the function $z^\alpha$ maps the sector $\Sigma(0, \varphi/\alpha)$ into the sector $\Sigma(0, \varphi)$.
Such mappings can be associated with the Dunford--Cauchy representation of the fractional powers of $A$ \cite{Ashyralyev2009,bGavrilyuk2011}.
They are often studied in the theory of fractional resolvent families \cite{Li2010} and associated Cauchy problems \cite{Martinez2001}.

For non-negative $a_I,b_I$, the hyperbolic contour $\Gamma_I$ is contained within the region
$\Sigma(a_0 - a_I, \varphi_I) \setminus \Sigma(a_0,  \varphi_I) $.
Here, $\varphi_I$ is the angle between positive real semi-axes and asymptotes of $\Gamma_I$: $a_0 + \rho e^{\pm i\varphi_I}$ depicted in \cref{fig:FCP_hyp_cont}~(\textbf{b}), i.e., $\tan{\varphi_I} = -{\tfrac{b_I}{a_I}}$.
\begin{figure}[h!tb]
	\begin{centering}
		\iflatexml
			\includegraphics[width=0.85\textwidth,permil=true,viewport=8 210 550 364,clip=true]%
			{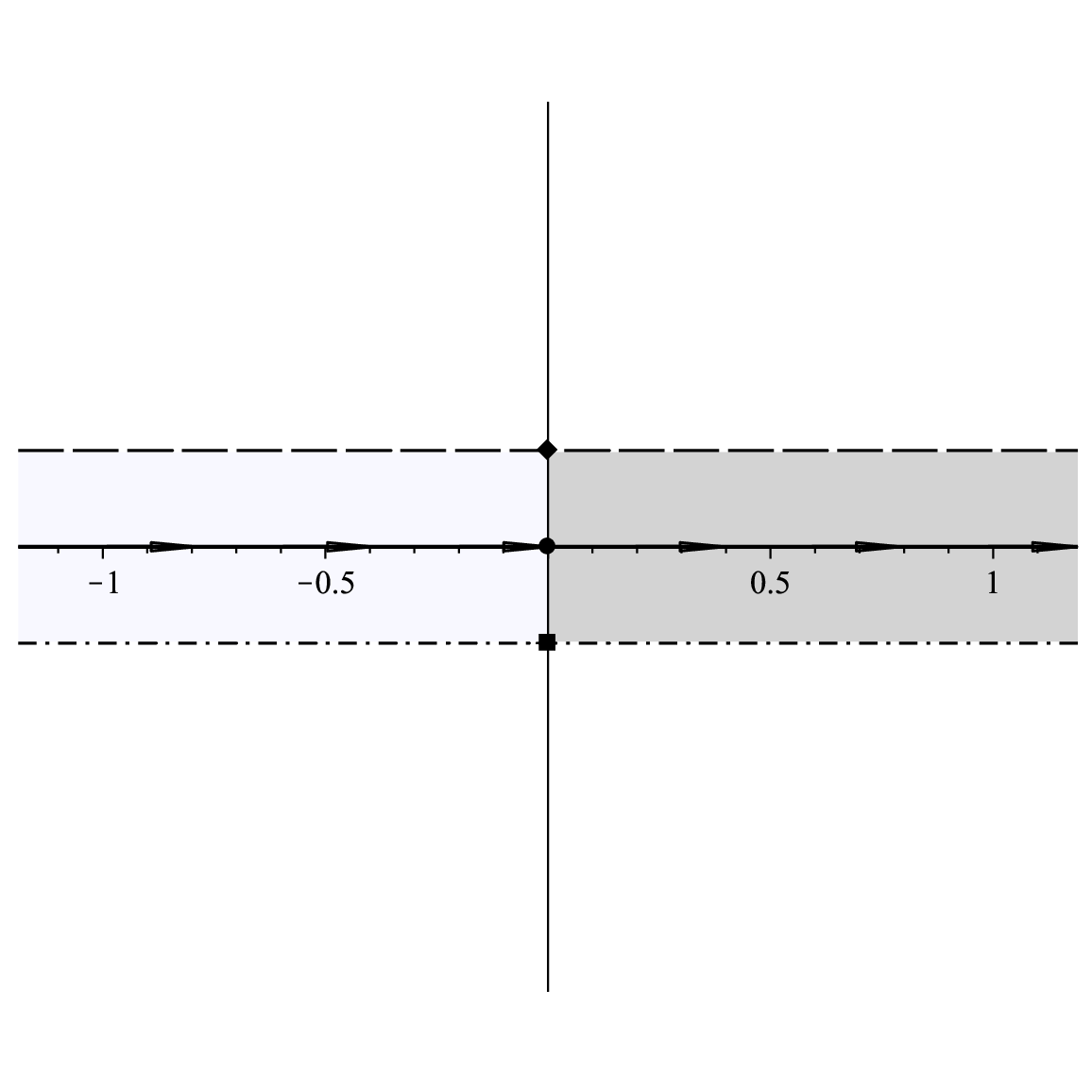}
		\else
			\begin{overpic}[width=0.85\textwidth,permil=true,viewport=8 210 550 364,clip=true]
				{FCP_strip_sphi_pi_6_sa_0_alpha_1_3_ac_pi_2_phic_pi_2_n1}
				\put(470,180){$d$}
				\put(440,50){$-d$}
				\put(970,90){$\xi$}
				\put(505,260){$\nu$}
				\put(0,260){\scriptsize ({\bf a})}
			\end{overpic}%
		\fi
		\\[6pt]
		\iflatexml
			\includegraphics[height=18.5em,viewport=180 207 388 481,clip=true]%
			{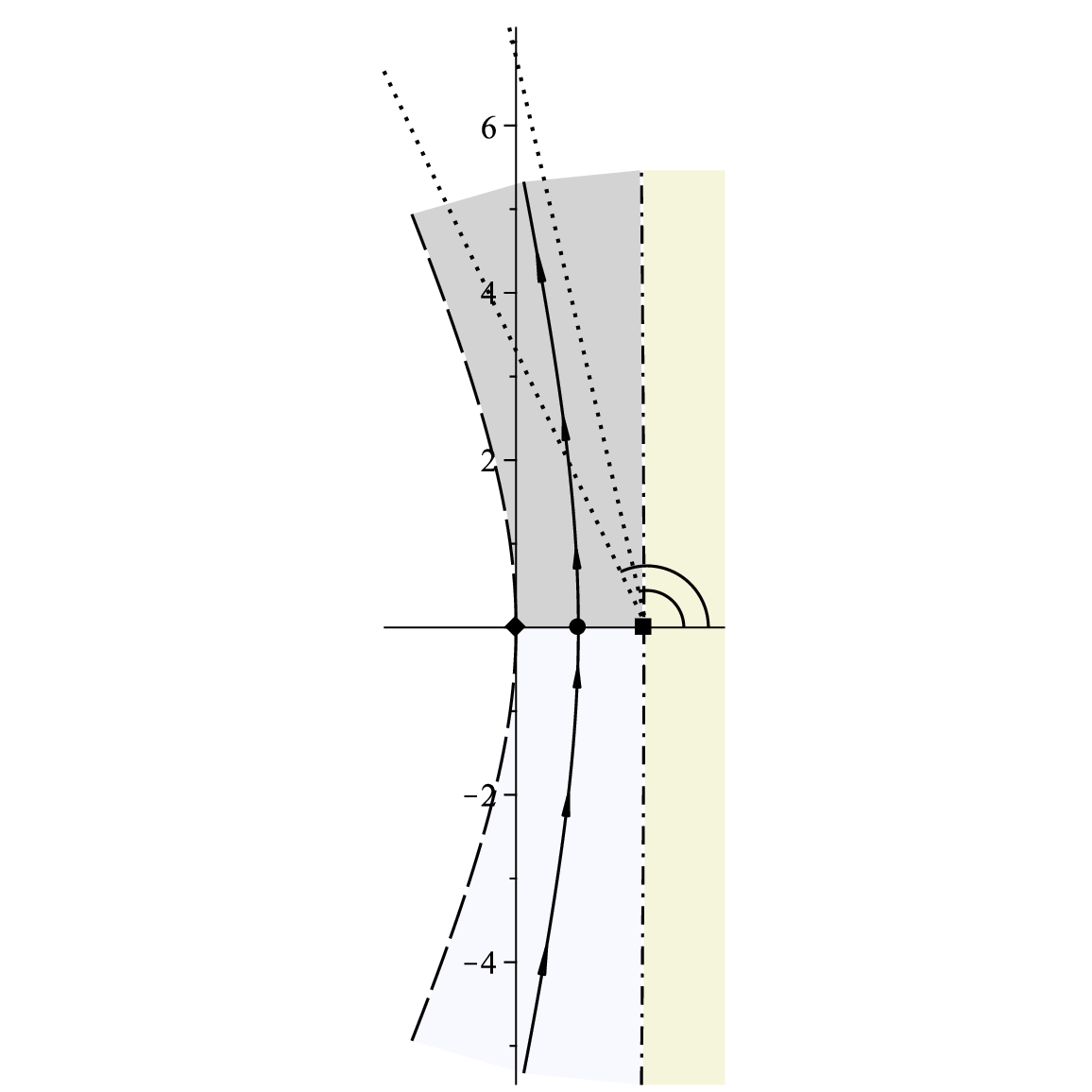}
			\includegraphics[height=18.5em,viewport=180 241 422 541,clip=true]%
			{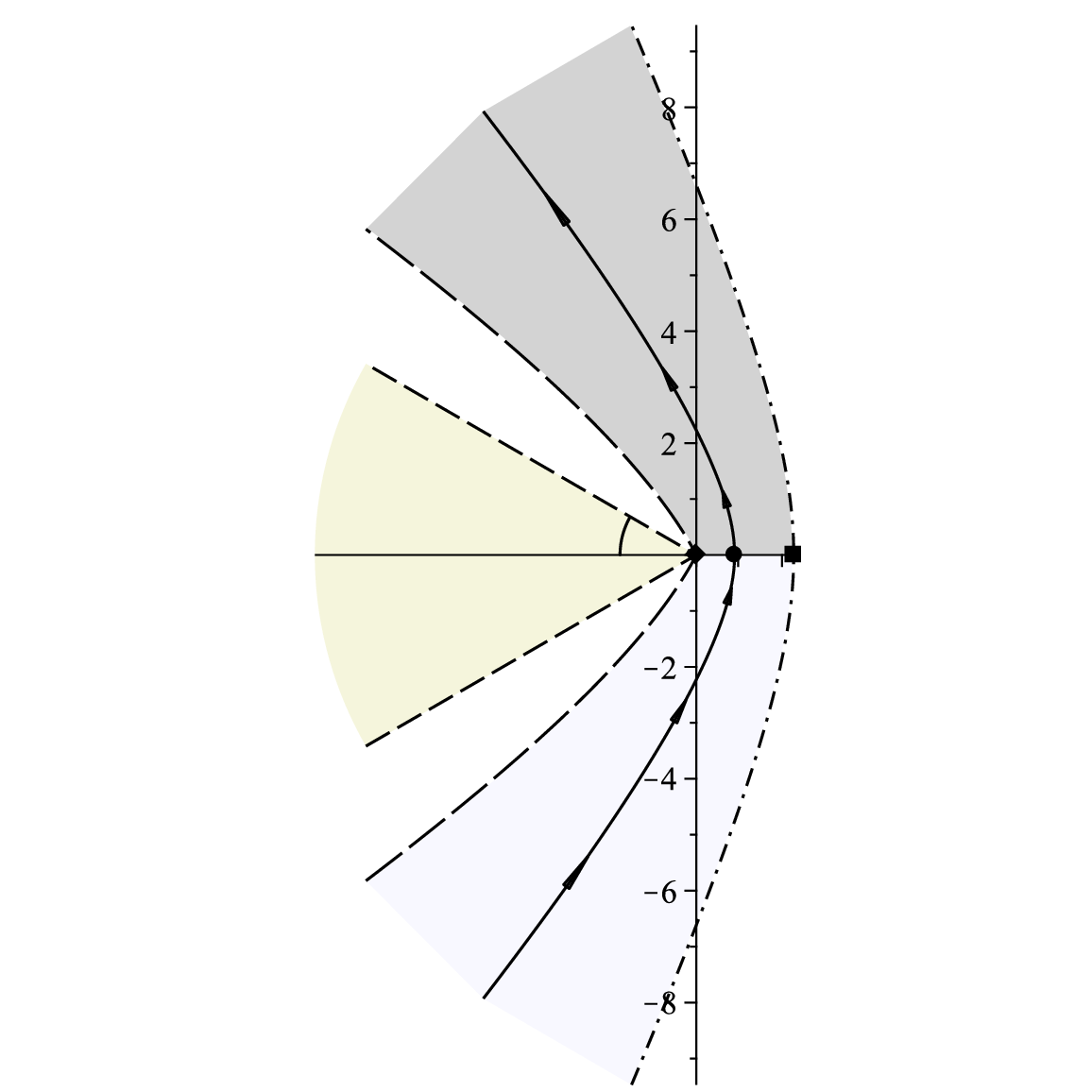}
		\else
			\hspace*{0.25em}
			\begin{overpic}[height=18.5em,viewport=180 207 388 481,clip=true]
				{FCP_scal_sec_cont_sphi_pi_6_sa_0_alpha_1_3_ac_pi_2_phic_pi_2_n1}
				\put(9,73){$\Gamma_s$}
				\put(30,65){$\Gamma_I$}
				\put(46,80){$\Gamma_c$}
				\put(54,13){$\varphi_I$}
				\put(60,23){$\phi_s$}
				\put(0,96){\scriptsize ({\bf b})}
			\end{overpic}%
			\hspace*{2.8em}
			\begin{overpic}[height=18.5em,viewport=180 241 422 541,clip=true]
				{FCP_res_sec_cont_sphi_pi_6_sa_0_alpha_1_3_ac_pi_2_phic_pi_2_n1}
				\put(15,13){\small $\mathrm{Sp}(-A)$}
				\put(46,12){\small $\varphi_s$}
				\put(0,96){\scriptsize ({\bf c})}
			\end{overpic}%
		\fi
		\caption{Schematic plot of the complex neighborhood $D \equiv D_d$  of $\R$ where the parametrized integrand $\cF_\alpha(t,\xi)$ remains analytic and exponentially decaying for any $t \in [0, T]$ (\textbf{a}) along with the image of $D_d$ under the mapping $v \to z(v)$ defined by $\Gamma_I$ (\textbf{b}) and the region $z^\alpha(v)$, $v \in D_d$ (\textbf{c}).
		The "forbidden" regions of complex plane are indicated by ``beige'' color).
		($\alpha =1.3$, $\rho_s = \pi$, $\varphi_s = {\pi}/{6}$).
		}
		\label{fig:FCP_hyp_cont}
	\end{centering}
\end{figure}
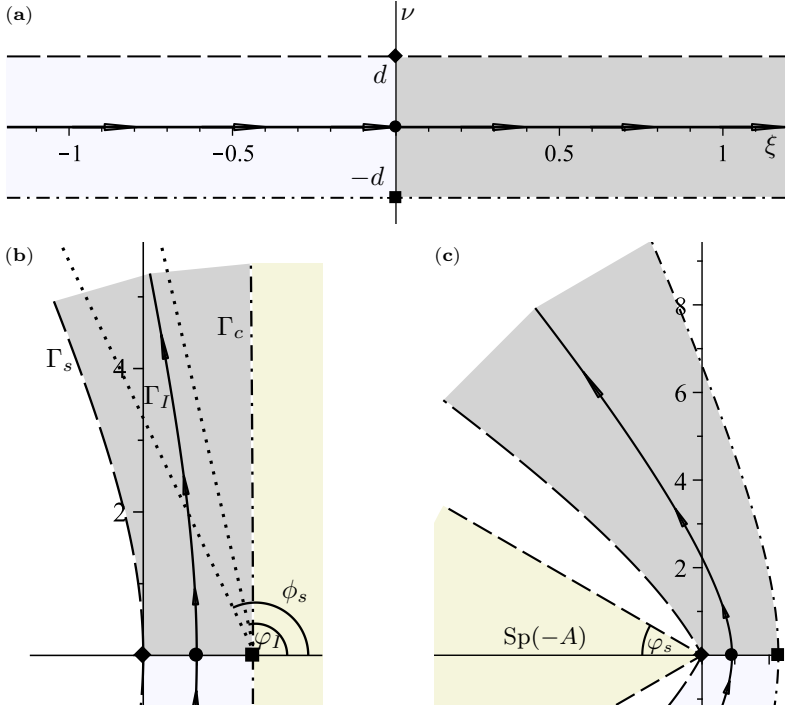
Consequently, the pair of positive contour parameters  $a_I$, $b_I$ is admissible if $z(\xi) \in \Sigma(0, \frac{\pi - \varphi_s}{\alpha}) \setminus \Sigma \left (a,  \tfrac{\pi}{2}\right ) $, for some $a>a_0$, i.e.,
\begin{equation}\label{eq:FCP_hyp_cont_par_adm_constr}
	\begin{aligned}
		\tan{\frac{\pi - \varphi_s}{\alpha}} \leq -\frac{b_I}{a_I}	, \quad
		a_0 - a_I \geq 0.
	\end{aligned}
\end{equation}

Next, we move on to derive exact formulas for $a_0$, $a_I$, $b_I$.
Let us assume that the chosen set of parameters satisfies \eqref{eq:FCP_hyp_cont_par_adm_constr}.
The substitution of $z(\xi)$ from \eqref{eq:int_cont_hyp_crit_sec}  into  \eqref{eq:FCP_SO_cont_repr} yields
\begin{equation}\label{eq:FCP_SO_exp_repr_par}
	S_\alpha(t) x
	= \frac{1}{2\pi i} \intl_{-\infty}^{\infty} \cF_\alpha(t,\xi) x  d\xi, \quad \cF_\alpha(t,\xi)  =  e^{z(\xi)t}z'(\xi) z^{\alpha-1}(\xi) \left (z^\alpha (\xi) I + A\right )^{-1},
\end{equation}
where
$z'(\xi) = -a_I\sinh(\xi) + ib_I\cosh(\xi)$.
The illustration provided by \cref{fig:FCP_hyp_cont} shows that both scalar and operator parts of the parametrized integrand $\cF_\alpha(t,\xi)$, $t \in [0, T]$  remain analytic when $\xi$ is extended into a certain complex neighborhood $D$ of $\R$.
According to the general theory of numerical integration \cite{Davis1984}, an accuracy of quadrature formula is characterized by a norm of the error-term in the Hardy space $\mathbf{H}^p(D)$
of functions, defined on the domain $D \subset \oC$.
The shape of $D$ depends on the chosen type of quadrature.
For the reasons that are soon to be understood,  we approximate integral \eqref{eq:FCP_SO_exp_repr_par} by the sinc-quadrature formula  \cite{Stenger1993,gm5}:
\begin{equation}\label{eq:FCP_SO_sinc_quad}
	S_\alpha(t) x \approx \frac{h}{2 \pi i}\sum_{k=-N}^{N}\cF_\alpha \left (t,kh\right ) x,
\end{equation}
with the discretization parameter $N \in \N$ and the step-size $h= h(N,\cF_\alpha)$.
Then, $D$ is formed by an infinite horizontal strip $D_d$ of the half-height $d$:
\[
	D_d=\left \{z \in \mathbb{C}: - \infty < \Re z < \infty, |\Im z|<d \right \}.
\]

The detailed error analysis of \eqref{eq:FCP_SO_sinc_quad} will be presented in \cref{sec:FCP_Prop_Approx}.
For now, it is sufficient to say that the error of sinc-quadrature decays as  ${\mathcal O}(e^{-\pi d/h})$ if the integrand is exponentially decaying and belongs to $\mathbf{H}^p(D_d)$ \cite{Stenger1993}.
Thus, in order to achieve a faster convergence rate of quadrature \eqref{eq:FCP_SO_sinc_quad}, we need to  maximize the height of the strip $D_d$, where $\cF_\alpha$ remains analytic, by tuning the parameters of $\Gamma_I$.

Let us consider the family of curves $\Gamma(\nu) = \{a_0 - a_I \cosh{(\xi+i v)} + ib_I \sinh{(\xi+i v)}:\, \xi \in (-\infty, \infty)\},$
which extends the definition of $\Gamma_I = \Gamma(0)$ to the arguments with nonzero imaginary part $\nu$.
Observe, that for a fixed $\nu>0$, the curve $\Gamma(\nu)$ is also a hyperbola, albeit with different semi-axes $a(\nu)$, $b(\nu)$:
\begin{equation}\label{eq:cont_hyp_par_family}
	\begin{aligned}
		\Gamma(\nu) & = \{a_0 - a(\nu) \cosh{\xi} + ib(\nu) \sinh{\xi}: \; \xi\in(-\infty,\infty)\}, \\
		a(\nu)      & = a_I \cos{\nu}+b_I\sin{\nu},	 \quad
		b(\nu) = b_I \cos{\nu}-a_I\sin{\nu}.
	\end{aligned}
\end{equation}
Hence, the  mapping $w \rightarrow z(w)$ transforms $D_d$ into the region of complex plane
bounded by two hyperbolas $z(\xi + id)$, $z(\xi - id)$, which will be denoted as $\Gamma_s$ and $\Gamma_c$, correspondingly.
We choose parameters $a_0$, $a_I$, $b_I$, so that $\Gamma_s$ has the vertex at zero
and its asymptotes form the angle $\phi_s \equiv \min{\left\{\pi, \frac{\pi - \varphi_s}{\alpha}\right\}}$ with $\R_+$, as shown in \cref{fig:FCP_hyp_cont}~(\textbf{c}).
In addition, we require that the asymptotes of  $\Gamma_c$ form the angle
$\phi_c \in \left[\tfrac{\pi}{2}, \phi_s\right)$ with $\R_+$,  which will be called a critical angle; see \cref{fig:FCP_hyp_cont}~(\textbf{b}).
The above requirements for $\Gamma_I$, $\Gamma_s$, $\Gamma_c$ are codified in the following system of equations:
\begin{equation*}
	\begin{cases}
		\Re{z(i d )} & = 0,                                                                                \\
		-b(d)        & = a(d)\tan{\phi_s},                                                                 \\
		\tan{\phi_c} & = \lim\limits_{\xi \rightarrow \infty} \frac{\Im{z(\xi - id)}}{\Re{z(\xi - i d )}},
	\end{cases}
	\Leftrightarrow
	\begin{cases}
		a_I \cos{d} + b_I \sin{d}                                   & = a_0,              \\
		a_I \sin{d} - b_I \cos{d}                                   & = a_0\tan{\phi_s} , \\
		\frac{a_I \sin{d} + b_I \cos{d}}{b_I \sin{d} - a_I \cos{d}} & =\tan{\phi_c},
	\end{cases}
\end{equation*}
which is sufficient to ensure \eqref{eq:FCP_hyp_cont_par_adm_constr} and will lead to the maximal possible $d$, when $\phi_c = \pi/2$.
The system composed from the first two equations is linear with respect to $a_I$, $b_I$; thus,
\begin{equation}\label{eq:FCP_hyp_cont_aIbI}
	\begin{aligned}
		a_I
		 & =a_0(\cos{d} + \tan{\phi_s}\sin{d})
		= \frac{a_0}{\cos{\phi_s}} \cos{(d - \phi_s)},
		\\
		b_I
		 & = a_0(\sin{d} - \cos{d}\tan{\phi_s} )
		= \frac{a_0}{\cos{\phi_s}} \sin{(d - \phi_s)}.
	\end{aligned}
\end{equation}
By that means, the left-hand side of the third equation is transformed as
\begin{equation*}
	\frac{a_I \sin{d} + b_I \cos{d}}{b_I \sin{d} - a_I \cos{d}}
	=\frac{\sin{2d} - \tan{\phi_s}\cos{2d}}{-\cos{2d} - \tan{\phi_s}\sin{2d} }
	=\frac{\tan{\phi_s} - \tan{2d}}{1 + \tan{\phi_s}\tan{2d}}
	=\tan{(\phi_s - 2d)},
\end{equation*}
which, after back-substitution, implies $\tan{(\phi_s - 2d)} = \tan{\phi_c}$.
Due to the constraints on $d$, $\phi_c$, $\phi_s$ we are interested only in the following solution
of the last equation:
\begin{equation}\label{eq:FCP_hyp_cont_d}
	d = \frac{1}{2}\left(\phi_s - \phi_c \right).
\end{equation}
For $\phi_c = \pi/2$, and an arbitrary fixed $a_0 > 0$ we obtain
\begin{equation}\label{eq:FCP_hyp_cont_par_final}
	\begin{aligned}
		\phi_s = \min{\left\{\pi, \frac{\pi - \varphi_s}{\alpha}\right\}},
		 & \quad d  = \frac{\phi_s}{2} - \frac{\pi}{4}, \\
		a_I
		= \frac{a_0}{\cos{\phi_s}} \cos{\left (\frac{\phi_s}{2} + \frac{\pi}{4}\right )},
		 & \quad b_I
		= -\frac{a_0}{\cos{\phi_s}} \sin{\left (\frac{\phi_s}{2} + \frac{\pi}{4}\right )}.
	\end{aligned}
\end{equation}
Here, $\alpha \in (0,2)$ is the order of fractional derivative from \eqref{eq:FCP_DE}, $\varphi_s$ is the spectral angle parameter defined in \eqref{eq:SpSector} and $a_0 \in \R_+$ is given.

\section{Numerical Method}\label{sec:FCP_num_method}
To begin with the description of the numerical scheme, let us introduce some notation.
We rewrite formula \eqref{eq:FCP_InhomSol_rep} in the form
\[
	u(t) = u_{\mathrm{h}}(t) + u_{\mathrm{ih}}(t).
\]
Here, $u_{\mathrm{h}}(t)$ denotes the solution to the homogeneous part ($f(t)\equiv0$) of the given \linebreak{problem \eqref{eq:FCP_DE}, \eqref{eq:FCP_BC}} and $u_\mathrm{ih}(t)$ the solution to the inhomogeneous part ($u_0=u_1 \equiv 0$):
\begin{equation}\label{eq:FCP_hom_inhom}
	u_{\mathrm{h}}(t)  = S_{\alpha,1}(t) u_0 + S_{\alpha,2}(t) u_1, \quad
	u_{\mathrm{ih}}(t) =  J_\alpha S_{\alpha}(t)f(0) + \intl_0^t S_{\alpha}(t - s)   J_\alpha f'(s) d s .
\end{equation}

\subsection{Alternative Propagator Representation}
We consider the representation of the solution to the homogeneous part $u_{\mathrm{h}}(t)$ first.
In the seminal paper \cite{gm5}, Gavrylyuk and Makarov showed that the numerical method for $S_1(t) = e^{-At}$ naively obtained from representation \eqref{eq:FCP_SO_cont_repr} is unsuitable for small values of $t$ because its accuracy degrades when $t$ approaches 0.
They traced back the root cause of this behavior to the fact that the considered representation of $e^{-At}$ is, formally speaking, divergent at $t=0$, which result in the unremovable error of the quadrature-based numerical method for such $t$.
It turns out that propagator representation \eqref{eq:FCP_SO_cont_repr} poses the same adverse feature for any fractional $\alpha$.
One could learn more about its impact on a numerical solution of \eqref{eq:FCP_DE} by analyzing the results of works \cite{Fischer2019,Rieder2023}.

In order to get around the divergence issue, we propose an alternative formula for $S_{\alpha,1}(t)$, constructed in the vein of \cite{gmv,gm5}.
It is based on the following proposition, which can be regarded as a generalization of Lemma 3.3 from \cite{McLean2010}.

\begin{proposition}\label{prop:FCP_res_cor}
	Let  $A$ be the sectorial operator satisfying the conditions of \cref{thm:FCP_sol_rep}.
	If $x \in D(A^{m + \gamma})$ and  $z^\alpha \notin \mathrm{Sp}(A) \cup \{0\}$, then
	for any $\gamma \geq 0$
	\begin{equation}\label{eq:FCP_res_cor_norm_est}
		\left\| z^{\alpha-\beta}(z^\alpha I-A)^{-1} x - \frac{1}{z^\beta}\suml_{k=0}^{m}\frac{A^k x}{z^{\alpha k}}\right \|
		\leq \frac{K(1+M) \left \|A^{m+\gamma} x \right \|}{|z|^{m\alpha + \beta}(1+|z|^\alpha)^\gamma} ,
	\end{equation}
	with some constant  $K > 0$ and $M$ defined by \eqref{eq:ResSector}.
\end{proposition}
\begin{proof}
	The function $z^\alpha R(z^\alpha) =  \left (I-\tfrac{A}{z^\alpha} \right )^{-1}$ remains analytic and bounded for any  $z^\alpha \notin \mathrm{Sp}(A) \cup \{0\}$, so its Neumann series converges unconditionally.
	Therefore,
	\begin{multline*}
		\left\| z^{\alpha-\beta}(z^\alpha I-A)^{-1}  - \frac{1}{z^\beta}\suml_{k=0}^{m}\frac{A^k }{z^{\alpha k}}\right \|
		=  \left\| \frac{1}{z^\beta} \left (I-\frac{A}{z^\alpha} \right )^{-1}  - \frac{1}{z^\beta}\suml_{k=0}^{m}\frac{A^k }{z^{\alpha k}}\right \|\\
		= \left\| \frac{1}{z^\beta}\left(\suml_{k=0}^{\infty}\frac{A^k }{z^{\alpha k}} - \suml_{k=0}^{m}\frac{A^k }{z^{\alpha k}} \right) \right\|
		= \left\| \frac{1}{z^\beta} \left (I-\frac{A}{z^\alpha} \right )^{-1} \frac{A^{m+1} }{z^{\alpha (m+1)}} \right\| \\
		= \frac{1}{|z|^{\alpha m+\beta}}\left\| (z^\alpha I-A)^{-1}A^{m+1}\right\|
		= \frac{1}{|z|^{\alpha m+\beta}}\left\| A^{1-\gamma} (z^\alpha I-A)^{-1}A^{m+\gamma}\right\|.
	\end{multline*}
	the last transformation is justified by the fact that $R(z,A)$ and $A^{1-\gamma}$ commutes.
	Target estimate \eqref{eq:FCP_res_cor_norm_est} follows directly from the above formula, after we apply inequality (2.30) from \cite{bGavrilyuk2011} with $z= z^\alpha$.
\end{proof}
It is worth noting that, if the argument $x$ posses certain spatial regularity  $x \in D(A^{\gamma})$, $\gamma >0$,
estimate \eqref{eq:FCP_res_cor_norm_est} guarantees a faster decay of the corrected term's norm $\|z^{\alpha-\beta}(z^\alpha I+A)^{-1} x - z^{-\beta}x\| < C |z|^{-\beta - \alpha\gamma}$, as $z \rightarrow \infty$, when compared to the norm of the original term in \eqref{eq:FCP_SO_cont_repr} bounded by $\left\|z^{\alpha-\beta}(z^\alpha I+A)^{-1} \right\| < C |z|^{-\beta}$, $C>0$.

The next result defines an improper integral representation for the components of $u_{\mathrm{h}}(t)$ from \eqref{eq:FCP_hom_inhom} and shows the way how the aforementioned correction is incorporated into the formula for $S_{\alpha}(t)$.
\begin{lemma}\label{lem:FCP_Pro_parametrized}
	{Assume that the given $A$ and $\alpha$ satisfy the conditions of \cref{thm:FCP_sol_rep}.
		For any $u_0 \in D(A^\gamma)$, $\gamma > 0$ and $u_1 \in X$ the operator functions $S_\alpha(t)u_0$, $S_{\alpha,2}(t)u_1$ admit the following representation:}
	\begin{eqnarray}
		\begin{aligned}
			S_{\alpha}(t) u_0
			 & = \frac{1}{2\pi i} \intl_{-\infty}^{\infty}e^{z(\xi)t}F_{\alpha,1}(\xi)u_0\,d\xi + u_0 ,                \\
			F_{\alpha,1} (\xi)
			 & = z'(\xi)\left( z^{\alpha-1}(\xi)\left (z^{\alpha}(\xi) I + A\right )^{-1} - \frac{1}{z(\xi)}I \right),
		\end{aligned}\label{eq:FCP_SO_exp_cor_repr_par}\\
		\begin{aligned}
			S_{\alpha,2}(t) u_1
			 & = \frac{1}{2\pi i} \intl_{-\infty}^{\infty}e^{z(\xi)t}F_{\alpha,2}(\xi)u_1\,d\xi,		\hspace*{5.3em} \\
			F_{\alpha,2} (\xi)
			 & = z'(\xi)z^{\alpha-2}(\xi)\left (z^{\alpha}(\xi) I + A\right )^{-1},
		\end{aligned} \label{eq:FCP_ISO_exp_cor_repr_par}
	\end{eqnarray}
	where $z(\xi) = a_0 - a_I\cosh(\xi) + ib_I\sinh(\xi), \quad \xi \in	(-\infty, \infty )$ and $a_0$, $a_I$, $b_I$ are specified by \eqref{eq:FCP_hyp_cont_par_final}.
	Moreover, for arbitrary finite $t \geq 0$ integrals in \eqref{eq:FCP_SO_exp_cor_repr_par} and \eqref{eq:FCP_ISO_exp_cor_repr_par} are uniformly convergent.
\end{lemma}
\begin{proof}
	Assume $\Gamma$ is a contour fulfilling the conditions of \cref{thm:FCP_sol_rep}.
	Due to the estimate $\|S_{\alpha,\beta}(t)x\| \leq  C e^{rt} \int_\Gamma  \tfrac{ |z|^{\alpha-\beta}}{1+|z|^\alpha} dz$ from \cite{SytnykWohlmuth2023}, the integral representation of operator function $S_{\alpha,2}(t)$  is uniformly convergent for any bounded non-negative $t$.
	Formula \eqref{eq:FCP_ISO_exp_cor_repr_par} is obtained as a result of the parametrization of \eqref{eq:FCP_SO_cont_repr} on the contour $\Gamma_I$ defined by \eqref{eq:int_cont_hyp_crit_sec}.
	In order to prove~\eqref{eq:FCP_SO_exp_cor_repr_par}, we apply the identity
	$\frac{1}{2\pi i} \intl_\Gamma  {e^{zt}}/{z}\,dz = \Res\limits_{z=0} {e^{zt}}/{z}  = 1$
	to rewrite \eqref{eq:FCP_SO_cont_repr} with $\beta=1$ in the following manner:
	\begin{equation}\label{eq:FCP_SO_cor_repr}
		\begin{aligned}
			S_{\alpha}(t) u_0
			 & = S_{\alpha,1}(t) u_0 - \frac{1}{2\pi i} \intl_\Gamma  \frac{e^{zt}}{z} u_0\,dz   + u_0                                  \\
			 & = \frac{1}{2\pi i} \intl_\Gamma  e^{zt} \left( z^{\alpha-1} (z^\alpha I + A)^{-1} - \frac{1}{z}I \right) u_0\, dz + u_0.
		\end{aligned}
	\end{equation}
	\Cref{prop:FCP_res_cor} and the inequality $|e^{zt}| < \max{\{e^{\Re(z)t},1\}}$, $t \in [0, T]$, $T>0$ guarantee that the last integral converges uniformly,
	so we are permitted to parameterize it on the contour $\Gamma = \Gamma_I$ defined by \eqref{eq:int_cont_hyp_crit_sec}.
	This yields representation \eqref{eq:FCP_SO_exp_cor_repr_par}.
\end{proof}
\Cref{lem:FCP_Pro_parametrized} is essential for all remaining analysis.
Unlike \eqref{eq:FCP_SO_cont_repr} or \eqref{eq:FCP_SO_exp_repr_par}, the new representation of $S_{\alpha,1}(t)$ by formula \eqref{eq:FCP_SO_exp_cor_repr_par} remains convergent at $t=0$.
For that matter, it can be used as vehicle for the uniformly convergent numerical method.
We shall use the term "corrected propagator representation" as a reference to  \eqref{eq:FCP_SO_exp_cor_repr_par}.

\subsection{Propagator Approximation}\label{sec:FCP_Prop_Approx}
As we can see from \cref{lem:FCP_Pro_parametrized}, the task of approximating the homogeneous part $u_{\mathrm{h}}(t)$ of the mild solution to \eqref{eq:FCP_DE}, \eqref{eq:FCP_BC}, defined by \eqref{eq:FCP_hom_inhom},  is reduced to the task of numerically evaluating improper integrals \eqref{eq:FCP_SO_exp_cor_repr_par} and \eqref{eq:FCP_ISO_exp_cor_repr_par}.
In this part, we describe how this is achieved using the trapezoidal quadrature rule.
Then, we proceed to study the accuracy of the obtained approximation using the theory of sinc-quadrature \cite{Stenger1993} along with its generalizations to propagator approximations \cite{gm5}.
In what follows, symbols $\wt{S}_{\alpha,\beta}^N(t)$ are used to denote the operators that approximate $S_{\alpha,\beta}(t)$.

For some $h>0$ and  $N \in \N$, let
\begin{equation}\label{eq:FCP_SO_cor_sinc_quad}
	\begin{aligned}
		\wt{S}_{\alpha,1}^N(t) x_1
		 & = \frac{h}{2 \pi i}\sum_{k=-N}^{N} \cF_{\alpha,1} (t, kh) + x_1, \quad \wt{S}_{\alpha,2}^N(t) x_2
		= \frac{h}{2 \pi i}\sum_{k=-N}^{N} \cF_{\alpha,2} (t, kh),
		\\
		\cF_{\alpha, \beta} (t, \xi)
		 & = e^{z(\xi)t}F_{\alpha,\beta} (\xi) x_\beta,  \quad \beta = 1, 2,
		\quad  x_1 \in D(A^\gamma), \quad x_2 \in X.
	\end{aligned}
\end{equation}
The functions $F_{\alpha,2}(\xi)$, $F_{\alpha,2}(\xi)$, $z(\xi)$  and the parameter $\gamma$ in the above formulas for propagator approximations have the meaning prescribed by \cref{lem:FCP_Pro_parametrized}.
Similarly to $S_\alpha(t)$, we use $\wt{S}_{\alpha}^N(t)$ to denote $\wt{S}_{\alpha,1}^N(t)$, where appropriate in the sequel.
Recall that $\Gamma_I$ is symmetric with respect to the real axis; hence, one can further reduce the number of summands in \eqref{eq:FCP_SO_cor_sinc_quad} using the following argument \cite{bGavrilyuk2011}.
\begin{remark}\label{rem:FCP_prop_half_contour}
	Let $\overline{z}$ denote the complex conjugate of $z$. If the operator $A$ is defined in such a way that $R(z,A) + R(\overline{z},A) = 2 R(\Re{z},A)$, for any $z \in \oC \setminus \mathrm{Sp}(A)$ and $x_\beta$ is defined over the field of real numbers, then
	\[
		\frac{h}{2 \pi i}\sum_{k=-N}^{N} \cF_{\alpha,\beta} (t, kh)  = \frac{h}{\pi} \left (\frac{1}{2}\cF_{\alpha, \beta} (t, 0) + \Re{\left \{ \sum_{k=1}^{N} \cF_{\alpha,\beta} (t, kh) \right \}} \right ),
	\]
	and the number of resolvent evaluations for $\wt{S}_{\alpha,\beta}^N(t)$ in formula \eqref{eq:FCP_SO_cor_sinc_quad} can be reduced from $2N+1$ \mbox{to $N+1$.}
\end{remark}

The error of \eqref{eq:FCP_SO_cor_sinc_quad} admits the following decomposition:
\[
	\|S_{\alpha,\beta}(t)x_\beta - \wt{S}_{\alpha,\beta}^N(t)x_\beta \| = \|S_{\alpha,\beta}(t)x_\beta - \wt{S}_{\alpha,\beta}^\infty(t)x_\beta + \wt{S}_{\alpha,\beta}^\infty(t)x_\beta - \wt{S}_{\alpha,\beta}^N(t)x_\beta\|
\]
\[
	\le \frac{1}{2\pi} \left \|\intl_{-\infty}^{\infty}\cF_{\alpha,\beta}\left(t, \xi\right)  d\xi -h\sum_{k=-\infty}^{\infty}\cF_{\alpha,\beta}\left(t, kh\right) \right \| +
	\frac{h}{2 \pi}\left \|\sum_{|k|>N}\cF_{\alpha,\beta}\left (t, kh\right ) \right \|,
\]
where $\|\cdot\|$ is the norm of $X$, as before.
This two-term representation of the error is common in the analysis of the accuracy of sinc-quadrature (see Section 3.2 in \cite{Stenger1993}).
The contribution from the first term is responsible for the replacement of integrals of $\cF_{\alpha,\beta}\left(t,\xi \right)$ from \eqref{eq:FCP_SO_exp_cor_repr_par} and \eqref{eq:FCP_ISO_exp_cor_repr_par} by the infinite series $\wt{S}_{\alpha,\beta}^\infty(t)$ of discrete function values $\cF_{\alpha,\beta}\left(t, kh\right)$.
As such, it is commonly called the discretization error.
To determine the value of $h$, one needs to balance it with the contribution from a so-called truncation-error term, that comes second in the formula above.

Let
$ {\bf H}^1(D_d)$
be a family of all functions $\mathcal{F}: {\mathbb C} \rightarrow X$, which are analytic in
the strip $D_d$, equipped with the norm
\[
	\| {\mathcal{F}} \|_{{\bf H}^1 (D_d)}
	= 	\lim\limits_{\epsilon \rightarrow 0}\intl_{\partial D_d\left (\epsilon \right )}\|{\mathcal F}(z)\||dz| ,
\]
where
$
	D_d(\epsilon)=\{z \in \mathbb{C}:\; | \Re(z)| < 1/\epsilon, \
	|\Im(z)|<d(1-\epsilon)\}
$
and $\partial D_d(\epsilon)$ is the boundary of $ D_d(\epsilon)$.
The discretization errors of \eqref{eq:FCP_SO_cor_sinc_quad} satisfy the estimate \cite{Stenger1993,gm5}:
\begin{equation}\label{eq:FCP_prop_discr_error}
	\left \|S_{\alpha,\beta}(t) x_\beta - \wt{S}_{\alpha,\beta}^\infty(t) x_\beta \right \|	\leq  \frac{e^{-\pi d/h}}{2 \sinh (\pi d/h)}\|{\cF_{\alpha,\beta}}(t,\cdot)\|_{{\bf H}^1(D_d)} .
\end{equation}
Thus, in order to bound this term, one needs to obtain estimates for the ${{\bf H}^1(D_d)}$ norms of the functions $\cF_{\alpha,\beta}(t,z)$, $\beta =1,2$.
These are provided by the next lemma.

\begin{lemma} \label{lem:FCP_SO_ISO_Hp_bound}
	Let  $A$ be a sectorial operator satisfying the conditions of \cref{thm:FCP_sol_rep}.
	For any $t \geq 0$, $\alpha \in (0, 2)$, $x_1 \in D(A^\gamma)$, $x_2 \in X$, $\gamma>0$ and arbitrary small $\delta>0$
	\begin{equation}\label{eq:FCP_SO_exp_cor_repr_norm_est}
		\begin{aligned}
			\|{\cF_{\alpha,1}}(t,\cdot)\|_{{\bf H}^1(D_{d-\delta})}
			 & \leq  C_{\alpha,1}^\pm (\gamma,\delta)e^{a_0 t}\|A^{\gamma}x_1\|, \\
			\|{\cF_{\alpha,2}}(t,\cdot)\|_{{\bf H}^1(D_{d-\delta})}
			 & \leq  C_{\alpha,2}^\pm (\gamma,\delta)e^{a_0 t} \|x_2\|,
		\end{aligned}
	\end{equation}
	with constants $C_{\alpha,\beta}^\pm (\gamma,\delta) = C_{\alpha,\beta}(\gamma, \delta -d) + C_{\alpha,\beta}(\gamma, d - \delta )$ and
	\begin{equation}\label{eq:FCP_SO_exp_cor_repr_norm_const}
		\begin{aligned}
			C_{\alpha,1}(\gamma, \nu) & =
			\frac{K_1 b(\nu)}{ \alpha \gamma (a(\nu) - a_0) r_0^\gamma(\nu)}, \\
			C_{\alpha,2}(\gamma, \nu) & =
			K_2\frac{b(\nu) \left(b^2(\nu)+(a(\nu) - a_0)^2 \right)^{\alpha/2}}{(a(\nu) - a_0)^2  r_0(\nu)},
		\end{aligned}
	\end{equation}
	that are independent of $t$. Here, $K_1, K_2 >0$ and $r_0(\nu) = \inf\limits_{\xi \in \R} r(\xi, \nu)$,
	\begin{equation}\label{eq:FCP_r}
		r(\xi, \nu) =
		\frac{1}{\cosh^\alpha{\xi}} + \left( b^2(\nu)\tanh^2{\xi} + \left (a(\nu) - \frac{a_0}{\cosh{\xi}} \right )^2\right)^{\frac{\alpha}{2}}.
	\end{equation}
\end{lemma}
\begin{proof}
	To estimate the norms $	\|{\cF_{\alpha,\beta}}(t,\cdot)\|_{{\bf H}^1(D_{d_1})}$, $\beta = 1,2$ we start from  \eqref{eq:FCP_SO_cor_sinc_quad}, split out the scalar part in each norm and use bounds \eqref{eq:FCP_res_cor_norm_est},
	\eqref{eq:ResSector} for the operator-dependent parts, correspondingly.
	As a result, we obtain
	\begin{align*}
		\left\| \cF_{\alpha,1} (t, \xi) x \right\|
		 & \leq \left| e^{z(\xi)t}z'(\xi) \right|\frac{(1+M) K}{|z(\xi)|(1+|z(\xi)|^\alpha)^\gamma} \left\|A^\gamma x_1\right\|, \\
		\left\|\cF_{\alpha,2} (t, \xi) x \right\|
		 & \leq\left| e^{z(\xi)t}z'(\xi) \right|\frac{|z(\xi)|^{\alpha-2}M\left\|x_2\right\|}{(1+|z(\xi)|^\alpha)} .
	\end{align*}

	When the variable $\xi$ is extended from the real line into the strip, the integration hyperbola $z(\xi)$, adopted here from \cref{lem:FCP_Pro_parametrized}, transforms into the parametric family of hyperbolas $\Gamma(\nu)=\{z(\xi + i \nu):\ \xi \in (-\infty, \infty)\}$, $\nu \in [-d,d]$.
	Let $w = \xi + i \nu \in D_{d-\delta}$ for some \mbox{$\delta >0$, then}
	\begin{align*}
		\left |\frac{z'(w)}{z(w)} \right |
		=\left | \frac{a(\nu) \sinh{\xi}-i b(\nu)\cosh{\xi}} {a(\nu)
			\cosh{\xi}- a_0 - i b(\nu) \sinh{\xi} } \right| %
		=\frac{\sqrt{a^2(\nu) \tanh^2{\xi}+b^2(\nu)}}{\sqrt{b^2(\nu)\tanh^2{\xi} + \left (a(\nu) - \frac{a_0}{\cosh{\xi}} \right )^2}} ,
	\end{align*}
	with $a(\nu)$, $b(\nu)$ are defined by \eqref{eq:cont_hyp_par_family}.
	Consider the function $\eta_1(s, b_0)  = \frac{a^2(\nu)s^2+b^2(\nu)}{b^2(\nu)s^2 + (a(\nu) - b_0)^2}$, with some $b_0$ independent of $s$.
	The derivative of this function with respect to $s$
	\begin{align*}
		\eta_1'(s,b_0) & =  \frac{2s(a^2(\nu) - a(\nu) b_0 - b^2(\nu))(a^2(\nu) - a(\nu)b_0 + b^2(\nu))}{(b^2(\nu)s^2 + (a(\nu)-b_0)^2)^2} \\
		               & =\frac{2s\left( (a^2(\nu) - a(\nu) b_0)^2 - b^4(\nu)\right)}{(b^2(\nu)s^2 + (a(\nu)-b_0)^2)^2}
	\end{align*}
	determines the behavior of $\eta_1(s,b_0)$ for the values of $b_0$ that belong to the interval $(0, a_0)$, induced by the identity $\eta_1\left (\tanh\xi,  \frac{a_0}{\cosh{\xi}}\right ) = \left |\frac{z'(w)}{z(w)} \right |^2$.
	The sign of $\eta_1'(s,b_0)$ is equal to the sign of $a^2(\nu) - a(\nu) b_0)^2 - b^4(\nu) \leq a^4(\nu)  - b^4(\nu) = -a_0^4\frac{\cos^2(2(\phi_s+d-v))}{\cos^4(\phi_s)}$.
	Therefore, for any $b_0$, the maximum of  $\eta_1(s,b_0)$ is attained at $s=0$, whence
	\[
		\left |\frac{z'(w)}{z(w)} \right | \leq \frac{b(\nu)}{\sqrt{\left( a(\nu)-\frac{a_0}{\cosh\xi}\right)^2}} \leq \frac{b(\nu)}{|a(\nu)-a_0|}.
	\]

	The norm $\|{\cF_{\alpha,1}}(t,w) x \|$ can be further estimated as
	\begin{equation}\label{eq:FCP_SO_exp_cor_repr_norm_est_der}
		\begin{aligned}
			\left\| \cF_{\alpha,1} (t, w) \right\|
			 & \leq e^{\Re{z(w)}t}\left |\frac{z'(w)}{z(w)} \right | \frac{(1+M) K}{(1+|z(w)|^\alpha)^\gamma} \left\|A^\gamma x_1\right\| \\
			 & =\frac{b(\nu)}{a(\nu)-a_0}\frac{e^{(a_0 - a(\nu)\cosh{\xi}) t}}{(1+|z(w)|^\alpha)^\gamma}
			(1+M)K\left\|A^\gamma x_1\right\|                                                                                             \\
			 & \leq \frac{b(\nu)}{a(\nu)-a_0}\frac{(1+M)K e^{(a_0 - a(\nu)\cosh{\xi}) t
					}}
			{\left( r(\xi,\nu)\cosh^{\alpha}{\xi}\right)^\gamma }\|A^{\gamma}x_1\|                                                        \\
			 & \leq
			\frac{(1+M)K b(\nu)2^{\alpha\gamma}}
			{ (a(\nu)-a_0) r^{\gamma}(\xi,\nu)}
			e^{(a_0 - a(\nu)\cosh{\xi}) t - \alpha\gamma|\xi|}\|A^{\gamma}x_1\|.
		\end{aligned}
	\end{equation}
	Here, $r(\xi,\nu)$ is a strictly positive bounded function that is defined by the equality $1+|z(w)|^\alpha = r(\xi,\nu)\cosh^\alpha{\xi}$.
	Solving it for $r(\xi,\nu)$ gives us \eqref{eq:FCP_r}.

	Now, we turn our attention to $\|{\cF_{\alpha,2}}(t,w) x \|$.
	Let us consider a function $\eta_2(s,b_0)  = \frac{a^2(\nu)s^2+b^2(\nu)}{(b^2(\nu)s^2 + (a(\nu) - b_0)^2)^2}$, $s \in [0, 1]$.
	Similarly to $\eta_1(\xi)$, this function satisfies the identity \\
	$\eta_2\left (\tanh\xi,  \frac{a_0}{\cosh{\xi}}\right )\cosh^2{\xi} =  \left |\frac{z'(w)}{z(w)^2} \right|^2 $.
	By inspecting the derivative
	\[
		\eta_2'(s,b_0) = \frac{2s\left( (a^2(\nu) - a(\nu) b_0)^2 - 2b^4(\nu) - a^2(\nu)b^2(\nu)s^2\right)}{(b^2(\nu)s^2 + (a(\nu)-b_0)^2)^3}
	\]
	we learn that its sign is also determined by a sign of the numerator with only one real root $s=0$.
	Two other roots of $\eta_2'(s,b_0)$ are non-real because the quadratic function $(a^2(\nu) - a(\nu) b_0)^2 - 2b^4(\nu) - a^2(\nu)b^2(\nu)s^2 \leq (a^2(\nu) - a(\nu) b_0)^2 - 2b^4(\nu) \leq 2a^4(\nu)  - 2b^4(\nu)$ is negative for any $s$.
	Whence, we get $\left |\frac{z'(w)}{z^2(w)} \right | \leq \frac{b(\nu)}{\cosh(\xi)(a(\nu)-a_0)^2}$ and
	\[
		\begin{aligned}
			\left\| \cF_{\alpha,2} (t, w) \right\|
			 & \leq M|z'(w)|  \frac{|z(w)|^{\alpha-2}  e^{\Re{z(w)}t}}{(1+|z(w)|^\alpha)}
			\left\| x_2\right\|                                                                                                    \\
			 & \leq M\left |\frac{z'(w)}{z^2(w)} \right |  \frac{|z(w)|^{\alpha}  e^{(a_0 - a(\nu)\cosh{\xi}) t}}{1+|z(w)|^\alpha}
			\left\| x_2\right\|                                                                                                    \\
			 & \leq M\frac{b(\nu) \left(b^2(\nu)+(a(\nu)-a_0)^2 \right)^{\alpha/2}}{(a(\nu)-a_0)^2 (\cosh{\xi})^{1-\alpha}}
			\frac{e^{(a_0 - a(\nu)\cosh{\xi}) t}}{1+|z(w)|^\alpha}
			\left\| x_2\right\|                                                                                                    \\
			 & \leq
			2M\frac{b(\nu) \left(b^2(\nu)+(a(\nu) - a_0)^2 \right)^{\alpha/2}}{r(\xi,\nu)(a(\nu) - a_0)^2}
			e^{(a_0 - a(\nu)\cosh{\xi}) t -|\xi|}
			\left\| x_2\right\| .
		\end{aligned}
	\]

	The obtained estimates for $\left\| \cF_{\alpha,\beta} (t, w) \right\|$, $\beta=1,2$ demonstrate that these norms are exponentially decaying for any $t \geq 0$ as $\xi \rightarrow \infty$.
	Consequently, the integral terms from $\left\| \cF_{\alpha,\beta} (t, w) \right\|_{{\bf H}^1 (D_d)}$ over the vertical parts of $\partial D_{d-\delta}(\epsilon)$ vanish in the limit $\epsilon \rightarrow 0$ and
	we end up with the following expression:
	\[
		\left \|\mathop{\vphantom{\sum}}{\cF_{\alpha,\beta}}(t,\cdot) \right \|_{{\bf H}^1(D_{d-\delta})}
		=  \intl_{-\infty}^{\infty} \left\| \cF_{\alpha,\beta}(t,\xi + i (\delta-d))  \right\| + \left\| \cF_{\alpha,\beta}(t,\xi - i (d-\delta))  \right\| d \xi. \\
	\]
	After estimating the last integral using the bounds obtained above, we remove the dependence of the integrands on $r(\xi,\nu)$ by bounding its value from below with a positive function $r_0(\nu)<r(\xi,\nu)$ and, subsequently, evaluate the obtained integrals explicitly.
	This yields the pair of objective estimates from \eqref{eq:FCP_SO_exp_cor_repr_norm_est},
	with constants $K_1 = (1+M)K2^{\alpha\gamma+1}$ and  $K_2 = 4M$.
	The lemma is proved.
\end{proof}
Observe that the appearance of $\alpha \gamma$ in $C_{\alpha,1}(\gamma, \nu)$ from \eqref{eq:FCP_SO_exp_cor_repr_norm_const} forces the discretization error of $S_{\alpha}x_1$, $x_1 \in D(A^\gamma)$ to become unbounded in the limit $\gamma \rightarrow 0$.
\begin{lemma}\label{lem:FCP_SO_ISO_trunc_bound}
	Assume that $A$ and $\alpha$ satisfy the conditions of \cref{thm:FCP_sol_rep}.
	Then, for any $t \geq 0$, $x_1 \in D(A^\gamma)$, $x_2 \in X$, $\gamma>0$  the truncation error of $(2N+1)$-term approximations \eqref{eq:FCP_SO_cor_sinc_quad} with the step-size $h>0$ satisfies the estimate
	\begin{equation*}\label{eq:FCP_SO_exp_cor_repr_cont_est}
		\begin{aligned}
			\frac{h}{2\pi}\left \|\sum_{|k|>N}\!\!\cF_{\alpha,1}\left (t, kh\right ) \right \|
			 & \leq  \frac{h C_{\alpha,1} (\gamma,0)}{\pi (1 - e^{-\alpha \gamma h})}\frac{e^{a_0 t}} {e^{\alpha\gamma (N+1)h}}\|A^{\gamma}x_1\| %
			\leq  \frac{C_{\alpha,1} (\gamma,0)}{\pi}\frac{e^{a_0 t}} {e^{\alpha\gamma Nh}}\|A^{\gamma}x_1\|,                                    \\
			\frac{h}{2\pi}\left \|\sum_{|k|>N}\!\!\cF_{\alpha,2}\left (t, kh\right ) \right \|
			 & \leq  \frac{hC_{\alpha,2} (\gamma,0)}{\pi  (1 - e^{-h})}\frac{e^{a_0 t}} {e^{(N+1)h}}\|x_2\|
			\leq  \frac{C_{\alpha,2} (\gamma,0)}{\pi }\frac{e^{a_0 t}} {e^{Nh}}\|x_2\|,                                                          \\
		\end{aligned}
	\end{equation*}
	where $C_{\alpha,\beta}(\gamma, \nu)$ are defined by \eqref{eq:FCP_SO_exp_cor_repr_norm_const}.
\end{lemma}
\noindent The proof of this Lemma relies on the established estimates for $\left\| \cF_{\alpha,\beta} (t, w) \right\|$ and is analogous to the proof of the respective part in Theorem 3.1.7 from \cite{Stenger1993}.
For brevity, we omit it here.

We finally have all the necessary tools in place to proceed to an a priori accuracy estimate of the constructed numerical evaluation formulas \eqref{eq:FCP_SO_cor_sinc_quad} for propagators of \eqref{eq:FCP_DE} and \eqref{eq:FCP_BC}.
\begin{theorem}\label{thm:FCP_prop_appr}
	Let $A$ be a sectorial operator with the domain $D(A)$ and the spectrum $\mathrm{Sp}(A) \subset \Sigma(\rho_s, \varphi_s)$,  $\rho_s>0$, $\varphi_s < \pi/2$.
	Then, for any $\alpha \in (0, 2)$, $x_1 \in D(A^\gamma)$, $x_2 \in X$, $\gamma \in (0,1)$, such that $\varphi_s < \pi\left(1 - \tfrac{\alpha}{2}\right)$, and $t \in [0, T]$,  the sinc-quadrature-based approximations
	$\wt{S}_{\alpha,1}^N(t)x_1$, $\wt{S}_{\alpha,2}^N(t)x_2$
	converge to the values of the corresponding operator functions
	$S_\alpha(t)x_1$, $S_{\alpha,2}(t)x_2$
	at the rate $\mathcal{O}(e^{-c\sqrt{N}})$, $c>0$, as $N \rightarrow \infty$.
	Moreover, the following error bounds are valid
	\begin{equation}\label{eq:FCP_SO_cor_err_est}
		\left \|S_\alpha(t)x_1 - \wt{S}_{\alpha,1}^N(t)x_1 \right \| \leq
		{C_1}	\exp{\left(-c\sqrt{\alpha\gamma N}\right)}{e^{a_0 t}}
		\|A^{\gamma}x_1\|,
	\end{equation}
	\begin{equation}\label{eq:FCP_ISO_cor_err_est}
		\left \|S_{\alpha,2}(t)x_2 - \wt{S}_{\alpha,2}^N(t)x_2 \right \| \leq
		{C_2}\exp{\left(-c\sqrt{N}\right)}e^{a_0 t}
		\|x_2\|,
	\end{equation}
	with $c=\sqrt{2\pi d}$, provided that the step-size in \eqref{eq:FCP_SO_cor_sinc_quad} is given by $h_1$ and $h_2$, accordingly,
	\begin{equation}\label{eq:FCP_h}
		h_1=\sqrt{\frac{2 \pi d}{\alpha \gamma N}}, \quad h_2=\sqrt{\frac{2 \pi d}{N}}.
	\end{equation}
	Here,
	$d = \frac{\phi_\alpha - \phi_c}{2}$, $\phi_\alpha = \min\{\pi, \frac{\pi -\varphi_s}{\alpha}\}$   and $\phi_c \in \left[\tfrac{\pi}{2}, \phi_\alpha\right)$, $a_0>0$  are given.
	The constants $C_\beta$ \mbox{from  \eqref{eq:FCP_SO_cor_err_est} and \eqref{eq:FCP_ISO_cor_err_est}} are independent of $t,N$.
\end{theorem}
\begin{proof}
	To obtain error bounds for the approximants $\wt{S}_{\alpha,\beta}^N$, we depart from the previously established decomposition
	\[\|S_{\alpha,\beta}(t)x_\beta - \wt{S}_{\alpha,\beta}^N(t)x_\beta \|
		\leq \|S_{\alpha,\beta}(t)x_\beta - \wt{S}_{\alpha,\beta}^\infty(t)x_\beta\| + \|\wt{S}_{\alpha,\beta}^\infty(t)x_\beta - \wt{S}_{\alpha,\beta}^N(t)x_\beta\| ,
	\]
	and then use the results of \cref{lem:FCP_SO_ISO_Hp_bound,lem:FCP_SO_ISO_trunc_bound} to estimate the right-hand sides.
	This transcribes into
	\[
		\begin{aligned}
			\|S_{\alpha}(t)x_1 - \wt{S}_{\alpha}^N(t)x_1 \|
			 & \leq \frac{e^{-\frac{\pi d}{h}}\|{\cF_{\alpha,\beta}}(t,\cdot)\|_{{\bf H}^1(D_{d-\delta})}}{2 \sinh{\frac{\pi d}{h}}} + \frac{C_{\alpha,1} (\gamma,0)e^{a_0 t}\|A^{\gamma}x_1\|}{\pi \alpha \gamma e^{\alpha\gamma Nh}} \\
			 & \leq  e^{a_0 t} \left(C_{\alpha,1}^\pm (\gamma,\delta)\frac{e^{-\frac{\pi d}{h}}}{2 \sinh{\frac{\pi d}{h}}} +  \frac{C_{\alpha,1} (\gamma,0)}{\pi \alpha \gamma e^{\alpha\gamma Nh}} \right)
			\|A^{\gamma}x_1\|                                                                                                                                                                                                          \\
			 & \leq  e^{a_0 t} \left(\frac{c_0 C_{\alpha,1}^\pm (\gamma,\delta)}{e^{2\frac{\pi d}{h}}} + \frac{C_{\alpha,1} (\gamma,0)}{\pi \alpha \gamma e^{\alpha\gamma Nh}}  \right)
			\|A^{\gamma}x_1\| ,
		\end{aligned}
	\]
	and
	\[
		\begin{aligned}
			 & \|S_{\alpha,2}(t)x_2 - \wt{S}_{\alpha,2}^N(t)x_2 \|
			 & \leq  e^{a_0 t} \left(c_0 C_{\alpha,2}^\pm (\gamma,\delta)e^{-2\frac{\pi d}{h}} + \frac{C_{\alpha,2} (\gamma,0)}{\pi e^{Nh}} \right)
			\|x_2\|,
		\end{aligned}
	\]
	with $c_0 = (1-\exp{\frac{2\pi d}{h}})^{-1}$.
	Having the aim of balancing the order of error contributions from each term inside the brackets, we make two involving exponential functions asymptotically equal as $N \rightarrow \infty$.
	This yields two independent equations
	\[
		\frac{2 \pi d}{h}=\alpha\gamma Nh, \quad \frac{2 \pi d}{h}= Nh,
	\]
	with the solutions described by \eqref{eq:FCP_h}.
	After that, we substitute these expressions into the previously established error estimates to get the following bounds:
	\begin{align*}
		\|S_{\alpha,1}(t)x_1 - \wt{S}_{\alpha,1}^N(t)x_1 \|
		 & \leq \left(c_0 C_{\alpha,1}^\pm (\gamma,\delta) + C_{\alpha,1} (\gamma,0)\right)  e^{a_0 t} e^{-\sqrt{2\pi d \alpha\gamma N}}
		\|A^{\gamma}x_1\|,                                                                                                               \\
		\|S_{\alpha,2}(t)x_2 - \wt{S}_{\alpha,2}^N(t)x_2 \|
		 & \leq \left(c_0 C_{\alpha,2}^\pm (\gamma,\delta) + C_{\alpha,2} (\gamma,0)\right)  e^{a_0 t} e^{-\sqrt{2\pi d N}}
		\|x_2\|,
	\end{align*}
	which are reduced to  \eqref{eq:FCP_SO_cor_err_est} and \eqref{eq:FCP_ISO_cor_err_est} after denoting
	$C_1 = c_0 C_{\alpha,1}^\pm (\gamma,\delta) + C_{\alpha,1} (\gamma,0)$, $C_2 = c_0 C_{\alpha,2}^\pm (\gamma,\delta) + C_{\alpha,2} (\gamma,0)$.
\end{proof}
It follows from \eqref{eq:FCP_SO_cor_err_est} and \eqref{eq:FCP_ISO_cor_err_est}, that the value of the contour parameter $a_0$ can be used to control the error contribution of the factor $e^{a_0 t}$.
Throughout the rest of this work, we set $a_0 = \pi/6$ to make this factor reasonably bounded: $e^{a_0 t} \leq e^{5\pi/6}\leq 14$.

\subsection{Numerical Scheme for Homogeneous Part of Solution}\label{sec:FCP_hom_sol_appr}
We approximate the homogeneous part $u_{\mathrm{h}}(t)$ of the solution to \eqref{eq:FCP_DE},  \eqref{eq:FCP_BC} defined by \eqref{eq:FCP_hom_inhom} using the numerical methods for propagators approximation constructed in \cref{sec:FCP_Prop_Approx}.
Then, for every fixed $N>0$, the approximation $\wt{u}_\mathrm{h}^N(t)$ to $u_{\mathrm{h}}(t)$ is defined as
\begin{equation}\label{eq:FCP_hom_sol_appr}
	\wt{u}_{\mathrm{h}}^N(t) = \wt{S}_{\alpha,1}^{N_1}(t) u_0 + \wt{S}_{\alpha,2}^{N_2}(t) u_1 .
\end{equation}
The error of $\wt{u}_{\mathrm{h}}^N(t)$ is characterized by the following corollary, which is an immediate consequence of \cref{thm:FCP_prop_appr}.
\begin{corollary}\label{thm:FCP_hom_sol_err_est}
	Assume that the operator $A$, initial values $u_0$, $u_1$ and the fractional order $\alpha$ satisfy the conditions of \cref{thm:FCP_prop_appr} with $x_1 = u_0$, $x_2 = u_1$.
	For any given $N \in \N$, the approximate solution $\wt{u}_{\mathrm{h}}^N(t)$, defined by \eqref{eq:FCP_hom_sol_appr} with $N_1 = N$, $N_2 = \lceil \alpha \gamma N \rceil$,
	converges to the homogeneous solution  $u_{\mathrm{h}}(t)$  \mbox{of \eqref{eq:FCP_DE}, \eqref{eq:FCP_BC}} and the following error bound is valid:
	\begin{equation}\label{eq:FCP_hom_sol_err_est}
		\left\| u_{\mathrm{h}}(t) - \wt{u}_\mathrm{h}^N(t) \right\|  \leq {C_\gamma}	\exp{\left(-c\sqrt{\alpha\gamma N}\right)}{e^{a_0 t}}
		\|A^{\gamma}u_0\| .
	\end{equation}
	The constant $C_{\gamma}$ is dependent on $A$, $u_0$, and independent of $t,N$.
\end{corollary}

It is important to note that the smoothness assumptions for $u_0$, enforced by \cref{thm:FCP_hom_sol_err_est} and \cref{thm:FCP_prop_appr}, are compatible with the similar assumptions made in \cite{gmv,bGavrilyuk2011} for the Cauchy problem with the integer order derivative.
For a more concise discussion on the impact of the initial data smoothness on the properties of solution to
problem \eqref{eq:FCP_DE}, \eqref{eq:FCP_BC} we direct the reader to \cite{jin2019numerical}.

To compute the approximation $\wt{u}_{\mathrm{h}}^N(t)$, we suggest to use \cref{alg:FCP_hom_sol_appr} provided below.
In this algorithm, the evaluation of each propagator $\wt{S}_{\alpha,\beta}^{N_\beta}$, $\beta =1,2$ is decoupled into two cycles.
The first cycle is responsible for the evaluation of resolvents $(z(m h_1)^\alpha I+A)^{-1}$ at the quadrature points of $\Gamma_I$.
This amounts to the solution of  $2N_\beta +1$  linear equations that are all mutually independent and hence can be solved in parallel.
\begin{algorithm}[htb]
	\renewcommand{\algorithmicrequire}{\textbf{INPUT:}}
	\renewcommand{\algorithmicensure}{\textbf{OUTPUT:}}
	\begin{algorithmic}[1]
		\small
		\REQUIRE{\hspace{2em}  $\alpha, u_0, u_1$, $t_k$, $\varphi_s$, $N, \gamma$}
		\ENSURE{\hspace{0.8em}  $\left\{\wt{u}_\mathrm{h}^N(t_k)\right\}$}
		\STATE{$N_1 := N$;  $N_2 := \alpha \gamma N_1$}
		\STATE{Calculate $a_I, b_I$ and $h_1,h_2$ by \eqref{eq:FCP_hyp_cont_par_final} and \eqref{eq:FCP_h} }
		\FOR{$m=-N_1$ \TO $N_1$}
		\STATE{Solve $(z(m h_1)^\alpha I+A)v = u_0$ \label{alg:FCP_hom_sol_appr_Res1}}
		\STATE{$F_{1,m}:=  z(m h_1)^{\alpha-1}v - \frac{1}{z(m h_1)}u_0$}\label{alg:FCP_hom_sol_appr_cor1}
		\ENDFOR
		\FOR{each $t_k$}
		\STATE{$\wt{u}_\mathrm{h}^N(t_k) := u_0 + \frac{h_1}{2 \pi i} \suml_{m=-N_1}^{N_1} \!z'(mh_1){e^{z(mh_1)t_k }}F_{1,m}$ \label{alg:FCP_hom_prop_eval1}}
		\ENDFOR
		\IF{$\alpha > 1$}
		\FOR{$m=-N_2$ \TO $N_2$}
		\STATE{Solve $(z(m h_2)^\alpha I+A)v = u_1$ \label{alg:FCP_hom_sol_appr_Res2}}
		\STATE{$F_{2,m}: = z'(mh_2)z(m h_2)^{\alpha-2}v$}\label{alg:FCP_hom_sol_appr_cor2}
		\ENDFOR
		\FOR{each $t_k$}
		\STATE{$\wt{u}_\mathrm{h}^N(t_k) := \wt{u}_\mathrm{h}^N(t_k) + \frac{h_2}{2 \pi i} \suml_{m=-N_2}^{N_2} {e^{z(mh_2)t_k }}F_{2,m}$}
		\ENDFOR
		\ENDIF
		\RETURN{$\{\wt{u}_\mathrm{h}^N(t_k)\}$}
	\end{algorithmic}
	\caption{Algorithm for computing the homogeneous part approximation $\wt{u}_\mathrm{h}^N(t)$.%
		\label{alg:FCP_hom_sol_appr}%
	}
\end{algorithm}
\FloatBarrier
\noindent If $A$ is the discretization of a certain partial differential operator, every resolvent equation from line \ref{alg:FCP_hom_sol_appr_Res1} of the algorithm is actually a system of linear equations.
When this is the case, one can leverage additional level of parallelism here, as long as the size and the solution method of the resolvent equation warrant that and the computing environment permits for such possibility.
Furthermore, the total number of resolvent evaluations in \cref{alg:FCP_hom_sol_appr} can be reduced all the way down to  $N_1+N_2+2$ if the initial data and $A$ satisfy the conditions from \cref{rem:FCP_prop_half_contour}.

Given the solution of resolvent equations obtained in the first cycle of \cref{alg:FCP_hom_sol_appr},  the second cycle computes the resulting propagator approximation.
As we can see from line \ref{alg:FCP_hom_prop_eval1} of \cref{alg:FCP_hom_sol_appr}, for every fixed $t = t_k$ this step amounts to calculating the weighted sum of resolvents.
Hence, its computation is apparently independent of the computed values of solution at different times and can be performed simultaneously.
Such feature of the method alone results in the substantial computational advantage over existing sequential time-discretization methods \cite{Lubich2004,Baffet2017,Diethelm2019,Khristenko2021}, because the average computation cost per $\wt{u}_\mathrm{h}^N(t_k)$ for any $k \in \{1,\ldots, K\}$ is independent on the value of $t_k \in [0, T]$ and, unlike in the case of a sequential method, this cost goes down when $K$ grows.
Even in the worst-case scenario of $K=1$, $t_k<<1$, our method should still remain competitive with the mentioned sequential methods due to its parallelization capability and the uniform exponential convergence.
We postpone a more detailed comparison with existing methods until \cref{ex:FCP_ex3_inhom_R_FD}, where a fully discretized problem is considered.
It is important to point out that the described multi-level parallel evaluation strategy is well suited for the multi-node computing architectures, with each node containing the combination of a central processing unit and multiple hardware accelerators, that are ubiquitous nowadays.

For certain realizations of $A$ (c.f. \cite{McLean2010a}) and large values of $N$, the resolvent evaluation steps of \cref{alg:FCP_hom_sol_appr} might lead to the numerical instability when $|z|$ is large.
This problem can be alleviated by modifying lines \ref{alg:FCP_hom_sol_appr_Res1}-\ref{alg:FCP_hom_sol_appr_cor1} and \ref{alg:FCP_hom_sol_appr_Res2}-\ref{alg:FCP_hom_sol_appr_cor2} as described in \cite[Eq. (2.18)]{McLean2010}.
Another noticeable feature of the above algorithm is its use of the resolvent evaluations with the complex arguments.
This may require additional attention from the implementation point of view if the resolvent is evaluated numerically, for instance using the finite element method software that does not support complex arithmetic.
Alternatively, one could deal with the complex resolvent arguments by redefining  $A$ via the embedding of its domain into the real space of higher dimensionality.
This is always possible, since the resolvent equations from lines \ref{alg:FCP_hom_sol_appr_Res1} and \ref{alg:FCP_hom_sol_appr_Res2} of the algorithm are linear in $z$.
Such modifications are unnecessary for the numerical experiments conducted below; hence, we do not incorporate them into the algorithms for simplicity.

\FloatBarrier

\begin{example}\label{ex:FCP_ex1_hom_R_eigenfunction}
	Let us consider the standard example problem in which $A$ is a one-dimensional Laplacian accompanied by the Dirichlet boundary conditions on $[0, L]$:
	\begin{equation}\label{eq:FCP_ex1_hom_A}
		\begin{split}
			 & Au= - a\frac{d^2}{dx^2}u, \quad \forall u \in D(A),  \\
			 & D(A)=\{u(x) \in H^2(0,L): \quad  u(0) = u(L) =  0\},
		\end{split}
	\end{equation}
	where $a>0$ is some predefined constant.
	The initial values $u_0, u_1$ are chosen to be the eigenfunctions of the operator $A$ with indices $k_0$, $k_1$, correspondingly:
	\begin{equation}\label{eq:FCP_ex1_hom_IV}
		u_0 = \sin{\frac{\pi k_0 x}{L}}, \quad u_1 = \sin{\frac{\pi k_1 x}{L}}.
	\end{equation}
	The exact solution of fractional Cauchy problem \eqref{eq:FCP_DE}, \eqref{eq:FCP_BC} with such $A$, $u_0, u_1$ and $f(t) = 0$  can be represented as follows (see Section 1.3 in \cite{Bazhlekova2001}):
	\[
		u(t, x) = E_{\alpha,1}(-\lambda(k_0) t^\alpha)  \sin{\frac{\pi k_0 x}{L}} + H(\alpha-1)E_{\alpha,2}(-\lambda(k_1) t^\alpha) \sin{\frac{\pi k_1 x}{L}}.
	\]
	Here, $E_{\alpha,\beta}(z)  = \suml_{k=0}^{\infty} \frac{z^k}{\Gamma(\alpha n + \beta)}$ is the Mittag--Leffler function, $\lambda(k) = a\frac{\pi^2}{L^2} k^2$, $k  = k_0, k_1$ are the eigenvalues of $A$ and $H(\cdot)$ is the Heaviside function, which is added to make the above solution formula valid for all $\alpha \in (0, 2)$.

	It is easy to verify that for any $z \in {\mathbb C} \setminus \mathrm{Sp}(-A)$ and  $k \in \N$, the resolvent \\ $R(z, -A)\sin{\frac{\pi k x}{L}}$ admits the following representation
	\[
		R(z, -A)\sin{\frac{\pi k x}{L}}
		= (z I+A)^{-1}\sin{\frac{\pi k x}{L}}
		= \frac{1}{z +\lambda(k)} \sin{\frac{\pi k x}{L}}.
	\]
	Hence, all the resolvent evaluations in \cref{alg:FCP_hom_sol_appr} for such $u_0, u_1$ can be conducted explicitly.
	This allows us to focus on analyzing the error contribution from the numerical method for $u_{\mathrm{h}}(t)$, given by \eqref{eq:FCP_hom_sol_appr}, in the absence of the error associated with the discretization of spatial operator $A$.
	The results, presented below, were obtained using the implementation\footnote{The code is available at \url{github.com/DmytroSytnyk/FCP2023}} of \cref{alg:FCP_hom_sol_appr}  developed in Matlab.
	The standard double precision IEEE~754 arithmetic (and its extension to complex numbers) is used for computations everywhere in this and other examples.
	The evaluation of $E_{\alpha,\beta}(z)$ was performed via the contour method from~\cite{Garrappa2015}, using the accompanied Matlab implementation.
	The interested reader may also consider alternative methods from \cite{Seybold2009,McLean2021}.

	The behavior of the exact solution $u(t, x)$ for the simplest case $c =1$, $L =1$ is shown in \cref{fig:FCP_Ex1_ex_sol}, where it is plotted as a function of time for different values of $\alpha$ at $x = 0.5$.
	In the sub-parabolic case $\alpha \leq 1$, the solution remains positive for positive $u_0$ and it is monotonously decaying toward zero as $ t \rightarrow T$.
	More specifically, for small $\alpha$ (see graphs for $\alpha = 0.1, 0.3$ in the left plot of \cref{fig:FCP_Ex1_ex_sol}), $|u(t)|$ has a fast initial decay which tends to be getting slower as $t$ progresses.
	This effect becomes less noticeable as $\alpha$ goes toward $1$, at which point $u(t, x) = E_{1,1}(-\pi^2 t)  \sin{\pi x} = e^{-\pi^2 t} \sin{\pi x}$.
	In the sub-hyperbolic case $\alpha > 1$ (see the right plot of \cref{fig:FCP_Ex1_ex_sol}), the solution exhibits more complex behavior.
	It is akin to the damped oscillations with the initial amplitude equal to $u_0$ ($E_{\alpha,2}(0) = 0$ by definition) and the amount of damping that decays as $\alpha$ approaches $2$.
	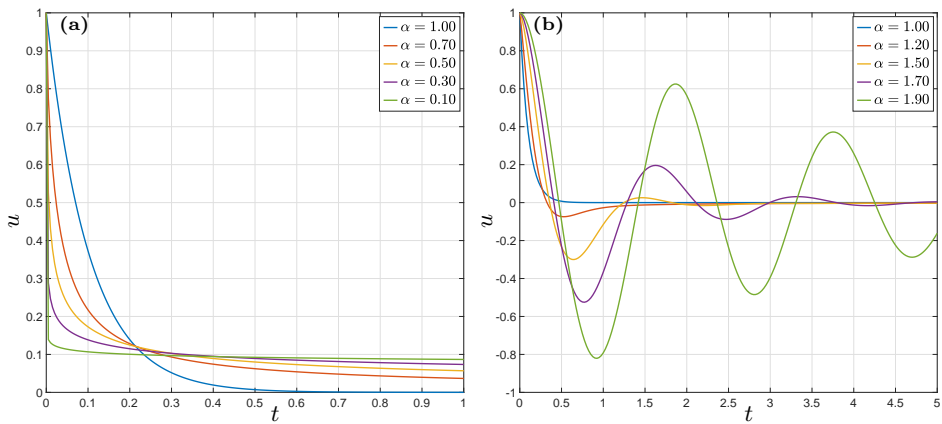
\begin{figure}[h!tb]
		\centering
		\newdimen\lentwosubfig
		\lentwosubfig=0.485\linewidth
		\iflatexml
			\includegraphics[width=0.98\lentwosubfig, viewport=100 60 1100 1000, clip=true]%
			{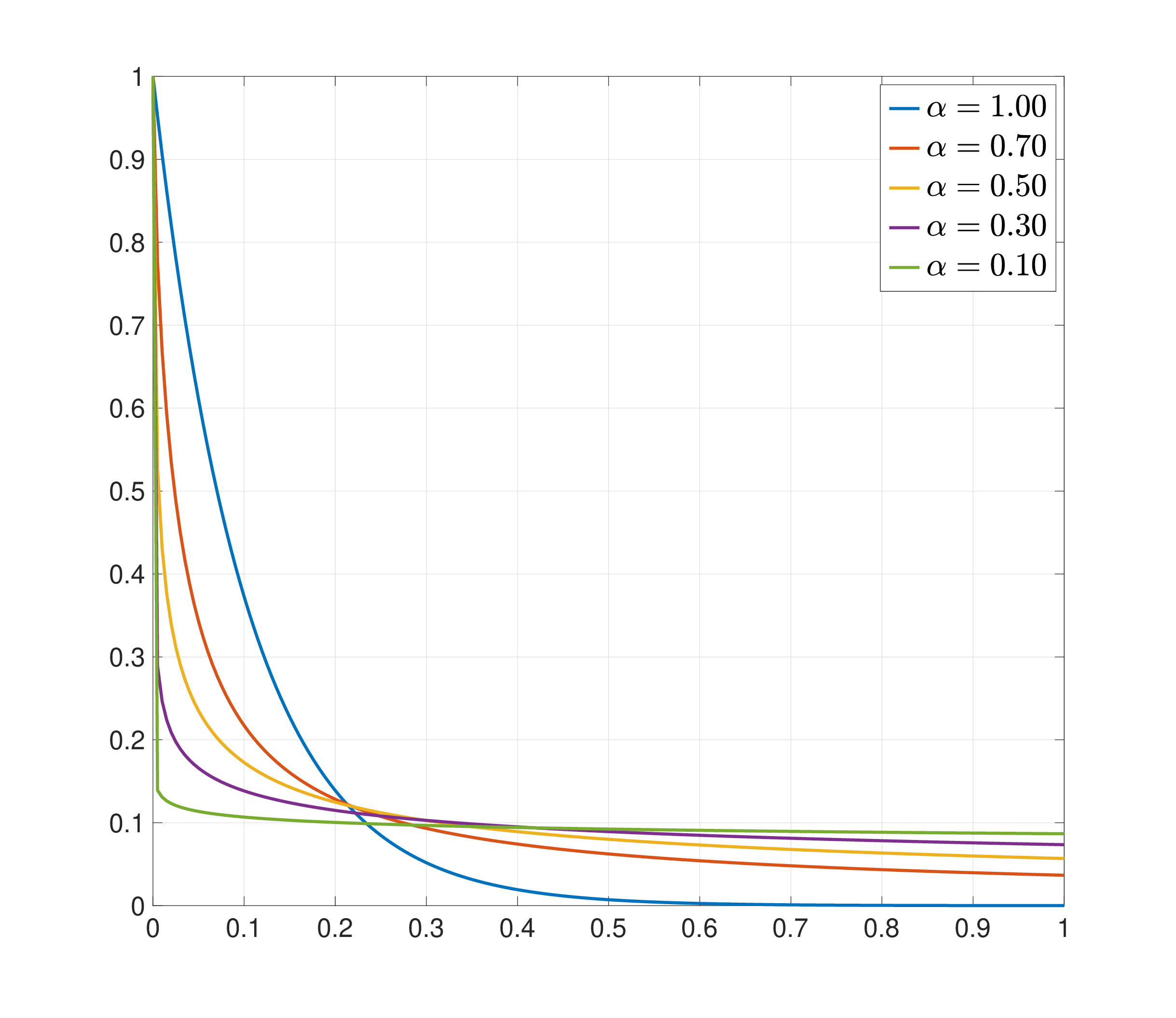}
			\includegraphics[width=0.98\lentwosubfig, viewport=100 60 1100 1000, clip=true]%
			{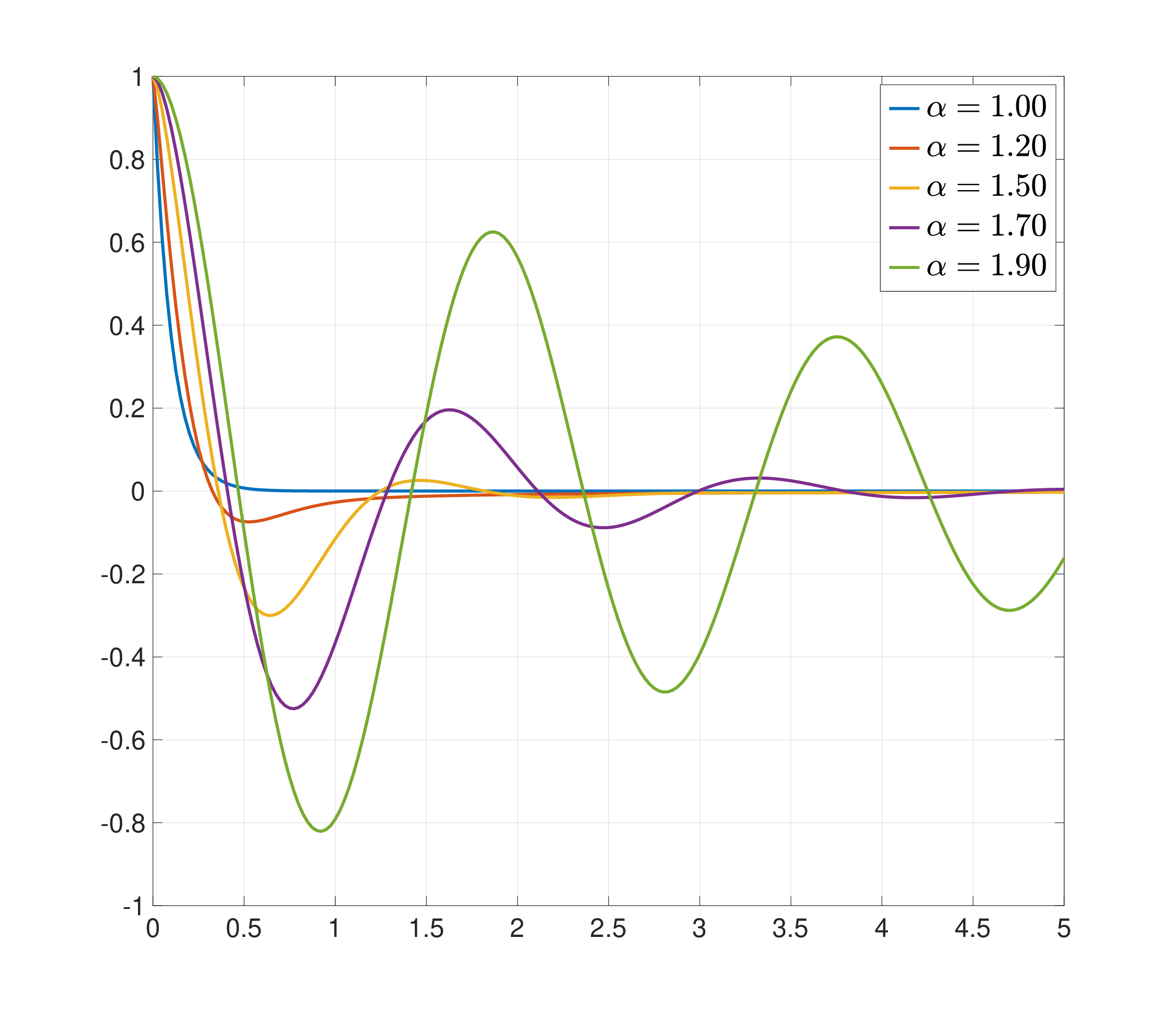}
		\else
			\hspace*{0.4em}
			\begin{overpic}[width=0.98\lentwosubfig, viewport=100 60 1100 1000, clip=true]%
				{Ex1/FCP_sol_ex_vs_t_alpha_0.1_0.3_0.5_1_n0_1_n2_4}
				\put(54,-1){\small $t$}
				\put(-3,41){\small \rotatebox{90}{$u$}}
				\put(8,84){\scriptsize \bf (a)}
			\end{overpic}\hfil
			\begin{overpic}[width=0.98\lentwosubfig, viewport=100 60 1100 1000, clip=true]%
				{Ex1/FCP_sol_ex_vs_t_alpha_1_1.3_1.5_1.7_1.9_n0_1_n2_4}
				\put(54,-1){\small $t$}
				\put(-3,41){\small \rotatebox{90}{$u$}}
				\put(8,84){\scriptsize \bf (b)}
			\end{overpic}
		\fi
		\caption[Example 1: Exact solution]{Exact solution $u(t,0.5)$ of problem \eqref{eq:FCP_DE}, \eqref{eq:FCP_BC} with $f(t) = 0$ and $A$, $u_0$, $u_1$ defined by \eqref {eq:FCP_ex1_hom_A}, \eqref{eq:FCP_ex1_hom_IV} ($L =1$, $k_0 = 1$, $k_1 = 4$, $a = 1$): {\bf (a)} the case $\alpha = 0.1, 0.3, 0.5, 0.7, 1$; {\bf (b)} the case $\alpha = 1, 1.2, 1.5, 1.7, 1.9$.} \label{fig:FCP_Ex1_ex_sol}
	\end{figure}

	To quantify the error of the numerical solution to Cauchy problem \eqref{eq:FCP_DE}, \eqref{eq:FCP_BC}, \eqref{eq:FCP_ex1_hom_A},  \eqref{eq:FCP_ex1_hom_IV}, calculated using \cref{alg:FCP_hom_sol_appr}, we define
	\[
		\cE_{\mathrm{h}}(t,x) = \left| u(t,x) - \wt{u}_{\mathrm{h}}^N(t,x) \right|, \quad \mathrm{err}_{\mathrm{h}}
		=  \sup\limits_{t \in [0, T]} 	\left\|  \cE_{\mathrm{h}}(t)\right\|_\infty.
	\]
	The behavior of $\cE_{\mathrm{h}}(t,x)$ as a function of $t$ for fixed $x = 0.5$, $L =1$, $k_0 = 1$, $k_1 = 4$, $a = 1$, $\varphi_s = \pi/60$, $\gamma =1$  and different values of $\alpha$, $N$ is illustrated in \cref{fig:FCP_Ex1_ex_err_vs_t}, using a semi-logarithmic scale for the plots.
	All graphs clearly illustrate the decay of $\cE_{\mathrm{h}}(t,x)$ on the whole time interval as $N$ increases.
	In the consequence of \cref{thm:FCP_prop_appr}, the error $\cE_{\mathrm{h}}(t,x)$ also depends on how much fractional order $\alpha$ deviates from 1.
	For $\alpha \leq 1$, this happens due to the direct presence of $\alpha$ in error bound~\eqref{eq:FCP_SO_cor_err_est} as a factor.
	For $\alpha>1$, this is explained by the contribution of $\alpha$ to the factor $c = \sqrt{2\pi d}$, from  the same error bound, via \eqref{eq:FCP_hyp_cont_d}.

	Aside of that, for $\alpha \leq 1$, we witness a sharp drop of $\cE_{\mathrm{h}}(t,x)$ in the vicinity $t=0$ (see \cref{fig:FCP_Ex1_ex_err_vs_t}~{(\textbf{a})-(\textbf{c})}), which does not seem to be predicted by the error bound.
	This behavior is attributed to the rather pessimistic estimate $|e^{z(\xi)t}| \leq e^{(a_0 - a(\nu)\cosh{\xi}) t} \leq e^{a_0 t}$, which was used to account for the contribution of the $t$-dependent term into both truncation and discretization errors of $\wt{S}_{\alpha,1}^N(t) u_0$ (see the proof of  \cref{lem:FCP_SO_ISO_Hp_bound,lem:FCP_SO_ISO_trunc_bound} above).
	Similar phenomenon was observed in \cite{McLean2010}, where a related fractional problem was considered. The influence of factor $e^{a_0 t}$ became more evident for larger $t$, as seen from the graphs of \cref{fig:FCP_Ex1_ex_err_vs_t}~{(\textbf{d})-(\textbf{f})}.
	For fixed $N$, the amplitude of error oscillations increases when $t$ approaches $5$ but remains approximately equal $\alpha$-wise (visually larger amplitude oscillations for smaller $\alpha$ in \cref{fig:FCP_Ex1_ex_err_vs_t} are caused by the semi-log nature of the plots).
	This observation supports the theoretical claim from \cref{thm:FCP_prop_appr} that the growth of $\cE_{\mathrm{h}}(t,x)$ in time is not influenced by $\alpha$ or  $d$.
	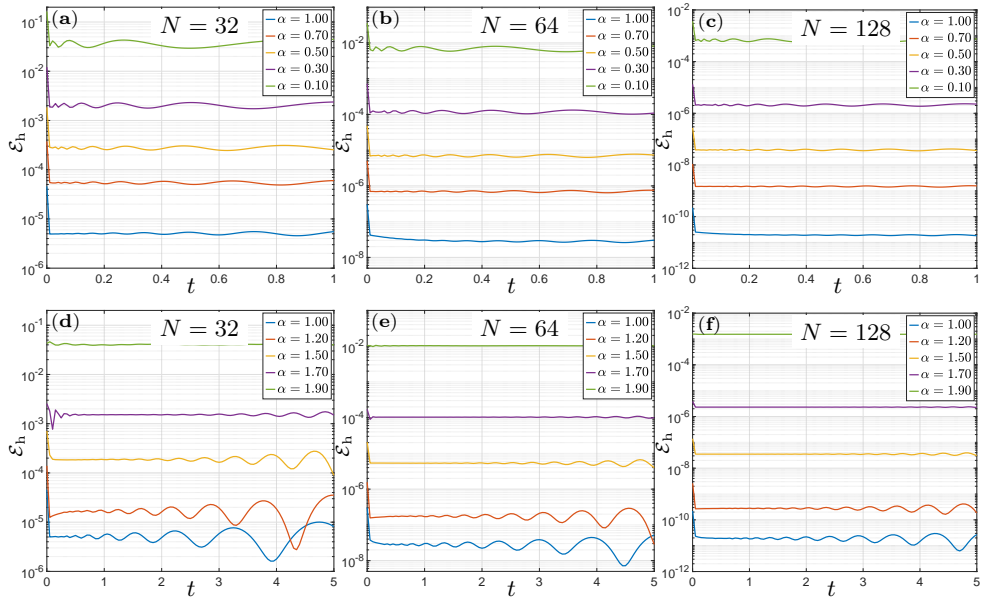
\begin{figure}[h!tb]
		\newdimen\lenthreesubfig
		\lenthreesubfig=0.325\linewidth
		\iflatexml
			\includegraphics[width=0.98\lenthreesubfig, viewport=00 00 1000 910, clip=true]%
			{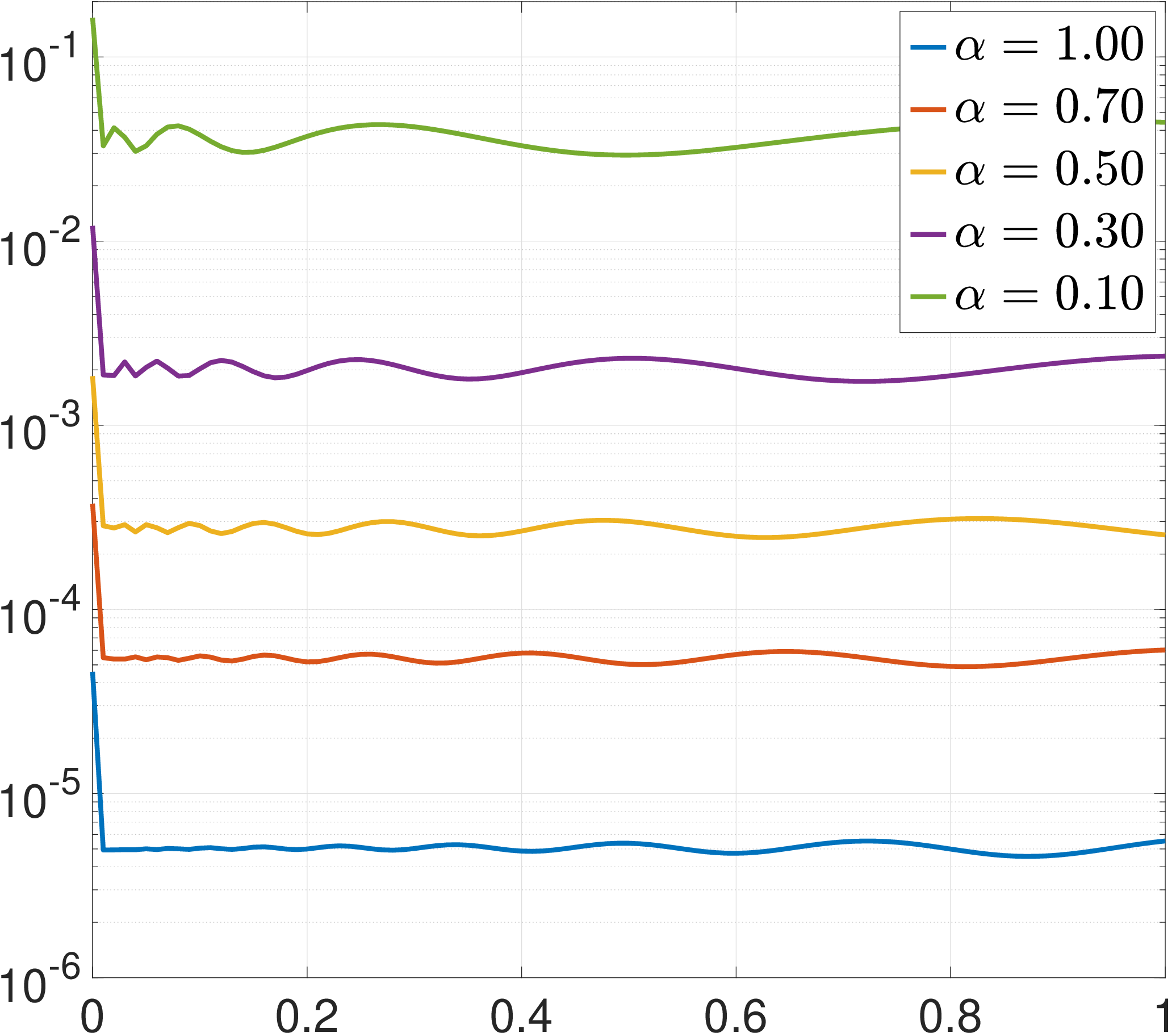}
			\includegraphics[width=0.98\lenthreesubfig, viewport=00 00 1000 910, clip=true]%
			{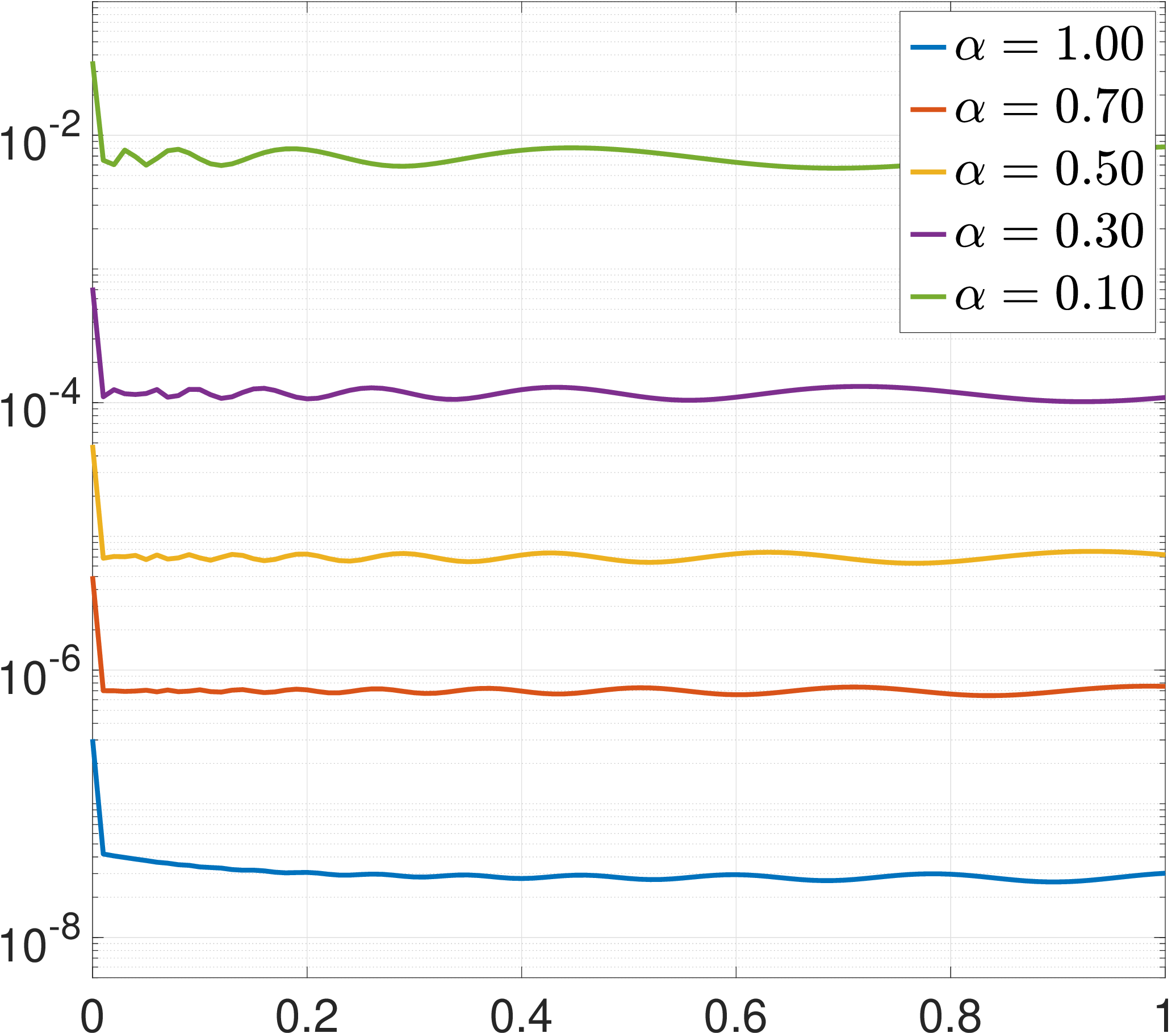}
			\includegraphics[width=0.98\lenthreesubfig, viewport=00 00 1010 910, clip=true]%
			{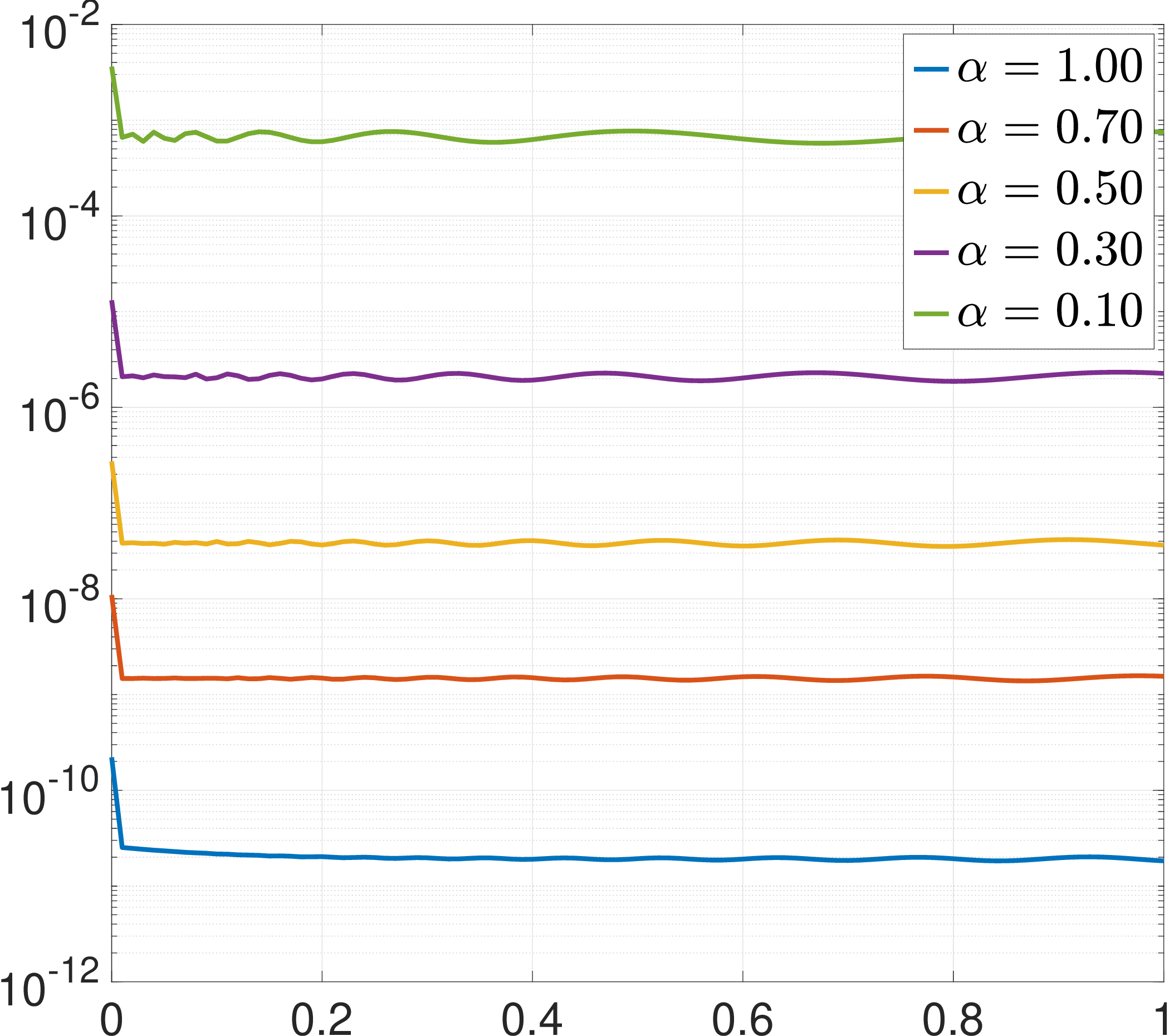}
		\else
			\hspace*{0.2em}
			\begin{overpic}[width=0.98\lenthreesubfig, viewport=00 00 1000 910, clip=true]%
				{Ex1/3x2/FCP_logerr_vs_t_alpha_0.1_0.3_0.5_0.7_1_n0_1_n2_4_N_32}
				\put(41,79){\colorboxo{white}{\small $N = 32$} }
				\put(-4,41){\scriptsize \rotatebox{90}{$\cE_{\mathrm{h}}$}}
				\put(52,-3){\small $t$}
				\put(9,82){\scriptsize ({\bf a})}
			\end{overpic}
			\hspace*{-0.2em}
			\begin{overpic}[width=0.98\lenthreesubfig, viewport=00 00 1000 910, clip=true]%
				{Ex1/3x2/FCP_logerr_vs_t_alpha_0.1_0.3_0.5_0.7_1_n0_1_n2_4_N_64}
				\put(41,79){\colorboxo{white}{\small $N = 64$} }
				\put(-2,41){\scriptsize \rotatebox{90}{$\cE_{\mathrm{h}}$}}
				\put(52,-3){\small $t$}
				\put(9,82){\scriptsize ({\bf b})}
			\end{overpic}
			\hspace*{-0.2em}
			\begin{overpic}[width=0.98\lenthreesubfig, viewport=00 00 1010 910, clip=true]%
				{Ex1/3x2/FCP_logerr_vs_t_alpha_0.1_0.3_0.5_0.7_1_n0_1_n2_4_N_128}
				\put(41,78){\colorboxo{white}{\small $N = 128$} }
				\put(-2,42){\scriptsize \rotatebox{90}{$\cE_{\mathrm{h}}$}}
				\put(54,-3){\small $t$}
				\put(11,81){\scriptsize ({\bf c})}
			\end{overpic}%
		\fi
		\\[6pt]
		\iflatexml
			\includegraphics[width=0.98\lenthreesubfig, viewport=00 00 1000 910, clip=true]%
			{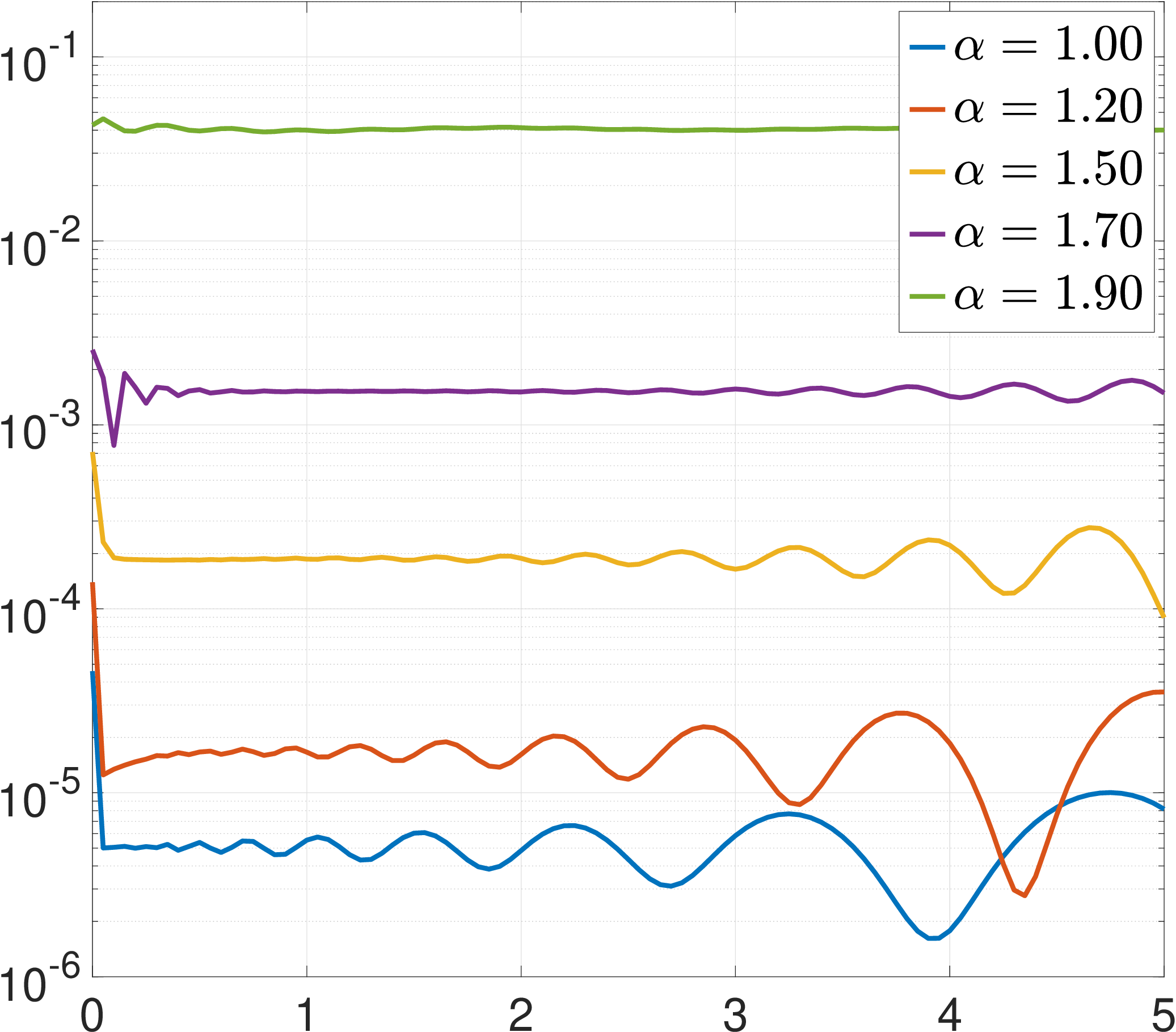}
			\includegraphics[width=0.98\lenthreesubfig, viewport=00 00 1000 910, clip=true]%
			{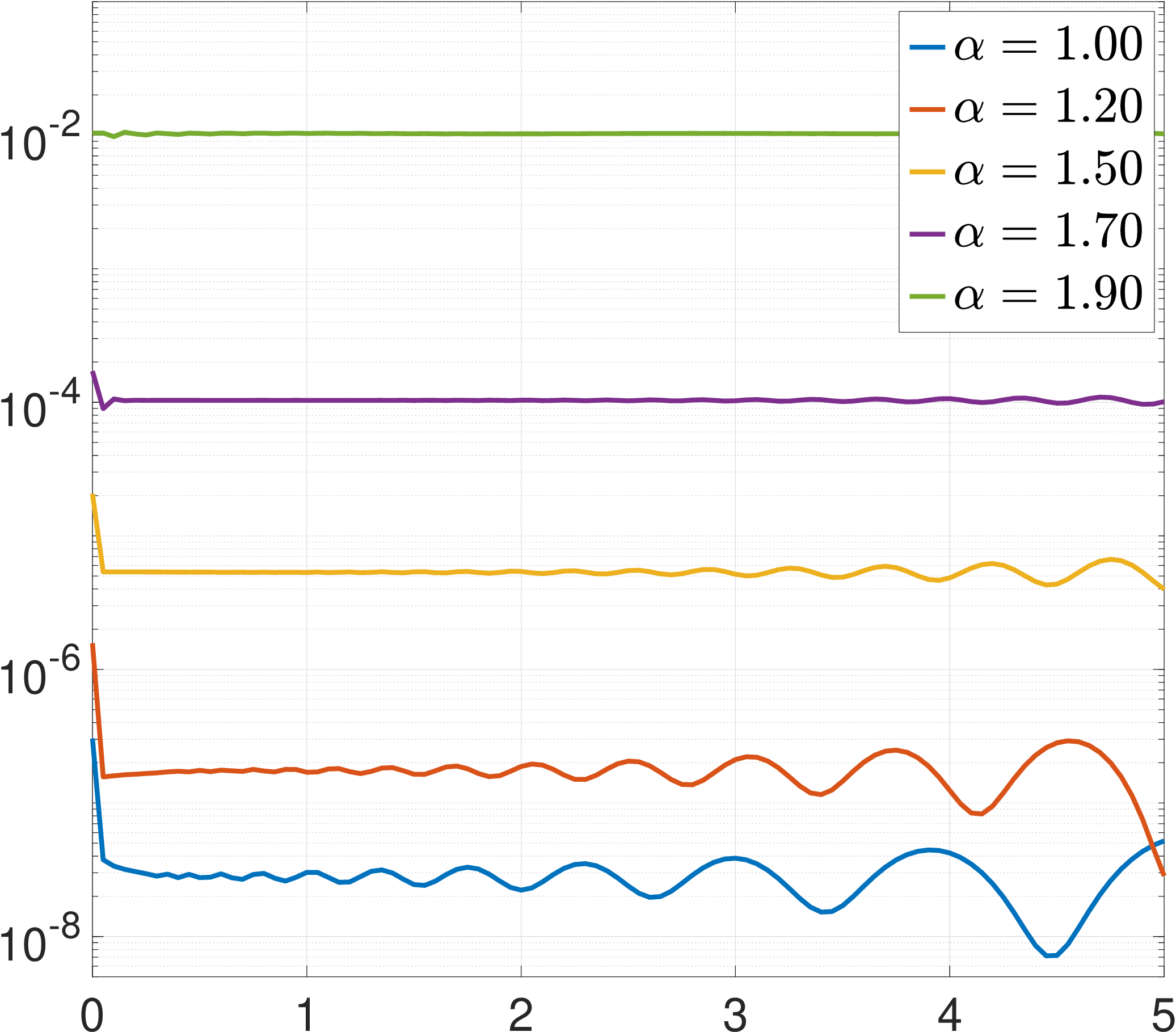}
			\includegraphics[width=0.98\lenthreesubfig, viewport=00 00 1010 910, clip=true]%
			{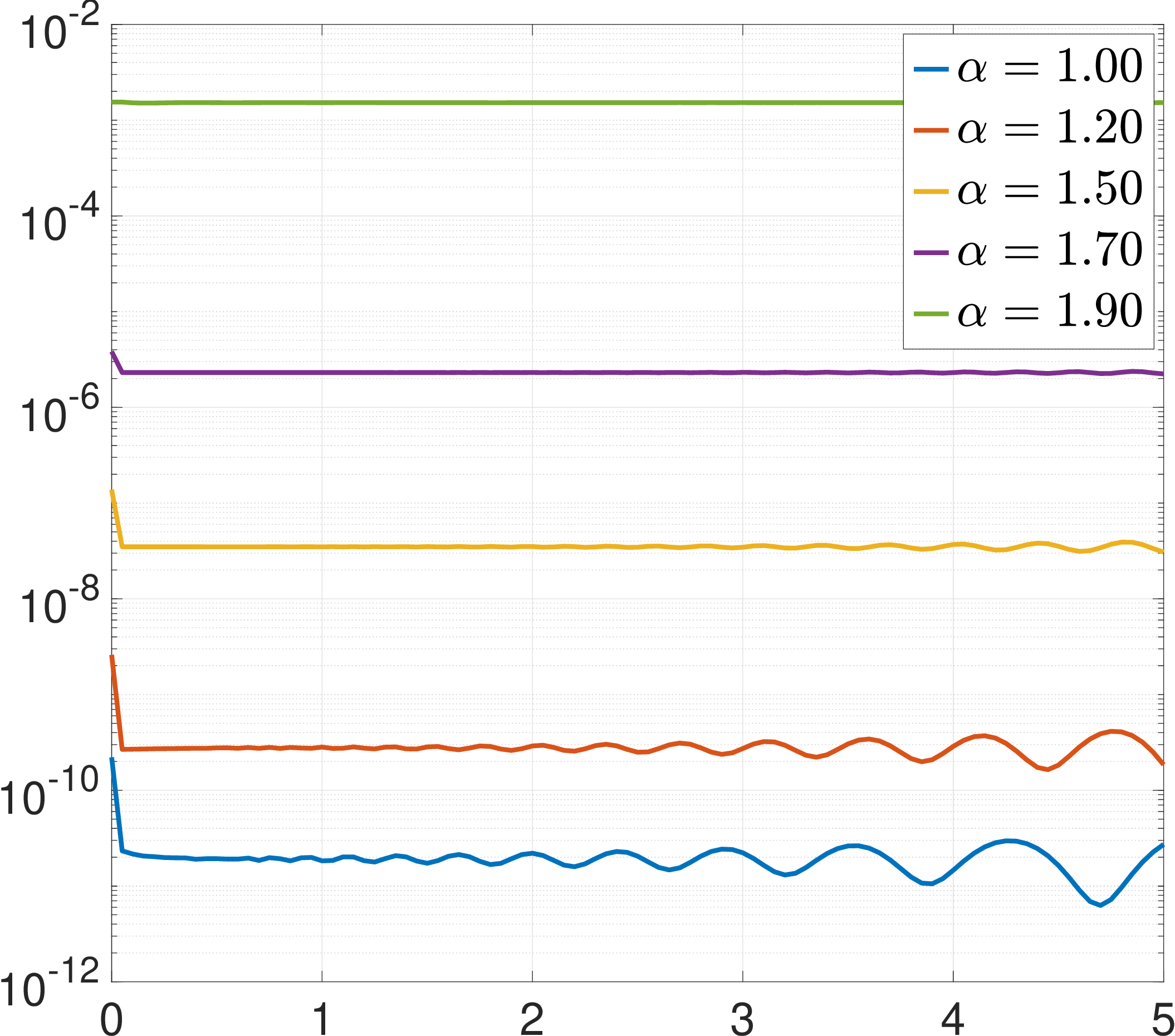}
		\else
			\hspace*{0.2em}
			\begin{overpic}[width=0.98\lenthreesubfig, viewport=00 00 1000 910, clip=true]%
				{Ex1/3x2/FCP_logerr_vs_t_alpha_1_1.2_1.5_1.7_1.9_n0_1_n2_4_N_32}
				\put(41,79){\colorboxo{white}{\small $N = 32$} }
				\put(-4,41){\scriptsize \rotatebox{90}{$\cE_{\mathrm{h}}$}}
				\put(52,-3){\small $t$}
				\put(9,82){\scriptsize ({\bf d})}
			\end{overpic}
			\hspace*{-0.2em}
			\begin{overpic}[width=0.98\lenthreesubfig, viewport=00 00 1000 910, clip=true]%
				{Ex1/3x2/FCP_logerr_vs_t_alpha_1_1.2_1.5_1.7_1.9_n0_1_n2_4_N_64}
				\put(41,79){\colorboxo{white}{\small $N = 64$} }
				\put(-2,41){\scriptsize \rotatebox{90}{$\cE_{\mathrm{h}}$}}
				\put(52,-3){\small $t$}
				\put(9,82){\scriptsize ({\bf e})}
			\end{overpic}
			\hspace*{-0.2em}
			\begin{overpic}[width=0.98\lenthreesubfig, viewport=00 00 1010 910, clip=true]%
				{Ex1/3x2/FCP_logerr_vs_t_alpha_1_1.2_1.5_1.7_1.9_n0_1_n2_4_N_128}
				\put(41,78){\colorboxo{white}{\small $N = 128$} }
				\put(-2,42){\scriptsize \rotatebox{90}{$\cE_{\mathrm{h}}$}}
				\put(52,-3){\small $t$}
				\put(11,81){\scriptsize ({\bf f})}
			\end{overpic}%
		\fi
		\caption[Example 1: Error versus t]{Error $\cE_{\mathrm{h}}(t,0.5)$ of the approximate solution $\wt{u}_{\mathrm{h}}^N$ to problem \eqref{eq:FCP_DE}, \eqref{eq:FCP_BC} with $f(t) = 0$, $A$, $u_0$, $u_1$ being defined by \eqref {eq:FCP_ex1_hom_A}, \eqref{eq:FCP_ex1_hom_IV} and $L =1$, $k_0 = 1$, $k_1 = 4$, $a = 1$. Graphs from the top row of subplots are for $\alpha = 0.1, 0.3, 0.5, 0.7, 1$ and {\bf (a)}  $N = 32$; {\bf (b)} $N = 64$ {\bf (c)}; $N = 128$.  Graphs from the bottom row of plots correspond to $\alpha = 1, 1.2, 1.5, 1.7, 1.9$ and {\bf (d)}  $N = 32$; {\bf (e)} $N = 64$; {\bf (f)} $N = 128$.}
		\label{fig:FCP_Ex1_ex_err_vs_t}
	\end{figure}

	In order to analyze the error dependency on the position of $\mathrm{Sp}(A)$, we evaluate the sup-norm error $\mathrm{err}_{\mathrm{h}}(N)$ for several values of diffusivity constant $a = 10^{-5}$, $0.1,$
	$1, 10$ from \eqref{eq:FCP_ex1_hom_A}, and a range of $\alpha$ values  (see  \cref{fig:FCP_Ex1_ex_err_vs_N}).
	The magnitude of the quantity
	$\rho_s = \inf\limits_{z \in \mathrm{Sp}(A)}\Re{z} = a\pi^2$
	corresponding to $a = 10^{-5}$ in \cref{fig:FCP_Ex1_ex_err_vs_N}~(\textbf{a}) is characteristic for
	problems with a singularly perturbed $A$ \cite{Kadalbajoo2010} and, in particular, advection (convection)-dominated flows \cite{Roos2008}.
	Our prior experiments suggest that existing numerical methods \cite{gm5,McLTh,McLean2010,thomee1}, with the integration contour which lies entirely in the same half-plane as $\mathrm{Sp}(A)$, face certain difficulties in handling problems with such small $\rho_s$.
	Those are caused by the implicit rescaling of $z(\xi)$ needed to fit $z(D_d)$ between $\mathrm{Sp}(A)$ and the origin.
	In contrast, the current method does not experience any accuracy degradation related to $\rho_s \to 0$, because the integration contour $\Gamma_I$ encircles $\mathrm{Sp}(A) \cup \{0\}$.
	In fact, \cref{fig:FCP_Ex1_ex_err_vs_N} shows that the sup-norm error decays exponentially with the order proportional to $\sqrt{\alpha N}$ as prescribed by \eqref{eq:FCP_hom_sol_err_est}, for all analyzed values of $a$.
	The convergence results of our method for $\alpha \leq 1$ are similar to those obtained in \cite{Colbrook2022a} for the specific case of \eqref{eq:FCP_DE}, when $A$ comes from the viscoelastic beam model.
	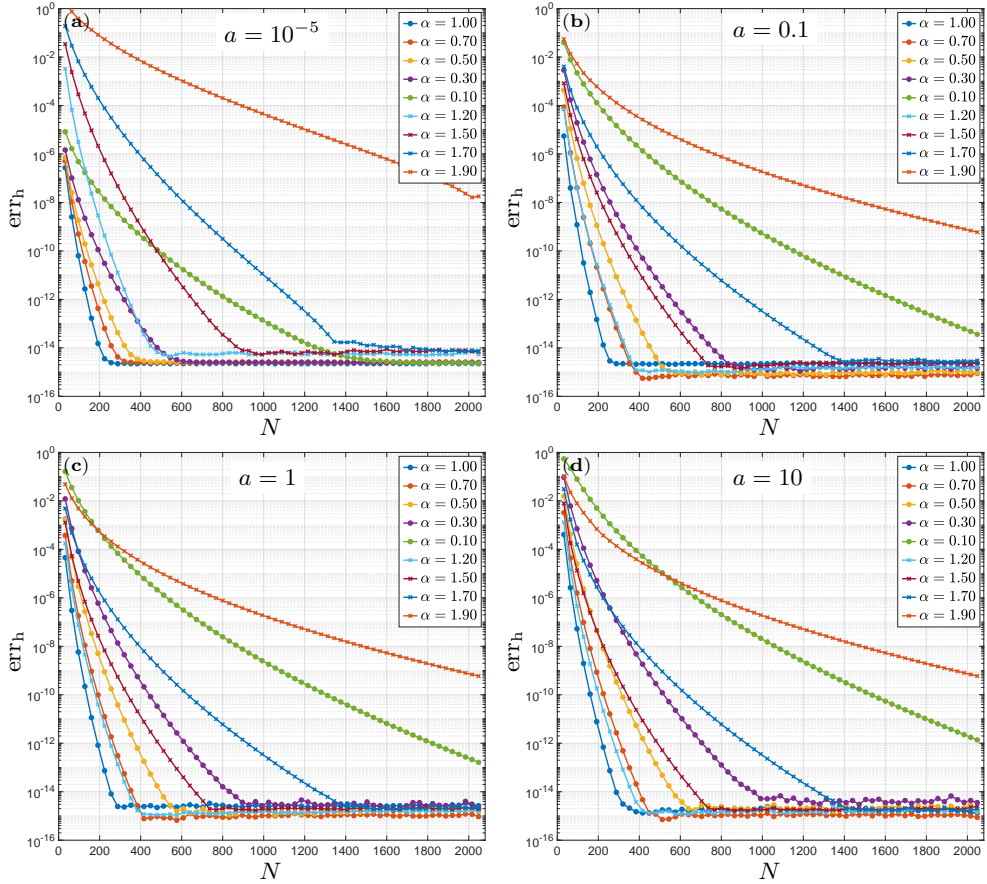
\begin{figure}[h!tb]
		\newdimen\lentwosubfig
		\lentwosubfig=0.48\linewidth
		\iflatexml
			\includegraphics[width=0.98\lentwosubfig, viewport=80 20 1100 960, clip=true]%
			{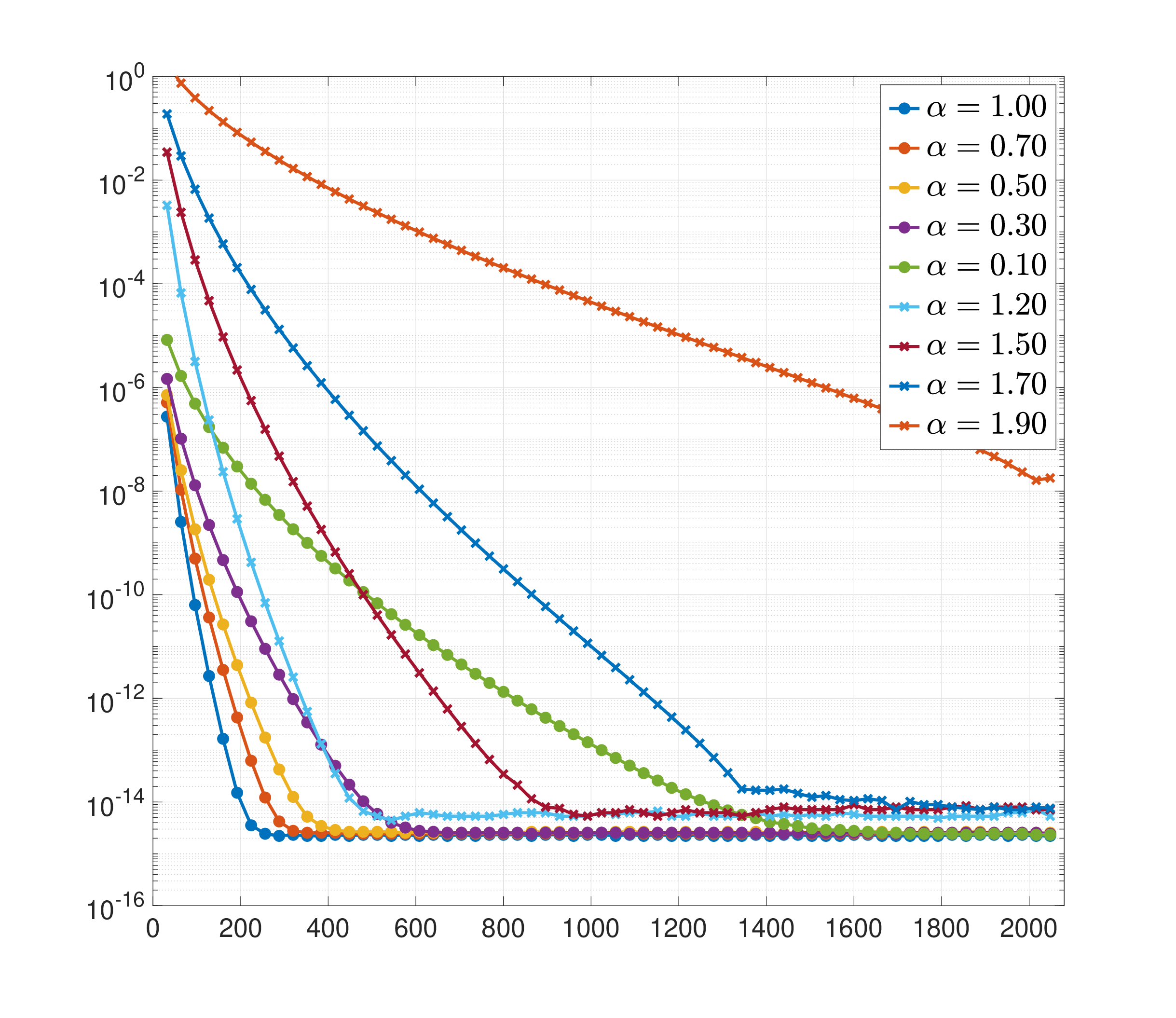}
			\includegraphics[width=0.98\lentwosubfig, viewport=80 20 1100 960, clip=true]%
			{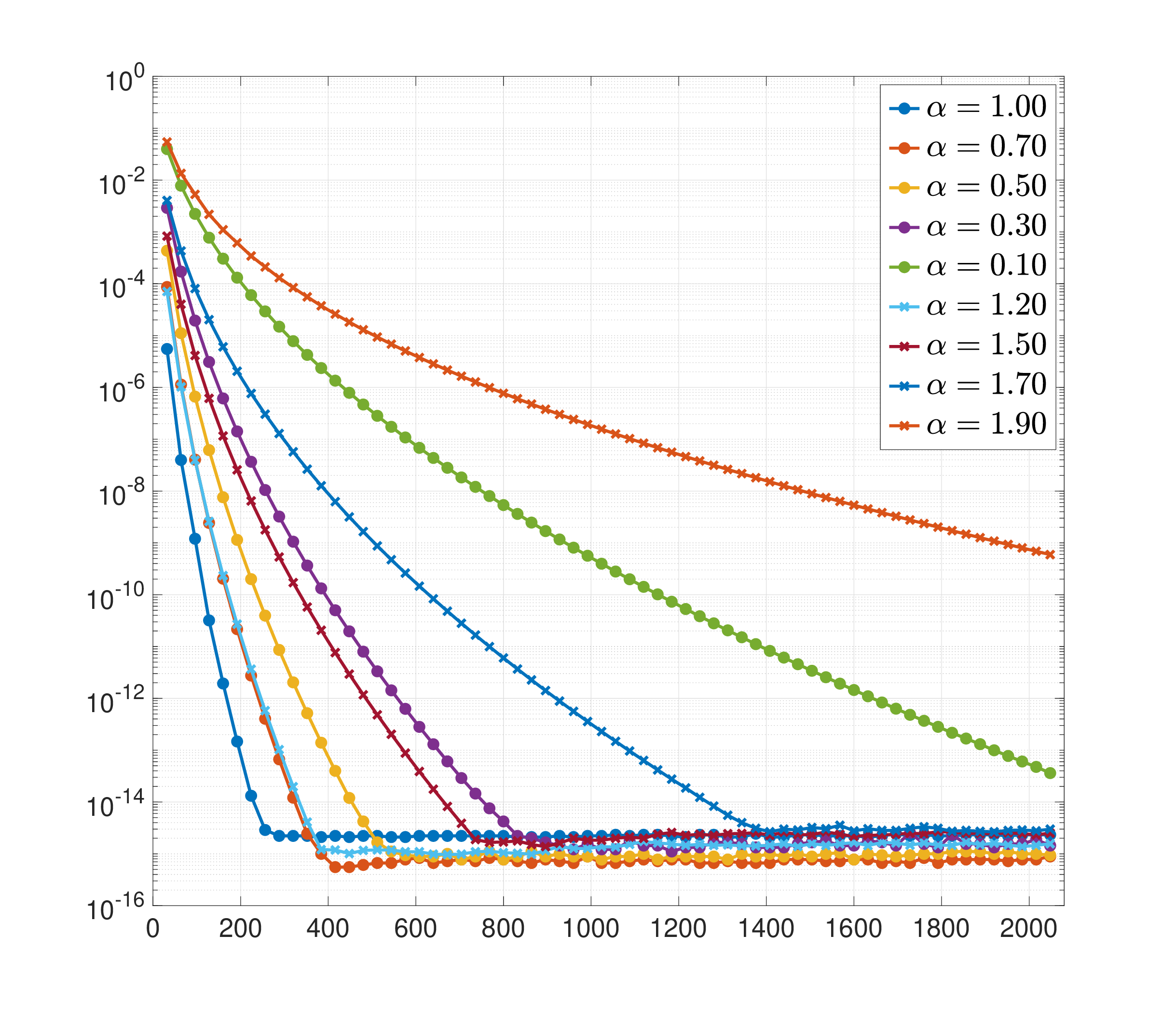}
		\else
			\hspace*{0.1em}
			\begin{overpic}[width=0.98\lentwosubfig, viewport=80 20 1100 960, clip=true]%
				{Ex1/FCP_error_vs_N_alpha_0.1_0.3_0.5_0.7_1_1.2_1.5_1.7_1.9_n0_1_n2_4_D_0.00001}
				\put(40,83){\colorboxo{white}{\small $a = 10^{-5}$} }
				\put(49,1){\small $N$}
				\put(-3,44){\small \rotatebox{90}{$\mathrm{err}_{\mathrm{h}}$}}
				\put(8,86){\scriptsize ({\bf a})}
			\end{overpic}
			\hspace*{0.5em}
			\begin{overpic}[width=0.98\lentwosubfig, viewport=80 20 1100 960, clip=true]%
				{Ex1/FCP_error_vs_N_alpha_0.1_0.3_0.5_0.7_1_1.2_1.5_1.7_1.9_n0_1_n2_4_D_0.1}
				\put(42,84){\colorboxo{white}{\small $ a = 0.1$} }
				\put(49,1){\small $N$}
				\put(-3,44){\small \rotatebox{90}{$\mathrm{err}_{\mathrm{h}}$}}
				\put(8,86){\scriptsize ({\bf b})}
			\end{overpic}
		\fi
		\\
		\iflatexml
			\includegraphics[width=0.98\lentwosubfig, viewport=80 20 1100 960, clip=true]%
			{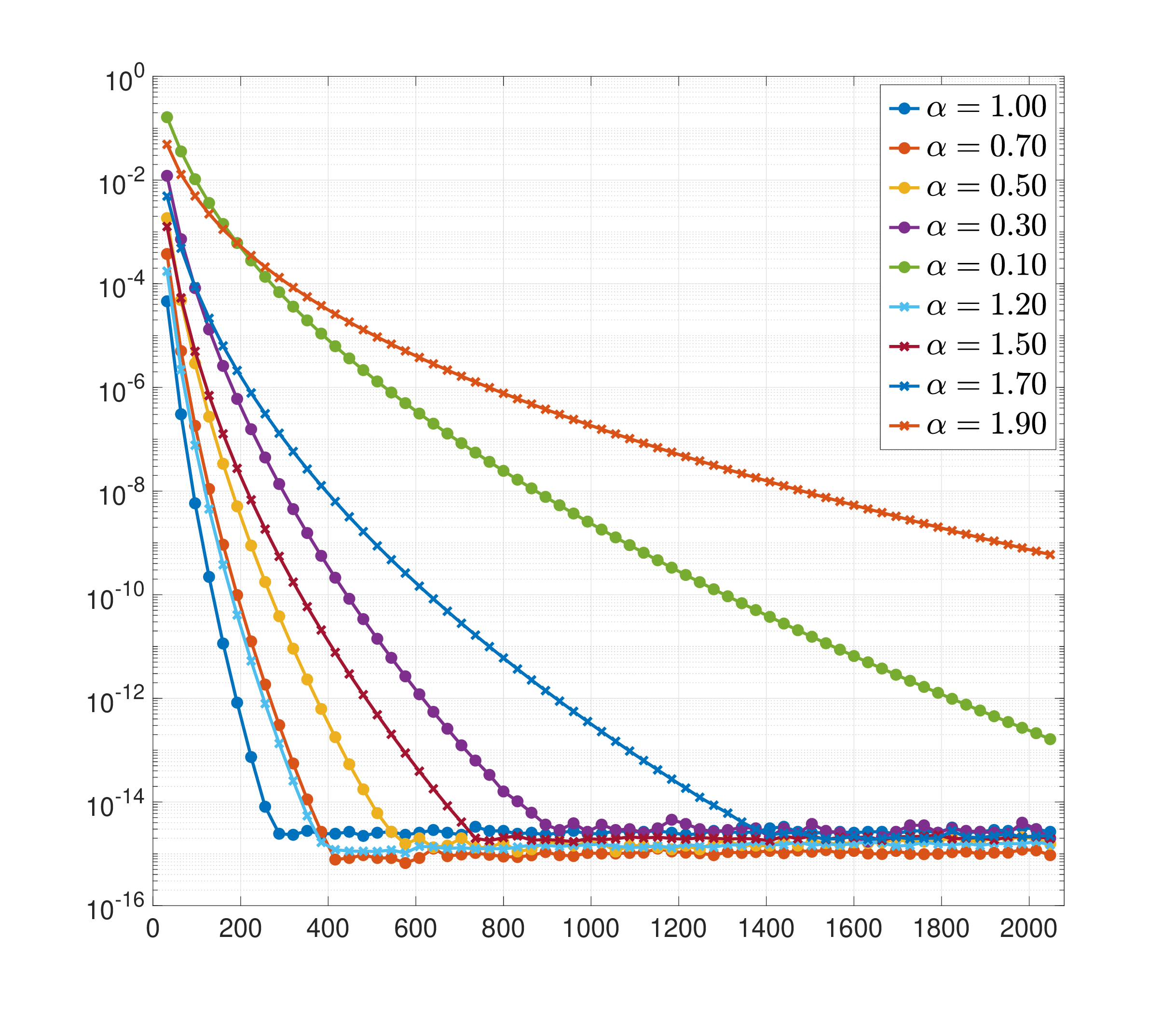}
			\includegraphics[width=0.98\lentwosubfig, viewport=80 20 1100 960, clip=true]%
			{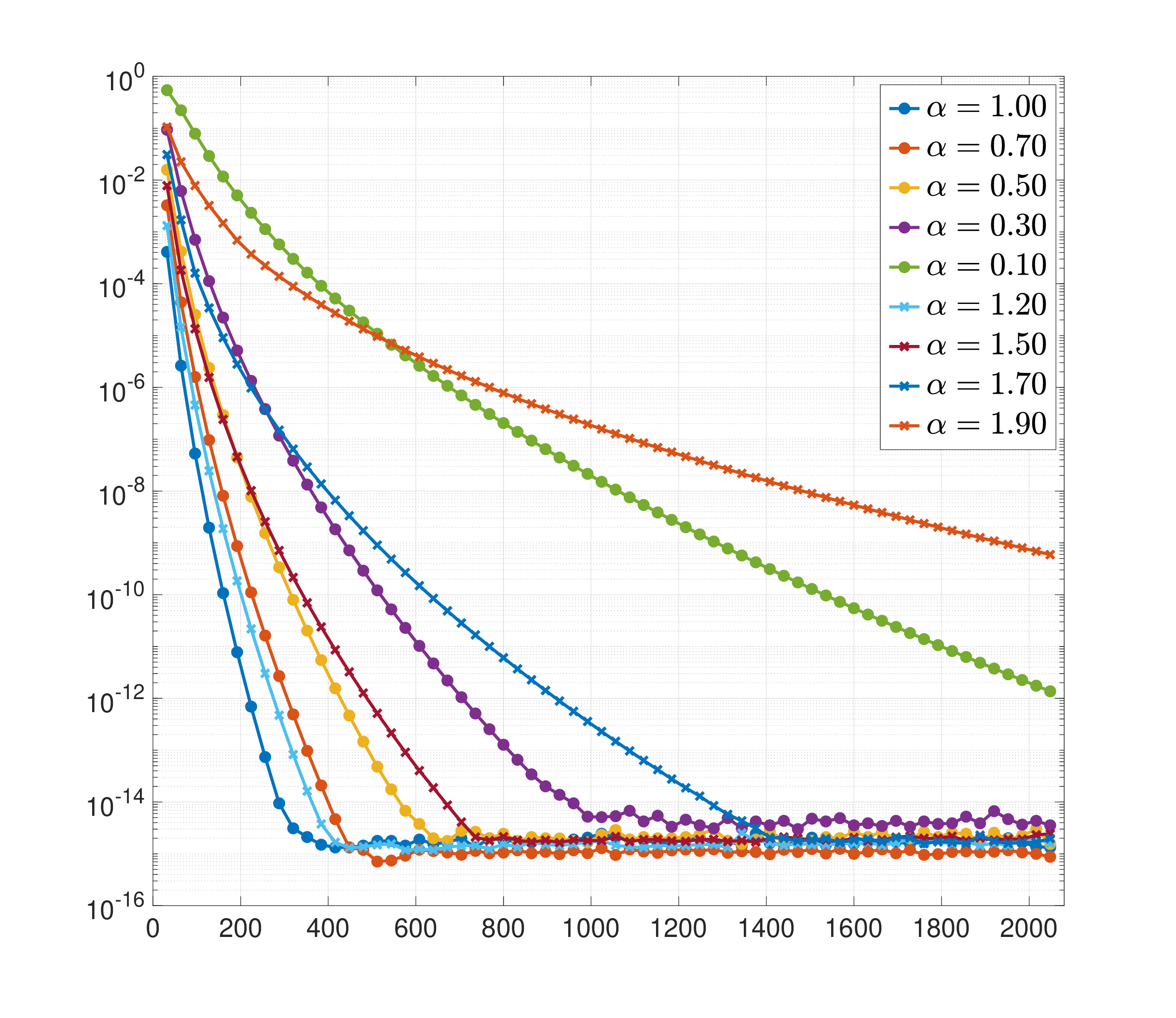}
		\else
			\hspace*{0.1em}
			\begin{overpic}[width=0.98\lentwosubfig, viewport=80 20 1100 960, clip=true]%
				{Ex1/FCP_error_vs_N_alpha_0.1_0.3_0.5_0.7_1_1.2_1.5_1.7_1.9_n0_1_n2_4_D_1}
				\put(43,83){\colorboxo{white}{\small $ a = 1$} }
				\put(49,1){\small $N$}
				\put(-3,44){\small \rotatebox{90}{$\mathrm{err}_{\mathrm{h}}$}}
				\put(8,86){\scriptsize ({\bf c})}
			\end{overpic}
			\hspace*{0.5em}
			\begin{overpic}[width=0.98\lentwosubfig, viewport=80 20 1100 960, clip=true]%
				{Ex1/FCP_error_vs_N_alpha_0.1_0.3_0.5_0.7_1_1.2_1.5_1.7_1.9_n0_1_n2_4_D_10}
				\put(42,83){\colorboxo{white}{\small $ a = 10$} }
				\put(49,1){\small $N$}
				\put(-3,44){\small \rotatebox{90}{$\mathrm{err}_{\mathrm{h}}$}}
				\put(8,86){\scriptsize ({\bf d})}
			\end{overpic}
		\fi
		\caption[Example 1: Error versus N]{Sup-norm error $\mathrm{err}_{\mathrm{h}}$ of the approximate solution $\wt{u}_{\mathrm{h}}^N$ to problem \eqref{eq:FCP_DE}, \eqref{eq:FCP_BC} with $f(t) = 0$, $A$, $u_0$, $u_1$ being defined by \eqref {eq:FCP_ex1_hom_A}, \eqref{eq:FCP_ex1_hom_IV} and $L =1$, $k_0 = 1$, $k_1 = 4$. Graphs from sublots correspond to the different values of diffusivity constant: {\bf (a)} $a = 1 \times 10^{-5}$; {\bf (b)} $a = 0.1$; {\bf (c)} $a = 1$; {\bf (d)} $a = 10$; }
		\label{fig:FCP_Ex1_ex_err_vs_N}
	\end{figure}
\end{example}

\subsection{Numerical Scheme for the Inhomogeneous Part}\label{sec:FCP_inhom_sol_appr}
In this part, we apply the propagator approximation method from \cref{sec:FCP_Prop_Approx} to obtain an efficient numerical algorithm for the inhomogeneous part $u_{\mathrm{ih}}(t)$ of the mild solution to~\eqref{eq:FCP_DE}, \eqref{eq:FCP_BC}, defined by \eqref{eq:FCP_hom_inhom}.
This formula combines the action of $S_{\alpha}(t)$ on a certain vector from $X$ with the subsequent action of the integral operator.
The numerical evaluation of such composition amounts to the reevaluation of  $S_{\alpha}(s)$ at each quadrature point $\left\{s_k\right\}_{k=1}^K$,  needed to approximate the outer integral.
As we have learned from the properties of the numerical method developed in \cref{sec:FCP_Prop_Approx},
this is not a problem for the first term of $u_{\mathrm{ih}}(t)$, where the argument $f(0)$ is fixed, because only $2N+1$ parallel resolvent evaluations are needed.
For the second term, however, the numerical evaluation of $S_{\alpha}(t-s) J_\alpha f'(s)$ for every new value of $t$ requires the reevaluation of resolvents for the entire set of new quadrature points on $\Gamma_I$.
This leads to the solution of up to $(2N+1)K$ additional stationary problems and may require additional storage and inter-process communication, when the parallel computing model is used for evaluation.
To reduce the number of required resolvent evaluations, we take advantage of the fact that the operator-dependent part of $S_{\alpha}(t)$ in representation \eqref{eq:FCP_SO_exp_cor_repr_par}  is itself a linear operator on $X$; hence, it can be interchanged with the integral operator acting in $t$ only
\begin{align*}
	 & \intl_0^t S_{\alpha}(t - s)   J_\alpha f'(s)\, ds
	= \intl_0^t \frac{1}{2\pi i}\intl_\Gamma  e^{z(t-s)}  z^{\alpha-1} (z^\alpha I + A)^{-1}  J_\alpha f'(s) \, dz \, ds                                                                         \\
	 & = 	\intl_0^t\frac{1}{2\pi i}   \intl_\Gamma   e^{z(t-s)}  \left( z^{\alpha-1} (z^\alpha I + A)^{-1} - \frac{1}{z}I \right) J_\alpha f'(s) \, dz +  J_\alpha f'(t)\, ds                    \\
	 & = 	\intl_0^t J_\alpha f'(s)\, ds + \frac{1}{2\pi i}   \intl_\Gamma   \left( z^{\alpha-1} (z^\alpha I + A)^{-1} - \frac{1}{z}I \right) \intl_0^t e^{z(t-s)}  J_\alpha f'(s) \, ds \, dz  .
\end{align*}
Here, we used formula \eqref{eq:FCP_SO_cor_repr} under the assumption that $J_\alpha f'(s) \in D(A^\gamma)$, $\gamma>0$ and, then, relied upon the uniform convergence of the corrected representation of $S_{\alpha}(t-s) $  with respect to $s \in [0, t]$, that had been established earlier.
As we can see from the newly obtained representation, now the evaluation of the time-dependent part is performed in the resolvent's argument.
This reduces the number of parallel resolvent evaluations per $t$ to $2N+1$ for this term.
The last representation permits us to rewrite the inhomogeneous part of solution $u_{\mathrm{ih}}(t)$ in the form
\vspace*{-6pt}
\begin{equation}\label{eq:FCP_inhom_cor_repr}
	\begin{aligned}
		u_{\mathrm{ih}}(t)  = & J_\alpha S_{\alpha}(t)f(0) + \intl_0^t J_\alpha f'(s)\, ds                                                           \\[-4pt]
		                      & +  \frac{1}{2\pi i}\intl_{-\infty}^\infty F_{\alpha,1}(\xi)\intl_0^t e^{z(\xi)(t-s)}  J_\alpha f'(s) \, ds \, d\xi .
	\end{aligned}
\end{equation}

Next, we address another ingredient essential to the numerical  evaluation of \eqref{eq:FCP_inhom_cor_repr}, which is an efficient quadrature method for the Riemann--Liouville integral  $J_\alpha v(t)$ defined by~\eqref{eq:FCP_RLInt}.
While evaluating this integral numerically, it is important to select the quadrature rule that, on the one hand, can handle the endpoint singularity appearing in the integrand when $\alpha<1$ and, on the other hand, is able to provide exponentially convergent approximation.
Among existing quadrature rules, only sinc-quadrature on a finite interval satisfies two mentioned properties simultaneously (see \cite{Stenger1993}).
We construct a version of such quadrature rule by transforming $J_\alpha$ into the integral over $(-\infty, \infty)$ and then applying the chosen sinc-quadrature formula.
Let $s = {t e^p}/\left ({1+e^p}\right )$, then
\[\begin{aligned}
		J_\alpha v(t)
		 & = \frac{1}{\Gamma(\alpha)}  \int\limits_{-\infty}^{\infty} \left( t - \frac{t e^p}{1+e^p}\right)^{\alpha -1 }v\left ( \frac{t e^p}{1+e^p} \right ) \, ds \\
		 & = \frac{t^{\alpha}}{\Gamma(\alpha)} \int\limits_{-\infty}^{\infty} \frac{e^p}{(1+e^p)^{\alpha + 1}} v\left ( \frac{t e^p}{1+e^p} \right ) \, dp.
	\end{aligned}\]
The reader may note that the singularity $(t-s)^{\alpha-1}$  from original definition \eqref{eq:FCP_RLInt} of $J_\alpha v(t)$ is no longer present in the last integral and
the new kernel of $J_\alpha v(t)$ decays exponentially as $p \rightarrow \infty$.
More precisely, there exist a constant $c>0$, such that
\begin{equation}\label{eq:RLInt_Decay}
	\frac{e^p}{(1+e^p)^{\alpha + 1}}
	= \left(e^{-\frac{p}{\alpha+1}}+e^{\frac{\alpha p}{\alpha + 1}} \right)^{-(\alpha + 1)}
	\leq c\begin{cases}
		e^{-\alpha p}, & p > 0, \\
		e^{p},         & p < 0. \\
	\end{cases}
\end{equation}

Our intent here is to approximate $J_\alpha v(t)$ by the time-dependent operator $\wt{J}_\alpha^N v(t)$, that takes into account the difference in a speed of kernel's decay as $p\rightarrow \pm \infty$, illustrated by the above bound.
To introduce the approximation $\wt{J}_\alpha^N v(t)$ properly, let us to recall the following definition \cite[Definition 3.1.5]{Stenger1993}.
The function $f$ is said to belong to the class $\mathbf{L}_{a,b}(D_d)$ if it is analytic in $D_d$ and there exist a constant $c > 0$ such that for all $z \in D_d$:
\[
	|v(z)|  \leq c\frac{|e^z|^a}{(1 + |e^z|)^{a + b}}.
\]
The constants $a, b > 0$ will be referred to as the decay orders (or the decay order if $a = b$).

\begin{proposition}\label{prop:FCP_RLInt}
	Assume that the function $v$: $[0,T] \rightarrow X$ is bounded $\|v(t)\| < \infty$, for any $t \in [0,T]$.
	If $v(z)$ admits analytic extension to the ``eye-shaped" region
	\begin{equation}\label{eq:FCP_D_d_2}
		D_d^2 = \left\{z \in \oC: \left|\arg{\left(\frac{z}{T-z}\right)}\right| < d \right\},
	\end{equation}
	for some $d \in (0, \pi/2)$, then the operator
	\begin{equation}\label{eq:FCP_RLInt_appr}
		\begin{aligned}
			\wt{J}_\alpha^N v(t)
			=  \frac{t^{\alpha}h }{\Gamma(\alpha)}\sum\limits_{k =-\lceil \varepsilon N \rceil}^{\lceil\delta N\rceil}  \frac{e^{kh}}{(1+e^{kh})^{\alpha + 1}} v\left ( \frac{t e^{kh}}{1+e^{kh}} \right ), \\
		\end{aligned}
	\end{equation}
	with $N \in \N$, $\varepsilon = \min\left \{1, \alpha\right \}$, $\delta = \min\left \{\frac{1}{\alpha}, 1\right \}$ and $h = \sqrt{\frac{2 \pi d}{\varepsilon N}}$,
	defines the convergent approximation to
	\begin{equation}\label{eq:FCP_RLInt_R}
		J_\alpha v(t)
		= \frac{t^{\alpha}}{\Gamma(\alpha)} \int\limits_{-\infty}^{\infty} \frac{e^p}{(1+e^p)^{\alpha + 1}} v\left ( \frac{t e^p}{1+e^p} \right ) \, dp.
	\end{equation}
	Moreover, for  all $t \in [0,T]$
	\begin{equation}\label{eq:FCP_RLInt_err_est}
		\left\|J_\alpha v(t) - \wt{J}_\alpha^N v(t)\right\|
		\leq \frac{C t^\alpha}{\Gamma(\alpha)} e^{- \sqrt{2 \pi d \varepsilon N} },
	\end{equation}
	where the constant $C>0$ is independent of $N$ and $t$.
\end{proposition}
\begin{proof}
	When $t = T$, the function $t{e^z}/(1+e^z)$ maps the infinite horizontal strip $D_d$ of half-height $d$ into the ``eye-shaped'' region $D_d^2$  (see \cite[Example 1.7.5]{Stenger1993}) around the interval $[0, T]$.
	For smaller values of $t$,  it maps $D_d$ into the region $tD_d^2 \equiv \left \{z \in C: z T/t \in D_d^2 \right \}$, which is a proper subset of $D_d^2$ as long as $t < T$.
	Consequently, if the assumptions regarding $v(z)$ are fulfilled, the integrand from \eqref{eq:FCP_RLInt_R} belongs to the class of functions $\mathbf{L}_{1,\alpha}(D_{d})$ for any $t \in (0, T]$.
	Then, the results regarding the convergence of \eqref{eq:FCP_RLInt_appr} to  \eqref{eq:FCP_RLInt_R}, as well as the form of~\eqref{eq:FCP_RLInt_appr} itself, and the error estimate stated in \eqref{eq:FCP_RLInt_err_est} follow from \cite[Theorem 4.2.6]{Stenger1993}.
\end{proof}

\begin{remark}\label{rem:FCP_RLInt_appr_ext}
	The results of \cref{prop:FCP_RLInt} remain valid if, instead of the boundedness of $v(z)$, we assume that the integrand from \eqref{eq:FCP_RLInt_R} belongs to the class $\mathbf{L}_{a,b}(D_d)$.
	In such case, the parameters $\epsilon, \delta$ from \eqref{eq:FCP_RLInt_appr} should be determined by $\epsilon = \min\{a,b\}/a$, $\delta = \max\{a,b\}/b$.
\end{remark}
The presence of factor $t^\alpha$ in error estimate \eqref{eq:FCP_RLInt_err_est} makes it possible to use fewer terms in~\eqref{eq:FCP_RLInt_appr} as $t$ decreases, if the end goal is to reach the prescribed accuracy uniformly in $t$.
Let us assume that the desired accuracy is achieved for some $t_0$ by setting $N=N_0$, then \mbox{for $t \in (0, t_0)$:}
\[
	t^\alpha e^{- \sqrt{2 \pi d \varepsilon N} } = t_0^\alpha e^{- \sqrt{2 \pi d \varepsilon N_0} }.
\]
After solving this equation for $N$, we obtain
\begin{equation}\label{eq:FCP_RLInt_appr_N_adj}
	N(t) = \left (\sqrt{N_0}+\frac{\alpha\ln{(t/t_0)}}{\sqrt{2 \pi d \varepsilon}} \right )^2.
\end{equation}
Formula \eqref{eq:FCP_RLInt_appr_N_adj} becomes instrumental in the situations where one needs to numerically evaluate $\wt{J}_\alpha^N v(t)$ for a range of $t$-values.
This is the case of the inhomogeneous part of solution representation given by \eqref{eq:FCP_inhom_cor_repr}, whose terms contain the integrals of $J_\alpha f'(s)$.

Before addressing the question on how to numerically evaluate \eqref{eq:FCP_inhom_cor_repr}, we would like to consider a prerequisite problem on how to quantify the contribution of the error in the argument $v(t)$ of $\wt{J}_\alpha^N v(t)$ to the overall error of approximation to $J_\alpha v(t)$.
\begin{corollary}\label{eq:FCP_RLint_NA}
	Assume that functions $v$, $\wt{v}$ satisfy the assumptions of \cref{prop:FCP_RLInt}.
	If $\|v(t) - \wt{v}(t)\| \leq \varkappa $, for all $t \in [0, T]$, then the error of approximation  $\wt{J}_\alpha^N \wt{v}(t)$, satisfies the bound
	\begin{equation}\label{eq:FCP_RLInt_biasederr}
		\left\|J_\alpha v(t) - \wt{J}_\alpha^N \wt{v}(t) \right\|
		\leq  \frac{C t^\alpha}{\Gamma(\alpha)} e^{- \sqrt{2 \pi d \varepsilon N} } + \frac{ t^{\alpha}(\alpha + 1)}{\Gamma(\alpha+1)}
		\varkappa .
	\end{equation}
\end{corollary}
\begin{proof}
	We rewrite \eqref{eq:FCP_RLInt_biasederr} as
	\[
		\left\|J_\alpha v(t) - \wt{J}_\alpha^N \wt{v}(t) \right\|
		\leq 	\left\|J_\alpha v(t) - \wt{J}_\alpha^N v(t) \right\|
		+ 	\left\| \wt{J}_\alpha^N v(t) - \wt{J}_\alpha^N \wt{v}(t) \right\|.
	\]
	The first term of this error decomposition is estimated by \eqref{eq:FCP_RLInt_err_est}, so we focus on the second~term
	\begin{equation}\label{eq:FCP_RLInt_biasederr1}
		\left\| \wt{J}_\alpha^N v(t) - \wt{J}_\alpha^N \wt{v}(t) \right\|
		\leq \frac{\varkappa t^{\alpha}h }{\Gamma(\alpha)}\sum\limits_{k =-[\varepsilon N]}^{[\delta N]}  \frac{e^{kh}}{(1+e^{kh})^{\alpha + 1}} ,
	\end{equation}
	where
	\begin{equation}\label{eq:FCP_RLint_sbound_alpha}
		\begin{aligned}
			h\sum\limits_{k =-[\varepsilon N]}^{[\delta N]}  \frac{e^{kh}}{(1+e^{kh})^{\alpha + 1}} & \leq  h\sum\limits_{k =0}^{[\delta N]}  e^{-kh \alpha}
			+  h\sum\limits_{k =1}^{[\varepsilon N]}  \frac{e^{-kh}}{(1+e^{-h[\varepsilon N]})^{\alpha + 1}}                                                 \\
			                                                                                        & \leq
			h\frac{1 - e^{-h \alpha [\delta N]}}{1 - e^{-h \alpha}}
			+h\frac{e^{-h}(1 - e^{-h [\varepsilon N]})}{(1 - e^{-h})(1+e^{-h[\varepsilon N]})^{\alpha +1}}                                                   \\
			                                                                                        & \leq
			\frac{h}{1 - e^{-h \alpha}}
			+\frac{he^{-h}}{1 - e^{-h}}
			\leq 	\frac{1}{\alpha} + 1 = \frac{\alpha +1}{\alpha}.                                                                                           \\
		\end{aligned}
	\end{equation}
	The bound $1 - e^{-h} = \sum\limits_{k=1}^\infty \frac{(-1)^{k+1} h^k}{k!} \geq h $ was used to cancel out $h$ in the last estimation step.
	The combination of 	\eqref{eq:FCP_RLInt_err_est}, \eqref{eq:FCP_RLInt_biasederr1} and \eqref{eq:FCP_RLint_sbound_alpha} completes the proof.
\end{proof}

With all the necessary results in place, now we move on to construct the approximation to $u_{\mathrm{ih}}(t)$.
To achieve that, we apply approximations \eqref{eq:FCP_SO_cor_sinc_quad},  \eqref{eq:FCP_RLInt_appr} and discretize the remaining time-dependent integrals in a similar fashion as the Riemann--Liouville integral $J_\alpha$ (see \cref{prop:FCP_RLInt}).
The rationale for such integral discretizations will become apparent when we analyze the error below. Meanwhile, let us introduce the proposed approximation $\wt{u}_{\mathrm{ih}}^N(t)$ of the inhomogeneous solution $u_{\mathrm{ih}}(t)$ from \eqref{eq:FCP_inhom_cor_repr}:
\begin{equation}\label{eq:FCP_inhom_sol_appr}
	\begin{aligned}
		\wt{u}_{\mathrm{ih}}^N(t)  = & \wt{J}_\alpha^{N_0} \wt{S}_{\alpha,1}^N(t)f(0) + h_1 \sum\limits_{k=-N_1}^{N_1}  \mathcal{G}_\alpha^{N_2}(0, t, kh_1) \\
		                             & + \frac{h_3 h_4}{2\pi i}\sum\limits_{\ell =-N_3}^{N_3}F_{\alpha,1} (\ell h_3)
		\sum\limits_{k=-N_4}^{N_4} \mathcal{G}_\alpha^{N_5}(z(l h_3), t, kh_4),
	\end{aligned}
\end{equation}
where $F_{\alpha,1} (\xi)$, $z(\xi)$ are defined in \cref{lem:FCP_Pro_parametrized} and
\[
	\begin{aligned}
		\mathcal{G}_\alpha^N(z, t, p) & = t\psi'(p) e^{z t (1-\psi(p))}\wt{J}_\alpha^{N}  f'\left ( t\psi(p) \right ), \quad \psi(p) = \frac{e^p}{1+e^p}.
	\end{aligned}\]

\begin{theorem}\label{thm:FCP_inhom_sol_err_est}
	Let $A$ be a sectorial operator satisfying the assumptions of \cref{thm:FCP_prop_appr}.
	If the function $f(t)$ from \eqref{eq:FCP_DE} admits the analytic extension to the "eye-shaped" domain $D_d^2$, $d \in (0, \pi/2)$ and
	\begin{equation}\label{eq:FCP_rhs_reg_strip}
		f(0),f'(z) \in D(A^{\chi}), \quad \forall z \in  D_d^2,
	\end{equation}
	with some $\chi > 0$,
	then for any $\alpha \in (0, 2)$ and $t \in [0, T]$,  the approximation $\wt{u}_{\mathrm{ih}}^N(t)$ from \eqref{eq:FCP_inhom_sol_appr} converges to the inhomogeneous part $u_{\mathrm{ih}}(t)$ of the mild solution to \eqref{eq:FCP_DE}, \eqref{eq:FCP_BC}, defined by \eqref{eq:FCP_hom_inhom}.
	Moreover, for any fixed $N \in \N$, the following error bound is valid:
	\begin{equation}\label{eq:FCP_inhom_sol_err_est}
		\left \|u_{\mathrm{ih}}(t) - \wt{u}_{\mathrm{ih}}^N(t) \right \|
		\leq
		{C_{\chi,f}}  \left(\frac{t}{\alpha \chi} + \frac{1+t}{\Gamma(\alpha)}t^\alpha +  \frac{\chi + t(1 + \chi) }{\chi} t^\alpha e^{a_0 t}  \right)
		e^{-c\sqrt{\alpha\chi N}},
	\end{equation}
	{with $c=\sqrt{2\pi d}$,  provided that the values of $N_i$ and $h,h_i$ in \eqref{eq:FCP_inhom_sol_appr}  are determined by the following~formulas}
	\begin{equation}\label{eq:FCP_inhom_sol_N}
		N_1 = N_4 = \lceil \alpha\chi N \rceil,
		\quad
		N_3 = N,
		\quad
		N_0 = N_2 = N_5 = \left\lceil \frac{\alpha\chi N}{\min\left \{1, \alpha\right \}} \right\rceil ,
	\end{equation}
	\begin{equation}\label{eq:FCP_inhom_cor_h}
		h = h_i=\sqrt{\frac{2 \pi d}{\alpha \chi N}}, \quad i = 0, \ldots 5.
	\end{equation}
	Here,
	$d = \frac{\phi_\alpha - \phi_c}{2}$, $\phi_\alpha = \min\left\{\pi, \frac{\pi -\varphi_s}{\alpha}\right\}$  and $\phi_c \in \left[ \tfrac{\pi}{2}, \phi_\alpha \right)$, $a_0>0$  are given.
	The constant $C_{\chi,f}$ from~\eqref{eq:FCP_inhom_sol_err_est} is dependent on $\left\| A^{\chi}f'(z)\right \|$, ${z \in D_d^2}$ and independent of $t,N$.
\end{theorem}
\begin{proof}
	We analyze error $\left \|u_{\mathrm{ih}}(t) - \wt{u}_{\mathrm{ih}}^N(t) \right \|$ of \eqref{eq:FCP_inhom_sol_appr} in a term-by-term manner.

	The first error term is estimated via \cref{eq:FCP_RLint_NA} and \cref{thm:FCP_prop_appr}, applied in succession:
	\begin{align*}
		 & \left\| J_\alpha S_{\alpha}(t)f(0) - \wt{J}_\alpha^{N_0}\wt{S}_{\alpha}^N(t)f(0) \right\|
		\leq \frac{ t^{\alpha}(\alpha + 1)}{\Gamma(\alpha+1)} \left\| S_{\alpha}(t)f(0) - \wt{S}_{\alpha}^N(t)f(0)\right\| \\
		 & +\frac{C_0 t^\alpha}{\Gamma(\alpha)} e^{- \sqrt{2 \pi d \varepsilon N_0} }
		\leq
		C {t^{\alpha} e^{a_0 t}} e^{-c\sqrt{\alpha\chi N}}\|A^{\chi}f(0)\|
		+ \frac{C_0 t^\alpha}{\Gamma(\alpha)} e^{- c\sqrt{\varepsilon N_0} }.
	\end{align*}

	The error bound for the second term can be decomposed as
	\vspace*{-8pt}
	\[	\left\|  \intl_0^t J_\alpha f'(s)\, ds  - h_1 \sum\limits_{k=-N_1}^{N_1}  \mathcal{G}_\alpha^{N_2}(0, t, kh_1)\right\|  \leq \eta_{1} + \eta_{2},
	\]
	with $\eta_1$ being the quadrature error of the outer integral:
	\vspace*{-8pt}
	\[
		\eta_1	 = \left\|  t\intl_{-\infty}^\infty \psi'(p) J_\alpha f'(t\psi(p))\, dp  - th_1 \sum\limits_{k=-N_1}^{N_1} \psi'(k h_1)J_\alpha f'(t\psi(k h_1)) \right\|,
	\]
	stated here after the substitution $s = \psi(p)$ is performed therein, whereas $\eta_2$ is the compound error of the discretized Riemann--Liouville operators:
	\vspace*{-8pt}
	\[\begin{aligned}
			\eta_2
			 & =  th_1 \sum\limits_{k=-N_1}^{N_1} \left\|  \psi'(k h_1) \left(J_\alpha f'(t\psi(k h_1)) - \wt{J}_\alpha^{N_2} f'(t\psi(k h_1)) \right) \right\| \\
			 & \leq t \max\limits_{s \in [0,t]}\left\|J_\alpha f'(s) - \wt{J}_\alpha^{N_2} f'(s) \right\| h_1\sum\limits_{k=-N_1}^{N_1}  \psi'(k h_1).
		\end{aligned}	\]
	It is worth noting that the last series is a specific version of the one from \eqref{eq:FCP_RLInt_biasederr1}, with $\alpha = 1$.
	Thus, formula \eqref{eq:FCP_RLint_sbound_alpha} along with the bound from \cref{prop:FCP_RLInt}, warranted by the analyticity assumptions on $f'(z)$, yield
	\begin{equation*}\label{key}
		\eta_2
		\leq 2t \max\limits_{s \in [0,t]}\left\|J_\alpha f'(s) - \wt{J}_\alpha^{N_2} f'(s) \right\|
		\leq \frac{C_2}{\Gamma(\alpha)}  t^{\alpha + 1} e^{- c\sqrt{\varepsilon N_2} }.
	\end{equation*}

	Let us return to  $\eta_1$.
	The aforementioned analyticity of $f'(z)$ induces the uniform convergence of the integral for $J_\alpha f'(z)$ in formula \eqref{eq:FCP_RLInt_R} with respect to $z \in D_d^2$.
	Furthermore, for an arbitrary value of $p \in (-\infty, \infty)$, the function $z\psi(p)$ from \eqref{eq:FCP_RLInt_R} maps the convex region $D_d^2$ defined by \eqref{eq:FCP_D_d_2}, onto itself.
	By repeating the argument from the Proof of \cref{prop:FCP_RLInt}, these two facts and the relation $D_d \xrightarrow{\psi} D_d^2$ permits us to conclude that $\psi'(p) J_\alpha f'(t\psi(p))$ is analytic for $p \in D_d$.
	Due to the form of $\psi'(p)$, it is also exponentially decaying as $|p| \rightarrow \infty$ in $D_d \subseteq \oC$, with the decay order $1$.
	Hence, the error of sinc-quadrature $\eta_1$ admits the bound established in Theorem 4.2.6 from \cite{Stenger1993}:
	\vspace*{-6pt}
	\[
		\eta_1 \leq  C_1 t e^{- c \sqrt{N_1} }.
	\]

	We treat the third term from \eqref{eq:FCP_inhom_sol_appr} in a similar way as the second term, albeit this time the error decomposition is conducted after the application of \eqref{eq:FCP_SO_cor_repr} in reverse:
	\[
		\begin{array}{rl}
			 & \left\|\intl_{-\infty}^\infty \!\!\frac{F_{\alpha,1}(\xi)}{2\pi i}\intl_0^t \!\! e^{z(\xi)(t-s)}  J_\alpha f'(s) \, ds d\xi
			- t\frac{h_3 h_4}{2\pi i}\sum\limits_{\ell =-N_3}^{N_3}\!\!\!F_{\alpha,1} (\ell h_3)\! \sum\limits_{k=-N_4}^{N_4} \!\!\!\mathcal{G}_\alpha^{N_5}(z(l h_3), t, kh_4)
			\right\|                                                                                                                       \\
			 & =\left\|
			t\intl_{-\infty}^\infty \psi'(p) \left( S_{\alpha}(t - t\psi(p)) - I\right) J_\alpha f'(t\psi(p)) \, dp
			\right .                                                                                                                       \\
			 & \left .
			\hspace*{7em}
			-t\frac{h_3 h_4}{2\pi i} \suml_{k=-N_4}^{N_4} \suml_{\ell =-N_3}^{N_3}\!\! F_{\alpha,1} (\ell h_3)\mathcal{G}_\alpha^{N_5}(z(\ell h_3), t, kh_4)
			\right\|  \leq \eta_3 + \eta_4 + \eta_5 .
		\end{array}\]

	\pagebreak
	\noindent
	Here, the quantity $\eta_4$ is used to denote the quadrature error for the outer integral:
	\vspace*{-8pt}
	\begin{equation}\label{eq:FCP_inhom_sol_appr_eta3}
		\begin{aligned}
			\eta_4
			 & = \left\|
			t\intl_{-\infty}^\infty  \psi'(p) \left( S_{\alpha}(t - t\psi(p)) - I\right) J_\alpha f'(t\psi(p)) \, dp
			\right .                                                                                                                       \\
			 & \left.
			\hspace*{4em}
			- t h_4 \suml_{k=-N_4}^{N_4} \psi'(k h_4 )\left( S_{\alpha}(t - t\psi(k h_4)) - I\right) J_\alpha f'(t\psi(k h_4 )) \right \|. \\
		\end{aligned}
	\end{equation}
	The upper bound for the last integrand is determined by the properties of the norm
	$
		\left\| \left( S_{\alpha}(t - s) - I\right) J_\alpha f'(s) \right\|
		= \frac{1}{2\pi} \left\| \int_{-\infty}^\infty e^{z(\xi)(t-s)} F_{\alpha,1}(\xi) J_\alpha f'(s) \, d \xi\right\|,
	$
	which can be estimated using inequality \eqref{eq:FCP_SO_exp_cor_repr_norm_est_der}.
	Indeed, setting $x_1 = J_\alpha f'(s)$  in \eqref{eq:FCP_SO_cor_sinc_quad} reveals that the above integrand equals to the expression for $\cF_{\alpha,1}(t-s, \xi)$.
	As such, it admits the estimate
	\[
		\begin{aligned}
			\left\| e^{z(\xi)(t-s)} F_{\alpha,1}(\xi)J_\alpha f'(s)\right\|
			\leq
			C_{\alpha,1}(\chi, 0)
			e^{\Re{(z(\xi)(t-s))} - \alpha\chi|\xi|}\|A^{\chi} J_\alpha f'(s)\|                                                   \\
			\leq
			\frac{C_{\alpha,1}(\chi, 0)}{\Gamma(\alpha)}
			e^{\Re{(z(\xi)(t-s))} - \alpha\chi|\xi|}
			\int\limits_{-\infty}^{\infty} \frac{s^\alpha e^p}{(1+e^p)^{\alpha + 1}}\left\| A^{\chi} f'(s\psi(p))\right \|  \, dp \\
			\leq
			\frac{C_{\alpha,1}(\chi, 0) s^\alpha}{\Gamma(\alpha+1)}
			e^{\Re{(z(\xi)(t-s))} - \alpha\chi|\xi|}
			\sup\limits_{p \in \R}\left\| A^{\chi}f'(s\psi(p))\right \|.
		\end{aligned}\]
	Here, we used the relation
	\begin{equation}\label{eq:FCP_inhom_sol_appr_proof1}
		\int\limits_{-\infty}^{\infty} \frac{s^\alpha e^p}{(1+e^p)^{\alpha + 1}}
		= \int\limits_0^s (s-p)^{\alpha-1}dp
		= \frac{s^\alpha}{\alpha} ,
	\end{equation}
	stemming from the equivalence of definitions \eqref{eq:FCP_RLInt_R} and \eqref{eq:FCP_RLInt}.
	The previous chain of estimates leads us to the following bound:
	\vspace*{-6pt}
	\begin{equation}\label{eq:FCP_inhom_sol_appr_eta3_1}
		\begin{aligned}
			\left\| \left( S_{\alpha}(t - s) - I\right) J_\alpha f'(s) \right\|
			\leq \frac{1}{2\pi} \intl_{-\infty}^\infty \left\| e^{z(\xi)(t-s)}  F_{\alpha,1}(\xi) J_\alpha f'(s)\right\| \, d \xi &   & \\
			\leq
			\frac{C_{\alpha,1}(\chi, 0) s^\alpha }{2\pi \Gamma(\alpha+1)}
			\sup\limits_{p \in \R}\left\| A^{\chi}f'(s\psi(p))\right \|
			\intl_{-\infty}^\infty e^{\Re{(z(\xi)(t-s))}  - \alpha\chi|\xi|} \, d \xi                                             & . & \\
		\end{aligned}
	\end{equation}

	The norm $\left\| \left( S_{\alpha}(t - s) - I\right) J_\alpha f'(s) \right\|
	$ remains bounded as long as the last integral converges.
	Since $s = t\psi(p)$ in the expression for $\eta_4$, this convergence requirement translates into the inequality $\Re{(z(\xi)(t-t\psi(p)))} \leq 0$, which needs to be valid as $|\xi|,|p| \rightarrow \infty$.
	Let $p = w + i \nu \in D_{d'}$; then, using the {definition of $a_I$, $b_I$ from \eqref{eq:FCP_hyp_cont_par_final} along with the identity $\Arg{\left( 1 - \psi(w+i \nu)\right)}=-\arctan{\frac{e^w \sin{\nu}}{1 + e^w\cos{\nu}}}$,} we rewrite the inequality in terms of the complex number arguments
	\vspace*{-6pt}
	\[
		\begin{aligned}
			\frac{\pi}{2} & \geq \lim\limits_{\substack{\xi \to \infty \\ w \to \infty}}
			\Arg{(z(\xi)(t-t\psi(w + i \nu)))}%
			= \lim\limits_{\xi \to \infty}\Arg(z(\xi)) +
			\lim\limits_{w  \to \infty}
			\Arg(1-\psi(w + i \nu))                                    \\
			              & = \frac{\phi_s}{2} + \frac{\pi}{4}
			- \lim\limits_{w \to \infty}
			\arctan{\frac{e^w \sin{\nu}}{1 + e^w\cos{\nu}}}%
			= \frac{\phi_s}{2} + \frac{\pi}{4} - \nu .
		\end{aligned}
	\]

	\pagebreak
	\noindent
	Therefore, the integrand from \eqref{eq:FCP_inhom_sol_appr_eta3} is analytic in $D_{d'}$,
	with the value of $d' =  \frac{\phi_s}{2} - \frac{\pi}{4}$, which is equal to $d$ from \eqref{eq:FCP_hyp_cont_par_final}.
	Furthermore, the integrand's norm is exponentially decaying as a function of $p \in D_{d}$, as a consequence of \eqref{eq:FCP_inhom_sol_appr_eta3_1}, the boundedness of $\sup\limits_{z \in D_d^2}\left\| A^{\chi}f'(z)\right \|$ and the convergence of the integral, established before.
	Hence, similarly to $\eta_1$, the bound \mbox{from Theorem 4.2.6 of \cite{Stenger1993} } yields
	\vspace*{-6pt}
	\[
		\eta_4 \leq  \frac{C_4 t}{\alpha \chi} e^{- c \sqrt{N_4} }.
	\]

	Next, we use \eqref{eq:FCP_SO_cor_err_est} in conjunction with \eqref{eq:FCP_RLint_sbound_alpha}  and \eqref{eq:FCP_inhom_sol_appr_proof1} to estimate the propagator approximation error $\eta_3$
	\vspace*{-6pt}
	\[
		\begin{aligned}
			\eta_3
			 & = t h_4 \left\|
			\suml_{k=-N_4}^{N_4} \psi'(k h_4 )\left( S_{\alpha}(t - t\psi(k h_4)) - I\right) J_\alpha f'(t\psi(k h_4 )) \right.                        \\
			 & \left .
			\hspace*{1.2em}
			-\frac{h_3 }{2\pi i} \suml_{k=-N_4}^{N_4} \psi'(k h_4 )\suml_{\ell =-N_3}^{N_3} e^{t z(\ell h_3)  (1-\psi(k h_4))}F_{\alpha,1} (\ell h_3)J_\alpha f'(t\psi(k h_4 ))
			\right\|                                                                                                                                   \\
			 & \leq t h_4 \suml_{k=-N_4}^{N_4} |\psi'(k h_4 )|
			\left\| \frac{1}{2\pi i} \intl_{-\infty}^\infty e^{t z(\xi) (1-\psi(k h_4))}F_{\alpha,1}(\xi) J_\alpha f'(t\psi(k h_4 )) \right.           \\
			 & \left.
			\hspace*{8em}
			- \frac{h_3}{2\pi i}\suml_{\ell =-N_3}^{N_3} e^{t z(\ell h_3) (1-\psi(k h_4))}F_{\alpha,1} (\ell h_3)J_\alpha f'(t\psi(k h_4 ))
			\right\|                                                                                                                                   \\
			 & \leq  {C'_3} {te^{a_0 t}}e^{-c\sqrt{\alpha \chi N_3}} h_4 \suml_{k=-N_4}^{N_4} |\psi'(k h_4 )| 	 \|A^{\chi} J_\alpha f'(t\psi(k h_4 )\| \\
			 & \leq \frac{2{C'_3} {t^{\alpha+1} e^{a_0 t}}}{\Gamma(\alpha)}   e^{-c\sqrt{\alpha \chi N_3}}
			\sup\limits_{p \in \R}\left\| A^{\chi} f'(t\psi(p))\right \|
			\int\limits_{-\infty}^{\infty} \frac{ e^p}{(1+e^p)^{\alpha + 1}}\, dp                                                                      \\
			 & \leq \frac{2{C'_3} {t^{\alpha+1} e^{a_0 t}}}{\Gamma(\alpha + 1)}   e^{-c\sqrt{\alpha \chi N_3}}
			\sup\limits_{p \in [0,t]}\left\| A^{\chi} f'(p)\right \|
			\leq C_3 t^{\alpha+1} e^{a_0 t} e^{-c\sqrt{\alpha \chi N_3}}
			\sup\limits_{p \in [0,t]}\left\| A^{\chi} f'(p)\right \| .
		\end{aligned}
	\]

	The remaining summand $\eta_5$ represent the effect of discretized Riemann--Liouville operators on the error of the third approximation term in \eqref{eq:FCP_inhom_sol_appr}.
	We estimate it as
	\vspace*{-6pt}
	\[\begin{aligned}
			\eta_5 & =
			\frac{t h_3 h_4}{2\pi}
			\left\| \suml_{k=-N_4}^{N_4}\!\!\! \psi'(k h_4 )
			\suml_{\ell =-N_3}^{N_3}\!\!\! e^{t z(\ell h_3) (1-\psi(k h_4))}
			F_{\alpha,1} (\ell h_3)J_\alpha f'(t\psi(k h_4 ))
			\right.       \\
			       &
			\hspace*{4em}
			\left.
			- \suml_{k=-N_4}^{N_4} \!\!\!\psi'(k h_4 )
			\suml_{\ell =-N_3}^{N_3}\!\!\! e^{t z(\ell h_3) (1-\psi(k h_4))}
			F_{\alpha,1} (\ell h_3)\wt{J}_\alpha^{N_5}  f'\left ( t\psi(k h_4) \right )
			\right\|      \\
			       & \leq
			\frac{t h_3 h_4}{2\pi} \!\!\!
			\suml_{k=-N_4}^{N_4}\!\!\! |\psi'(k h_4 )|
			\suml_{\ell =-N_3}^{N_3} \!
			\left| e^{t z(\ell h_3) (1-\psi(k h_4))}\right|
			\\&
			\hspace*{5em}
			\times
			\left\|F_{\alpha,1} (\ell h_3)
			\left(
			J_\alpha f'(t\psi(k h_4)) - \wt{J}_\alpha^{N_5}  f'\left ( t\psi(k h_4) \right )
			\right)
			\right\|      \\
		\end{aligned}\]

	\pagebreak
	\noindent
	\[\begin{aligned}
			 & \leq
			\frac{(1+M)K b_I2^{\alpha\chi}h_3 h_4}
			{2\pi (a_I-a_0)}
			te^{t \Re{z(0)}}
			\suml_{k=-N_4}^{N_4}\!\!\! |\psi'(k h_4 )|
			\suml_{\ell =-N_3}^{N_3}
			\frac{e^{- \alpha\chi |\ell h_3|}}
			{r^{\chi}(\ell h_3,0)} \\
			 &
			\hspace*{5em}
			\times\left\|
			A^{\chi}
			\left( J_\alpha f'(t\psi(k h_4 )) - \wt{J}_\alpha^{N_5}  f'\left ( t\psi(k h_4) \right ) \right)
			\right \|.
		\end{aligned}\]
	Assumption \eqref{eq:FCP_rhs_reg_strip} enables us to estimate the last norm via \cref{prop:FCP_RLInt}:
	\begin{multline*}
		\left\|
		A^{\chi} \!\!
		\left( J_\alpha f'(t\psi(k h_4 )) - \wt{J}_\alpha^{N_5}  f'\left ( t\psi(k h_4) \right ) \right)
		\right \| \\
		\leq
		\left\|
		J_\alpha A^{\chi}\! f'(t\psi(k h_4 )) - \wt{J}_\alpha^{N_5}  A^{\chi} f'\left ( t\psi(k h_4)  \right)
		\right \|\\
		\leq
		\frac{C(\alpha, \chi) (t\psi(k h_4))^\alpha}{\Gamma(\alpha)} e^{- \sqrt{2 \pi d \varepsilon N_5} }
		\leq
		\frac{C(\alpha, \chi) t^\alpha}{\Gamma(\alpha)} e^{- \sqrt{2 \pi d \varepsilon N_5} }.
	\end{multline*}
	This decouples the inner and outer series in the above estimate for $\eta_5$.
	Thus,
	\[\begin{aligned}
			\frac{h_3}{2}\suml_{\ell =-N_3}^{N_3}
			\!\!\frac{e^{- \alpha\chi |k h_3|}}
			{r^{\chi}(k h_3,0)}
			 & \leq
			\frac{h_3}{2r_m^{\chi}} \left(
			\frac{1 - e^{- \alpha\chi N_3 h_3}}{1 - e^{- \alpha\chi h_3}}
			+ \frac{e^{- \alpha\chi h_3}(1 - e^{- \alpha\chi N_3 h_3})}{1 - e^{- \alpha\chi h_3}}
			\right) \\
			 & \leq
			\frac{h_3}{2r_m^{\chi}} \frac{1 + e^{- \alpha\chi h_3}}{1 - e^{- \alpha\chi h_3}}
			\leq \frac{1}{\alpha \chi r_m^{\chi}} ,
		\end{aligned}\]
	with $r_m = \inf\limits_{p \in \R} r_0(p)$ being strictly greater than zero, due to \eqref{eq:FCP_r}.
	By combining two previously obtained bounds with \eqref{eq:FCP_RLint_sbound_alpha}, we arrive at
	\[\begin{aligned}
			\eta_5
			 & \leq
			C'_5  \frac{C(\alpha, \chi) }
			{\chi \Gamma(\alpha+1)}  t^{\alpha + 1}e^{t (a_0 - a_I)}
			e^{- c \sqrt{\varepsilon N_5} }
			\leq
			C_5   \frac{t^{\alpha + 1}}{\chi}e^{t (a_0 - a_I)}
			e^{- c \sqrt{\varepsilon N_5} }.
		\end{aligned}\]
	The constant $C'_5 = \frac{(1+M)K b_I 2^{\alpha\chi + 1} }
		{\pi (a_I - a_0) } $, here,  is independent of $N_5$.

	The bounds derived for the quantities $\eta_i$, $i = 1, \ldots, 5$ show that they all are exponentially decaying as $N_i \to \infty$.
	We make these error bounds asymptotically equal to the error of the first term from \eqref{eq:FCP_inhom_sol_appr}, that is decaying on the order of $e^{-c\sqrt{\alpha\chi N}}$, provided that $\varepsilon N_0 = \alpha\chi N$ in the error estimate from the beginning of the proof.
	The resulting equations for $N_i$ are as~follows:
	\[
		N_1 = \alpha\chi N,
		\quad
		\varepsilon N_2 = \alpha\chi N,
		\quad
		\alpha\chi N_3 = \alpha\chi N,
		\quad
		N_4 = \alpha\chi N,
		\quad
		\varepsilon N_5 = \alpha\chi N,
	\]
	where $\varepsilon = \min\left \{1, \alpha\right \}$, as per \cref{prop:FCP_RLInt}.
	The solution of these equations gives us
	\eqref{eq:FCP_inhom_sol_N}.
	By collecting the derived error bounds for the terms of $\wt{u}_{\mathrm{ih}}^N(t)$, we end up with
	\begin{multline*}
		\left \|u_{\mathrm{ih}}(t) - \wt{u}_{\mathrm{ih}}^N(t) \right \| \leq
		\left(
		C t^{\alpha} e^{a_0 t} \|A^{\chi}f(0)\|
		+ \frac{C_0 t^\alpha}{\Gamma(\alpha)}
		+ C_1 t + \frac{C_4}{\alpha \chi}  t
		\right.\\
		\left.
		+ \frac{C_2}{\Gamma(\alpha)} t^{\alpha + 1}
		+ C_3 t^{\alpha+1} e^{a_0 t}
		\sup\limits_{p \in [0,t]}\left\| A^{\chi} f'(p)\right \|
		+C_5   \frac{t^{\alpha + 1}}{\chi}e^{a_0 t}
		\right) e^{-c\sqrt{\alpha\chi N}}.
	\end{multline*}
	This bound is reduced to \eqref{eq:FCP_inhom_sol_err_est}  by absorbing the individual constants into $C_{\chi, f}$, while retaining the asymptotic behavior with respect to $\alpha$, $\chi$ and $t$.
	The derived bound also proves the convergence of approximation \eqref{eq:FCP_inhom_sol_appr} to \eqref{eq:FCP_inhom_cor_repr} and, therefore, to the original definition for the inhomogeneous part of the solution given by \eqref{eq:FCP_hom_inhom}.
\end{proof}

\Cref{thm:FCP_inhom_sol_err_est} demonstrates that the proposed numerical method to approximate $u_{\mathrm{ih}}(t)$ inherits essential properties of the numerical method for propagator approximation,
it is based upon.
Firstly, the constructed approximation $\wt{u}_{\mathrm{ih}}^N(t)$ is exponentially convergent on the whole interval $t \in [0, T]$.
Secondly, bound \eqref{eq:FCP_inhom_sol_err_est} exhibits, similar to, \eqref{eq:FCP_SO_cor_err_est} dependence on the fractional order $\alpha$ and the argument smoothness parameter $\chi$.
Thirdly, just like~\eqref{eq:FCP_inhom_cor_repr}, formula~\eqref{eq:FCP_inhom_sol_appr} permits for an independent evaluation of resolvents $R(z^\alpha, -A)$ for the different values $z \in \Gamma_I$.
Moreover, the presence of factor $t^\alpha$ in \eqref{eq:FCP_inhom_sol_err_est} guaranties that the approximation $\wt{u}_{\mathrm{ih}}^N(t)$ matches the asymptotic behavior of the inhomogeneous part ${u}_{\mathrm{ih}}(t)$, when \mbox{$t \to 0+$ \cite{SytnykWohlmuth2023}.}
\begin{algorithm}[h!tb]
	\renewcommand{\algorithmicrequire}{\textbf{INPUT:}}
	\renewcommand{\algorithmicensure}{\textbf{OUTPUT:}}
	\begin{algorithmic}[1]
		\small
		\REQUIRE{\hspace{2em}  $\alpha, f(t)$, $t_k$, $\varphi_s$, $N, \chi$}
		\ENSURE{\hspace{0.8em}  $\left\{\wt{u}_\mathrm{ih}^N(t_k)\right\}$}
		\STATE{Calculate $a_I, b_I$ and $N_i$, $h,h_i$ by \eqref{eq:FCP_hyp_cont_par_final} and \eqref{eq:FCP_inhom_sol_N}, \eqref{eq:FCP_inhom_cor_h} }
		\FOR{$m=-N$ \TO $N$}
		\STATE{Solve $(z(m h_1)^\alpha I+A)v = f(0)$}
		\STATE{$F_{1,m}:=  z(m h_1)^{\alpha-1}v - \frac{1}{z(m h_1)}f(0)$}
		\ENDFOR
		\STATE{$M_1 := \left \lceil N_0\min\left \{1, \alpha\right \} \right \rceil$; $M_2 := \left \lceil N_0 \min\left \{\frac{1}{\alpha}, 1\right \} \right \rceil $}
		\FOR{each $t_k$}
		\STATE{
		$\wt{u}_\mathrm{ih}^N(t_k) := \frac{t^{\alpha}h_0}{\Gamma(\alpha)}\!\! \suml_{ \ell=-M_1}^{M_2}\!\! \frac{e^{\ell h_0}}{(1+e^{\ell h_0})^{\alpha + 1}}\!\! \left(f(0) + \frac{h}{2 \pi i}\!\!\suml_{m=-N}^{N}\!\!\!\! z'(m h) {e^{t_k \psi(\ell h_0) z(m h)}} F_{1,m}\right)$
		}
		\ENDFOR
		\FOR{each $t_k$}
		\STATE{$M_1 := \left \lceil N_2\min\left \{1, \alpha\right \} \right \rceil$; $M_2 := \left \lceil N_2 \min\left \{\frac{1}{\alpha}, 1\right \} \right \rceil $}
		\STATE{%
		$\wt{u}_\mathrm{ih}^N(t_k) := \wt{u}_\mathrm{ih}^N(t_k)\!
			+ \!\frac{ h_1 h_2}{\Gamma(\alpha)}t_k^{\alpha + 1} \!\!\!\suml_{\ell=-N_1}^{N_1}\!\!\!\! \psi'(\ell h_1) \psi^\alpha(\ell h_1)\!\!\suml_{m=-M_1}^{M_2}\!\!\!\frac{e^{m h_2}f'\left ( t_k \psi(\ell h_1)\psi(m h_2) \right )}{(1+e^{m h_2})^{\alpha + 1}}$
		}
		\ENDFOR
		\FOR{each $t_k$}
		\STATE{$M_1 := \left \lceil N_5\min\left \{1, \alpha\right \} \right \rceil$; $M_2 := \left \lceil N_5 \min\left \{\frac{1}{\alpha}, 1\right \} \right \rceil $}
		\FOR{$m=-N_3$ \TO $N_3$}
		\STATE{%
		$f_m :=\frac{ h_4 h_5}{\Gamma(\alpha)}t_k^{\alpha + 1} \!\!\! \suml_{\ell=-N_4}^{N_4}\!\!\! \psi'(\ell h_4) \psi^\alpha(\ell h_4)e^{t_k z(m h_3) (1-\psi(\ell h_4))} \!\!\! \suml_{j=-M_1}^{M_2}\!\!\!\!\frac{e^{j h_5}f'\left ( t_k\psi(\ell h_4)\psi(j h_5) \right )}{(1+e^{j h_5})^{\alpha + 1}}$
		}
		\STATE{Solve $(z(m h_3)^\alpha I+A)v = f_m$}
		\STATE{$F_{1,m}:=  z(m h_3)^{\alpha-1}v - \frac{1}{z(m h_3)}f_m$}
		\ENDFOR
		\STATE{%
			$\wt{u}_\mathrm{ih}^N(t_k) := \wt{u}_\mathrm{ih}^N(t_k) + \frac{h_3}{2\pi i}\sum\limits_{m =-N_3}^{N_3} z'(m h_3) F_{1,m}$}

		\ENDFOR
		\RETURN{$\{\wt{u}_\mathrm{ih}^N(t_k)\}$}
	\end{algorithmic}
	\caption{Algorithm for computing the inhomogeneous part approximation $\wt{u}_\mathrm{ih}^N(t)$.%
		\label{alg:FCP_inhom_sol_appr}%
	}
\end{algorithm}
\FloatBarrier
\begin{remark}
	In order to minimize the computational cost required for the evaluation of the approximation $\wt{u}_{\mathrm{ih}}^N(t)$, one can dynamically adjust the discretization parameters $N_2, N_5$ of the approximate Riemann--Liouville integrals from \eqref{eq:FCP_inhom_sol_appr} using formula \eqref{eq:FCP_RLInt_appr_N_adj} with $t_0$ equal to $t\psi(kN_1)$ and $t\psi(kN_4)$,~accordingly.
\end{remark}

The following example is aimed to numerically verify the quality of approximation \eqref{eq:FCP_inhom_sol_appr} to the inhomogenous part of solution given by \eqref{eq:FCP_inhom_cor_repr} or \eqref{eq:FCP_hom_inhom}.
\begin{example}\label{ex:FCP_ex2_inhom_R_eigenfunction}
	Let $A$ be defined as in \cref{ex:FCP_ex1_hom_R_eigenfunction}.
	Furthermore, let $f$ from \eqref{eq:FCP_DE} be a product of the eigenfunction of $A$ and the polynomial:
	\begin{equation}\label{eq:FCP_ex2_f}
		f(t) =  \suml_{i = 0}^m c_i t^i \sin{\frac{\pi k_i x}{L}} ,
	\end{equation}
	where $m$ and $c_i$, $i={0, \ldots, m}$ are given.
	For such $f$, we have
	\[
		J_\alpha f'(t) =  \suml_{i = 0}^m  c'_i t^{\alpha + i - 1} \sin{\frac{\pi k_i x}{L}}, \quad c'_i  = \frac{  \Gamma(i+1)}{\Gamma(\alpha+i)} c_i.
	\]
	The inhomogeneous part of the solution to \eqref{eq:FCP_DE}, \eqref{eq:FCP_BC} takes the form
	\[
		\begin{aligned}
			u(t,x)
			= & c'_0 \sin{\frac{\pi k_0 x}{L}} \intl_0^t \!\! \frac{E_{\alpha,1}(-s^\alpha \lambda(k) ) }{(t-s)^{1-\alpha}} \, ds
			\\&
			+  \suml_{i = 1}^m  c'_i \sin{\frac{\pi k_i x}{L}}  \intl_0^t \!\! E_{\alpha,1}(-(t - s)^\alpha \lambda(k) )  s^{\alpha + i - 1} \, ds, \\
		\end{aligned}
	\]
	which is derived using the fractional propagator representation from \cref{ex:FCP_ex1_hom_R_eigenfunction}.
	The integrals from the above formula for $u(t)$ cannot be evaluated explicitly for arbitrary $\alpha$.
	Thus, we rely upon the numerical evaluation of  $u(t)$ via exponentially convergent quadrature formulas \eqref{eq:FCP_RLInt_R} and \cite[Theorem 4.2.6]{Stenger1993} with discretization parameters $N_J$ and $N_I$, correspondingly.
	The analysis conducted in the Proof of \cref{prop:FCP_RLInt} suggests to set $N_J = N_I/\min\{1,\alpha\}$.
	This leaves us with only one discretization parameter $N_I$, which has to be chosen large enough for the error of the approximated $u(t) $ to be negligible with respect to the error of the numerical solution $\wt{u}_{\mathrm{ih}}^N(t)$.
	The latter one is obtained by \cref{alg:FCP_inhom_sol_appr} for the data specified in \eqref{eq:FCP_ex1_hom_A}, \eqref{eq:FCP_ex1_hom_IV} and \eqref{eq:FCP_ex2_f}, using the explicit resolvent evaluation formula from \cref{ex:FCP_ex1_hom_R_eigenfunction} and the software implementation mentioned there.
	We fix $m = 1$, $c_0 = 1$, $c_1 = 1$, $k_0  = 1$, $k_1 = 4$, $L = 1$  in \eqref{eq:FCP_ex2_f} and, after conducting several numerical experiments, settle with $N_I = 256$.
	The resulting behavior of $u(t)$ is visualized in \cref{fig:FCP_Ex2_ex_vs_appr_inhomsol}.
	\begin{figure}[h!tb]
		\vspace{-9pt}
		\newdimen\lentwosubfig
		\lentwosubfig=0.48\linewidth
		\iflatexml
			\includegraphics[width=0.98\lentwosubfig, viewport=80 60 1100 1000, clip=true]%
			{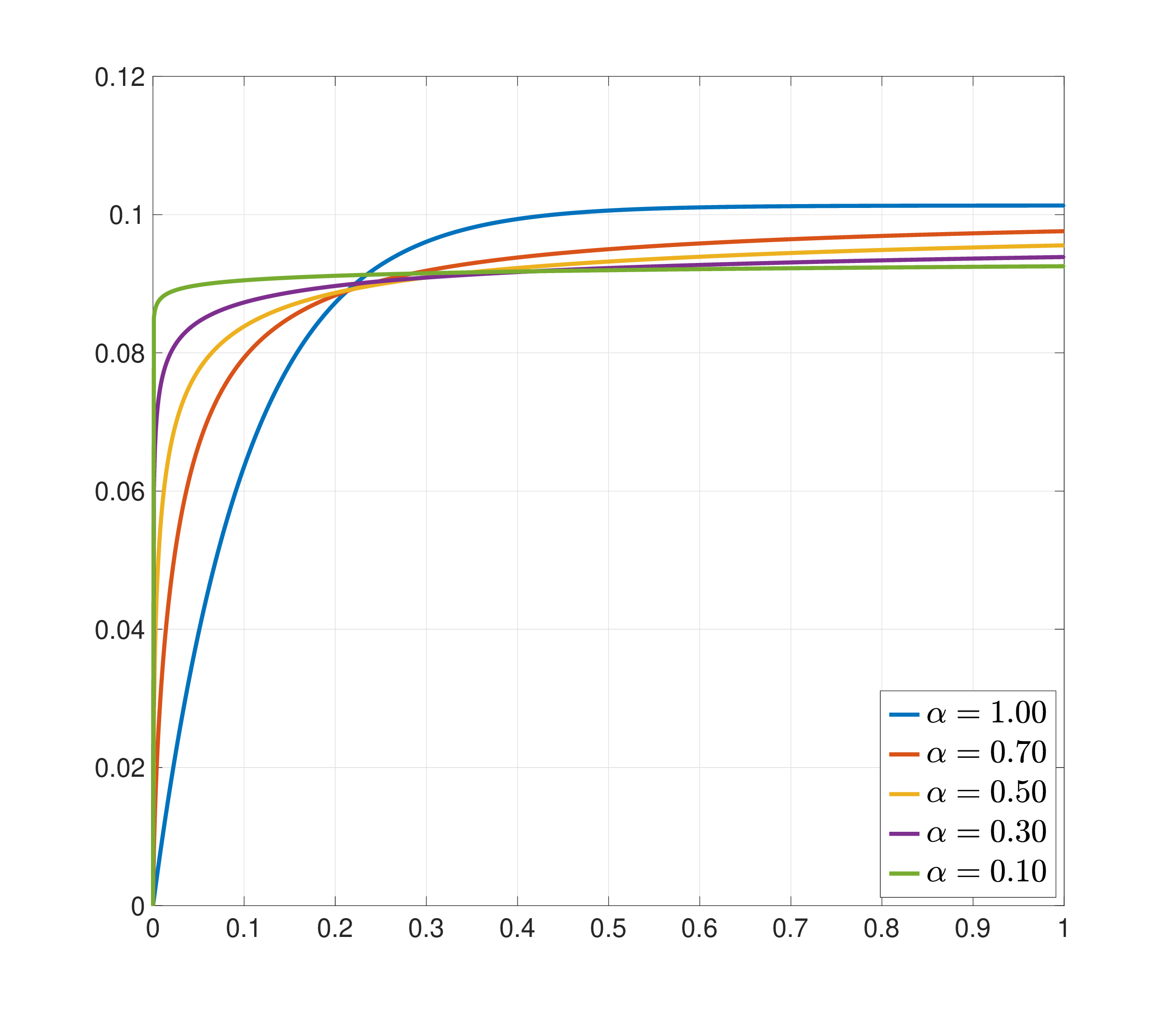}
			\includegraphics[width=0.98\lentwosubfig, viewport=80 60 1100 1000, clip=true]%
			{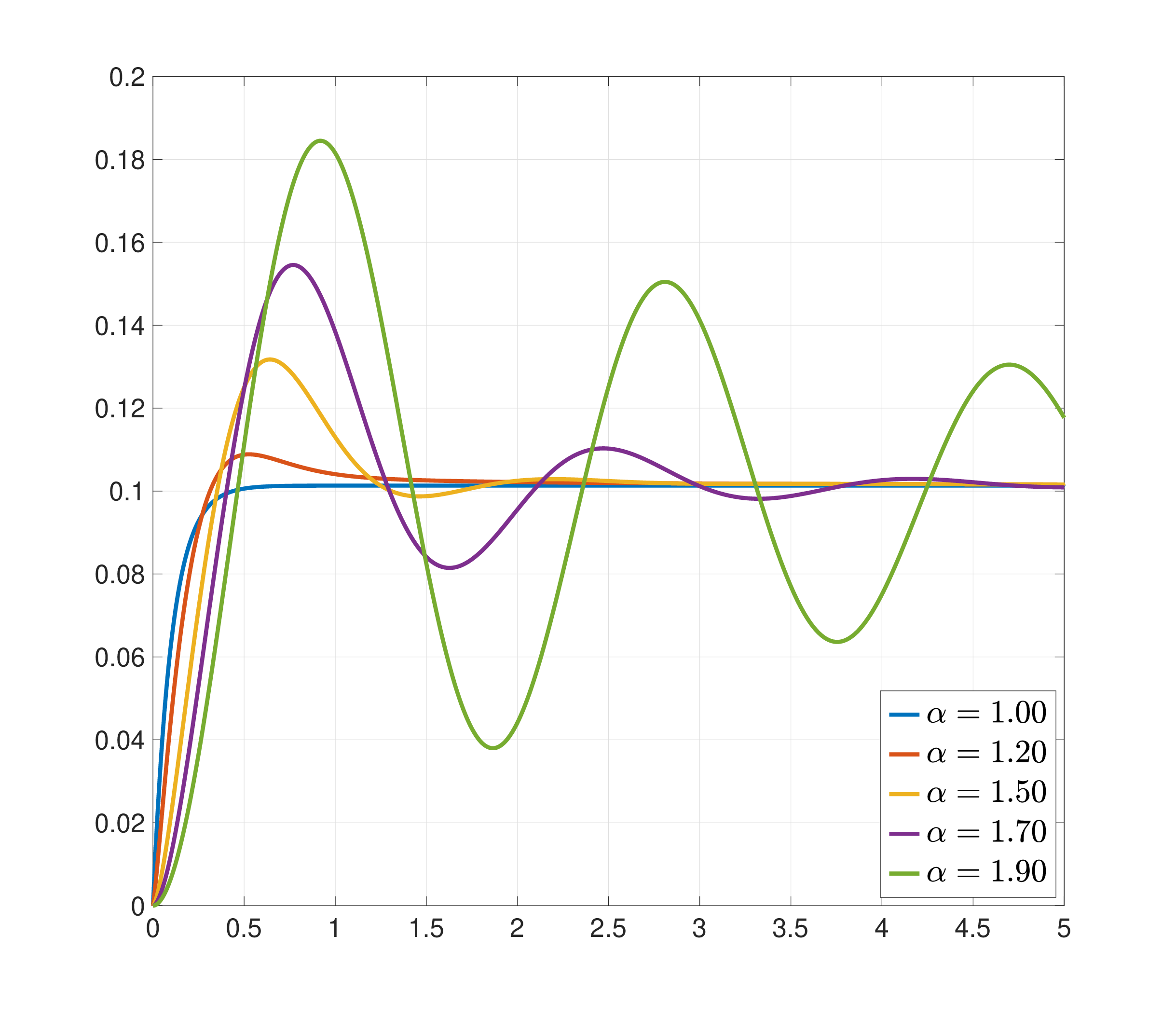}
		\else
			\hspace*{0.4em}
			\begin{overpic}[width=0.98\lentwosubfig, viewport=80 60 1100 1000, clip=true]%
				{Ex2/FCP_inhom_sol_ex_vs_t_alpha_0.1_0.3_0.5_0.7_1_nf0_1_ndf_4}
				\put(55,-1){\small $t$}
				\put(-2,44){\small \rotatebox{90}{$u$}}
				\put(8,82){\scriptsize ({\bf a})}
			\end{overpic}\hfil
			\begin{overpic}[width=0.98\lentwosubfig, viewport=80 60 1100 1000, clip=true]%
				{Ex2/FCP_inhom_sol_ex_vs_t_alpha_1_1.2_1.5_1.7_1.9_nf0_1_ndf_4}
				\put(55,-1){\small $t$}
				\put(-2,44){\small \rotatebox{90}{$u$}}
				\put(8,82){\scriptsize ({\bf b})}
			\end{overpic}
		\fi
		\caption[Example 2: Exact and approximate inhomogeneous solution]{Exact solution $u(t,0.5)$ of problem \eqref{eq:FCP_DE}, \eqref{eq:FCP_BC} with $f(t) = \sin{\pi x} + t \sin{4\pi x}$, $u_0 = u_1 = 0$  and $A$, defined by \eqref {eq:FCP_ex1_hom_A} with $a = 1$, $N_I = 256$:  {\bf (a)} the case $\alpha = 0.1, 0.3, 0.5, 0.7, 1$; {\bf (b)} the case $\alpha = 1, 1.2, 1.5, 1.7, 1.9$.}%
		\label{fig:FCP_Ex2_ex_vs_appr_inhomsol}
	\end{figure}
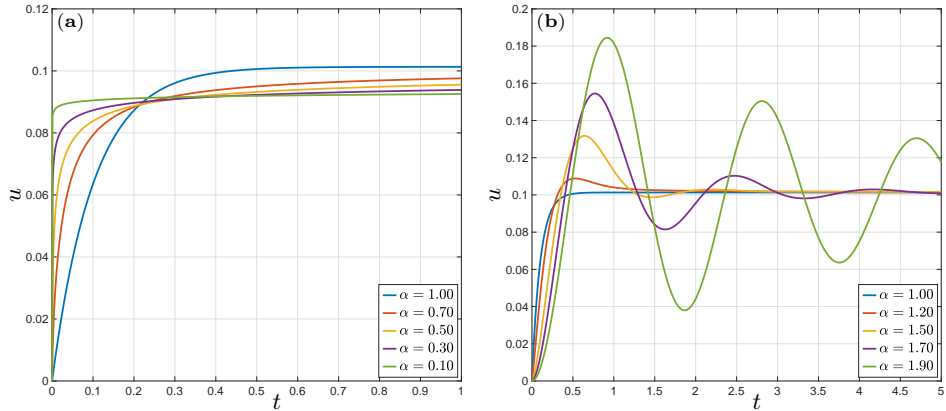

	To measure the error of numerical solution $\wt{u}_{\mathrm{ih}}^N(t)$, we define
	\[
		\cE_{\mathrm{ih}}(t,x) = \left| u(t,x) - \wt{u}_{\mathrm{ih}}^N(t,x) \right|,
		\quad \mathrm{err}_{\mathrm{ih}} =  \sup\limits_{t \in [0, T]} 	\left\|  \cE_{\mathrm{ih}}(t)\right\|_\infty.
	\]
	and plot the values of $\cE_{\mathrm{ih}}(t,x)$ as a function of $t$ for fixed $x = 0.5$ and different values of $\alpha$, $N$ in \cref{fig:FCP_Ex2_ex_err_vs_t}.
	It has exactly the same structure as \cref{fig:FCP_Ex1_ex_err_vs_t} for the homogeneous case.
	The top row of plots in \cref{fig:FCP_Ex2_ex_err_vs_t} correspond to the case when $\alpha \leq 1$.
	Taking into account the monotonous behavior of exact solution $u(t)$ for such $\alpha$, here we consider $t \in [0,1]$.
	For the bottom row of plots  with $\alpha \geq 1$, we choose the larger time horizon $T = 5$.
	According to \cref{fig:FCP_Ex2_ex_err_vs_t}~(\textbf{b}), with such $T$ the numerical simulation will cover at least two full solution oscillation periods.
	In both cases, we observe that the maximum of the error $\alpha$-wise goes down when $N$ increases from one subplot to the next in line.
	\begin{figure}[h!tb]
		\newdimen\lenthreesubfig
		\lenthreesubfig=0.33\linewidth
		\iflatexml
			\includegraphics[width=0.98\lenthreesubfig, viewport=00 50 1010 960, clip=true]%
			{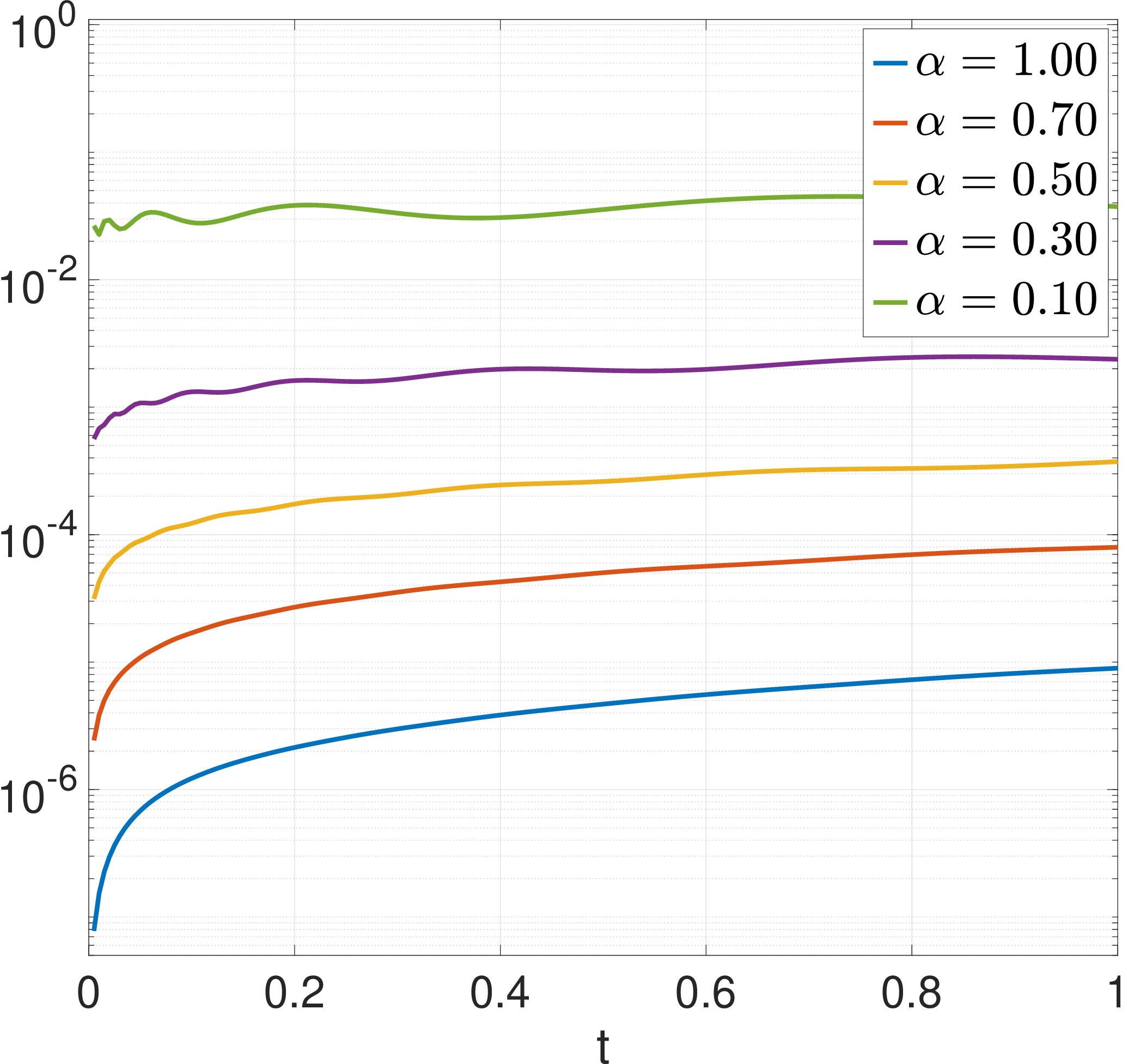}
			\includegraphics[width=0.98\lenthreesubfig, viewport=00 50 1010 960, clip=true]%
			{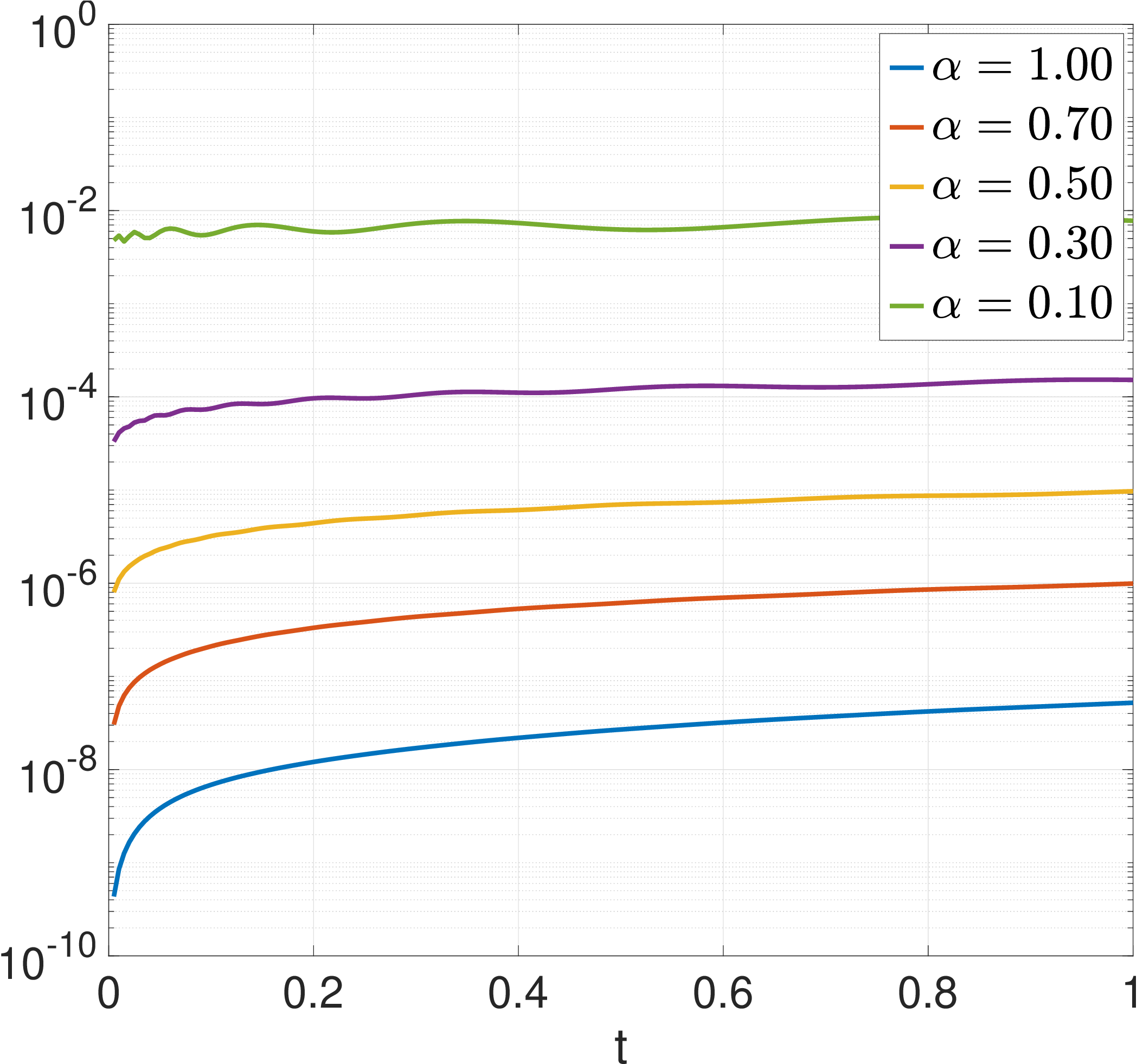}
			\includegraphics[width=0.98\lenthreesubfig, viewport=00 50 1010 960, clip=true]%
			{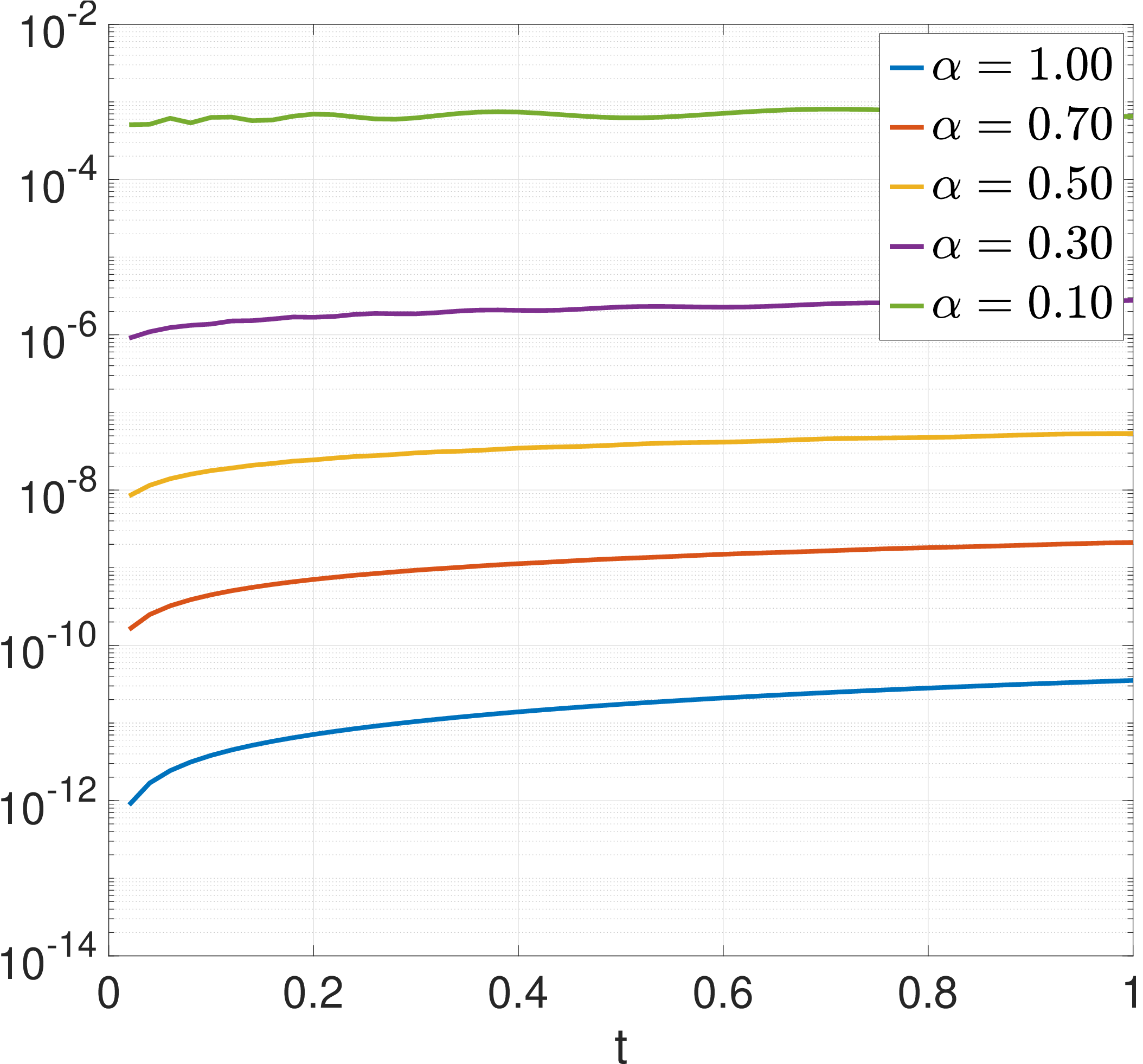}
		\else
			\begin{overpic}[width=0.98\lenthreesubfig, viewport=00 50 1010 960, clip=true]%
				{Ex2/3x2/FCP_ihnom_logerr_vs_t_alpha_0.1_0.3_0.5_0.7_1_nf0_1_ndf_4_N_32}
				\put(39,78){\colorboxo{white}{\small $N = 32$} }
				\put(-4,28){\scriptsize \rotatebox{90}{$\cE_{\mathrm{ih}}$}}
				\put(52,-3){\small $t$}
				\put(9,81){\scriptsize ({\bf a})}
			\end{overpic}
			\hspace*{-0.2em}
			\begin{overpic}[width=0.98\lenthreesubfig, viewport=00 50 1010 960, clip=true]%
				{Ex2/3x2/FCP_ihnom_logerr_vs_t_alpha_0.1_0.3_0.5_0.7_1_nf0_1_ndf_4_N_64}
				\put(41,78){\colorboxo{white}{\small $N = 64$} }
				\put(-2,41){\scriptsize \rotatebox{90}{$\cE_{\mathrm{ih}}$}}
				\put(52,-3){\small $t$}
				\put(11,81){\scriptsize ({\bf b})}
			\end{overpic}
			\hspace*{-0.1em}
			\begin{overpic}[width=0.98\lenthreesubfig, viewport=00 50 1010 960, clip=true]%
				{Ex2/3x2/FCP_ihnom_logerr_vs_t_alpha_0.1_0.3_0.5_0.7_1_nf0_1_ndf_4_N_128}
				\put(40,78){\colorboxo{white}{\small $N = 128$}}
				\put(-2,35){\scriptsize \rotatebox{90}{$\cE_{\mathrm{ih}}$}}
				\put(52,-3){\small $t$}
				\put(11,81){\scriptsize ({\bf c})}
			\end{overpic}%
		\fi
		\\[6pt]
		\iflatexml
			\includegraphics[width=0.98\lenthreesubfig, viewport=00 50 1010 960, clip=true]%
			{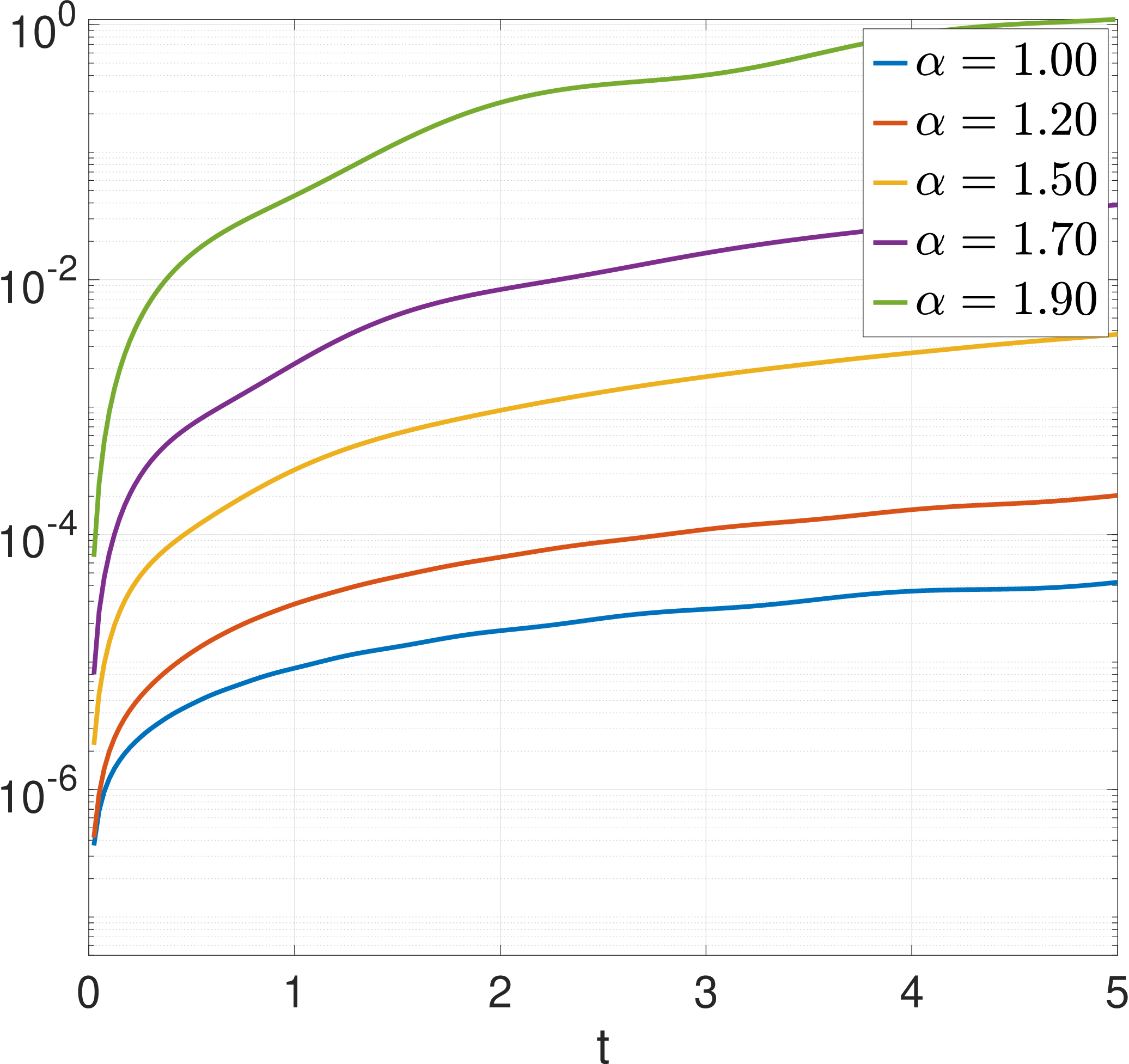}
			\includegraphics[width=0.98\lenthreesubfig, viewport=00 50 1010 960, clip=true]%
			{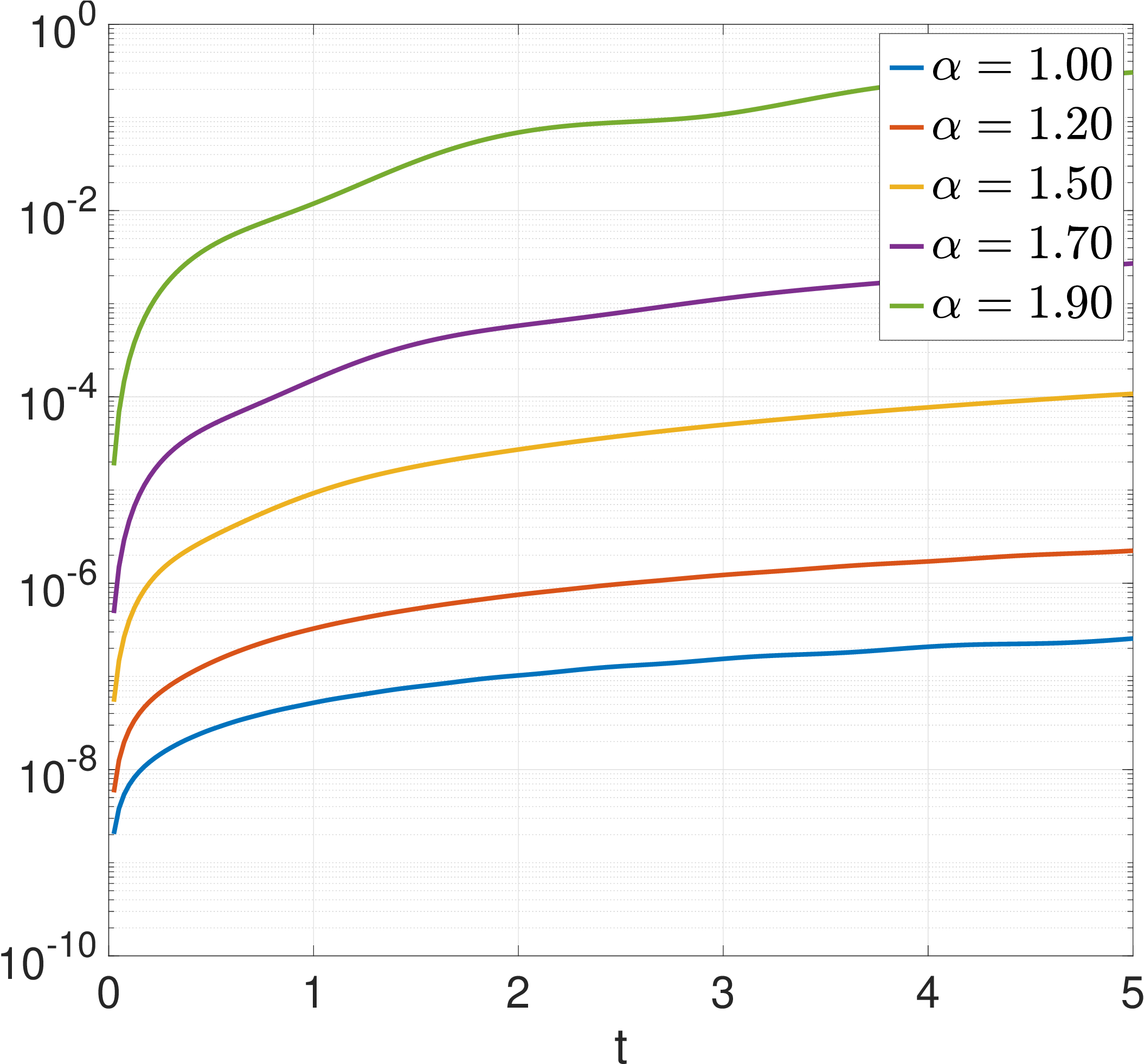}
			\includegraphics[width=0.98\lenthreesubfig, viewport=00 50 1010 960, clip=true]%
			{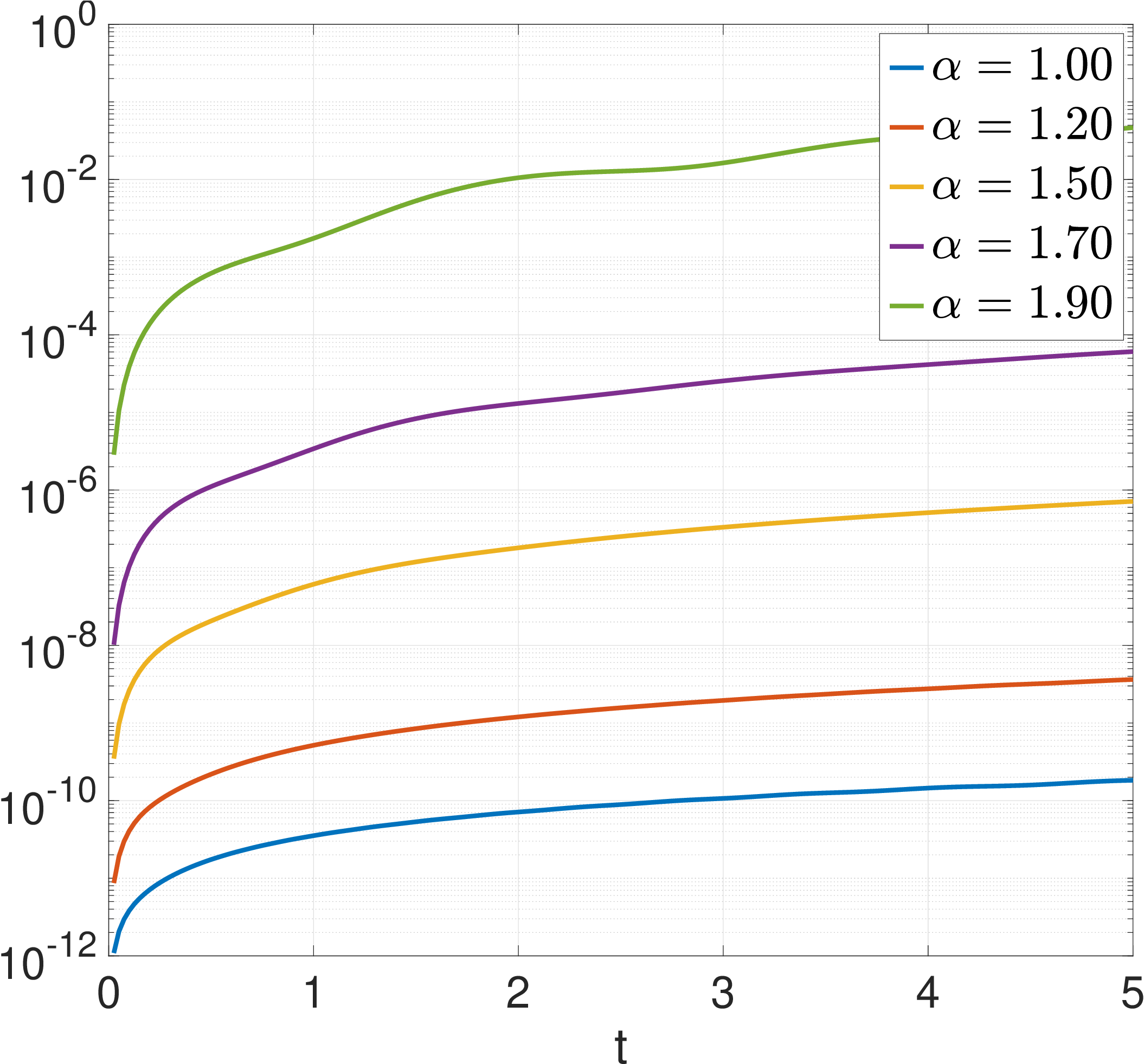}
		\else
			\begin{overpic}[width=0.98\lenthreesubfig, viewport=00 50 1010 960, clip=true]%
				{Ex2/3x2/FCP_inhom_logerr_vs_t_alpha_1_1.2_1.5_1.7_1.9_nf0_1_ndf_4_N_32}
				\put(39,78){\colorboxo{white}{\small $N = 32$} }
				\put(-4,28){\scriptsize \rotatebox{90}{$\cE_{\mathrm{ih}}$}}
				\put(52,-3){\small $t$}
				\put(9,81){\scriptsize ({\bf d})}
			\end{overpic}
			\hspace*{-0.2em}
			\begin{overpic}[width=0.98\lenthreesubfig, viewport=00 50 1010 960, clip=true]%
				{Ex2/3x2/FCP_inhom_logerr_vs_t_alpha_1_1.2_1.5_1.7_1.9_nf0_1_ndf_4_N_64}
				\put(41,78){\colorboxo{white}{\small $N = 64$} }
				\put(-2,41){\scriptsize \rotatebox{90}{$\cE_{\mathrm{ih}}$}}
				\put(52,-3){\small $t$}
				\put(11,81){\scriptsize ({\bf e})}
			\end{overpic}
			\hspace*{-0.1em}
			\begin{overpic}[width=0.98\lenthreesubfig, viewport=00 50 1010 960, clip=true]%
				{Ex2/3x2/FCP_inhom_logerr_vs_t_alpha_1_1.2_1.5_1.7_1.9_nf0_1_ndf_4_N_128}
				\put(40,78){\colorboxo{white}{\small $N = 128$}}
				\put(-2,35){\scriptsize \rotatebox{90}{$\cE_{\mathrm{ih}}$}}
				\put(52,-3){\small $t$}
				\put(11,81){\scriptsize ({\bf f})}
			\end{overpic}
		\fi
		\caption[Example 2: Error versus t]{Error $\cE_{\mathrm{ih}}(t,0.5)$ of the approximate solution $\wt{u}_{\mathrm{ih}}^N$  to problem \eqref{eq:FCP_DE}, \eqref{eq:FCP_BC} with parameters: $f(t) = \sin{\pi x} + t \sin{4\pi x}$; $u_0 = u_1 = 0$;  $A$ defined by \eqref {eq:FCP_ex1_hom_A}; $a = 1$; $N_I = 256$. Graphs from the top row of subplots correspond to $\alpha = 0.1, 0.3, 0.5, 0.7, 1$ and {\bf (a)}  $N = 32$; {\bf (b)} $N = 64$ {\bf (c)}; $N = 128$.  Graphs from the bottom row of plots correspond to $\alpha = 1, 1.2, 1.5, 1.7, 1.9$ and {\bf (d)}  $N = 32$; {\bf (e)} $N = 64$; {\bf (f)} $N = 128$.}
		\label{fig:FCP_Ex2_ex_err_vs_t}
	\end{figure}
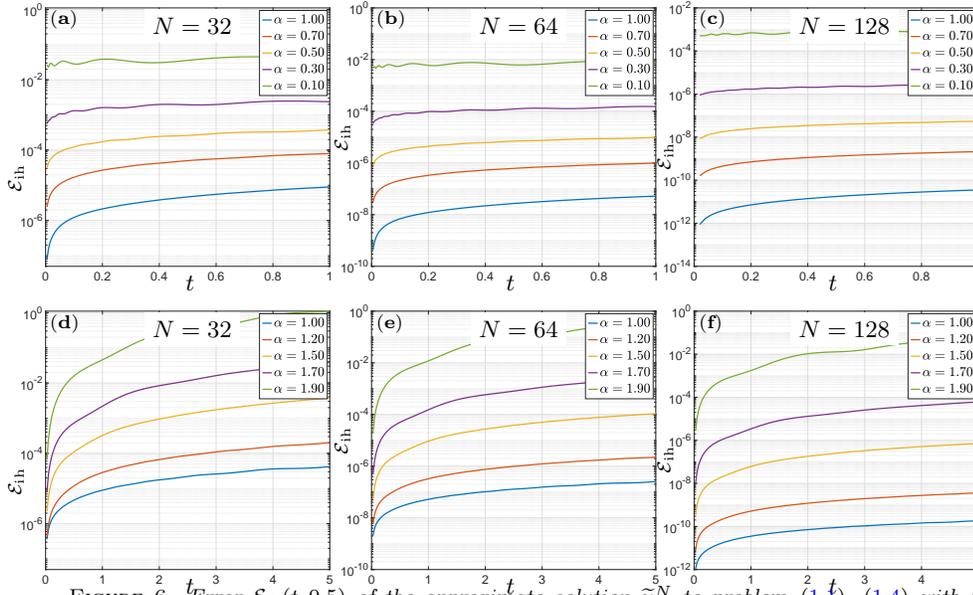
	In contrast to the homogeneous case, now we detect a notable growth in the experimental error as $t$ progresses, for each combination of $\alpha, N$.
	This effect is less sizable for the sub-parabolic case, depicted in \cref{fig:FCP_Ex2_ex_err_vs_t}~(\textbf{a})-(\textbf{c}),  than for the sub-hyperbolic case from \cref{fig:FCP_Ex2_ex_err_vs_t}~(\textbf{d})-(\textbf{f}).
	Such phenomena can be theoretically explained by the presence of factor $t^\alpha$ in the time-dependent part of the error estimate~\eqref{eq:FCP_inhom_sol_err_est}.
	\FloatBarrier
	In order to analyze the error dependency of $N$ in the similar fashion as in \cref{ex:FCP_ex1_hom_R_eigenfunction}, we additionally
	plot the graphs of $\mathrm{err}_{\mathrm{ih}}(N)$ for a range of $\alpha \in [0.1, 1.9]$ in \cref{fig:FCP_Ex2_ex_err_vs_N}.
	Here, we again see notable differences between the cases of $\alpha$ being less and greater than one.
	Judging by the shape of the curves in \cref{fig:FCP_Ex2_ex_err_vs_N}~(\textbf{a}), the errors of $\wt{u}_{\mathrm{ih}}^N(t,x)$ still decay exponentially with respect to $N$ for $\alpha < 1$, but the convergence slows down faster for smaller $\alpha$ than in \cref{fig:FCP_Ex1_ex_err_vs_N}.
	This can be attributed to the influence of the additional factor $t\alpha^{-1}$ from \eqref{eq:FCP_inhom_sol_err_est} which was not present in estimate \eqref{eq:FCP_hom_sol_err_est} for the solution of the homogeneous problem.
	The exponential dependence of the accuracy on $N$ is also observed in \cref{fig:FCP_Ex2_ex_err_vs_N}~(\textbf{b}) for $\alpha>1$ and $t \in [0, 5]$.
	This time, there is no additional convergence order degradation due to $\alpha$ and the plotted graphs look almost identical to the matching graphs from \cref{fig:FCP_Ex1_ex_err_vs_N}~(\textbf{c}).
	Moreover, the reader can clearly note the impact of the larger times on the numerical stability of the method.
	This is an evidence of the method's limitations to treat only moderate values of $T$.
	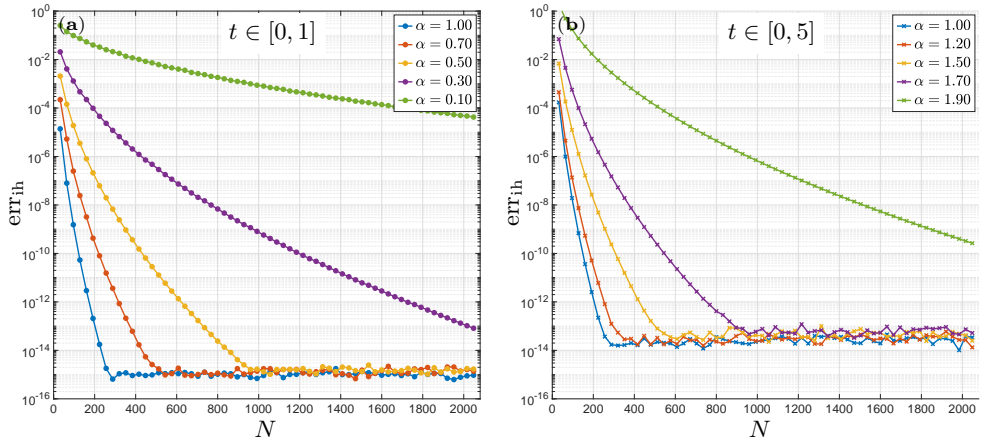
\begin{figure}[h!tb]
		\newdimen\lentwosubfig
		\lentwosubfig=0.48\linewidth
		\iflatexml
			\includegraphics[width=0.98\lentwosubfig, viewport=80 20 1100 960, clip=true]%
			{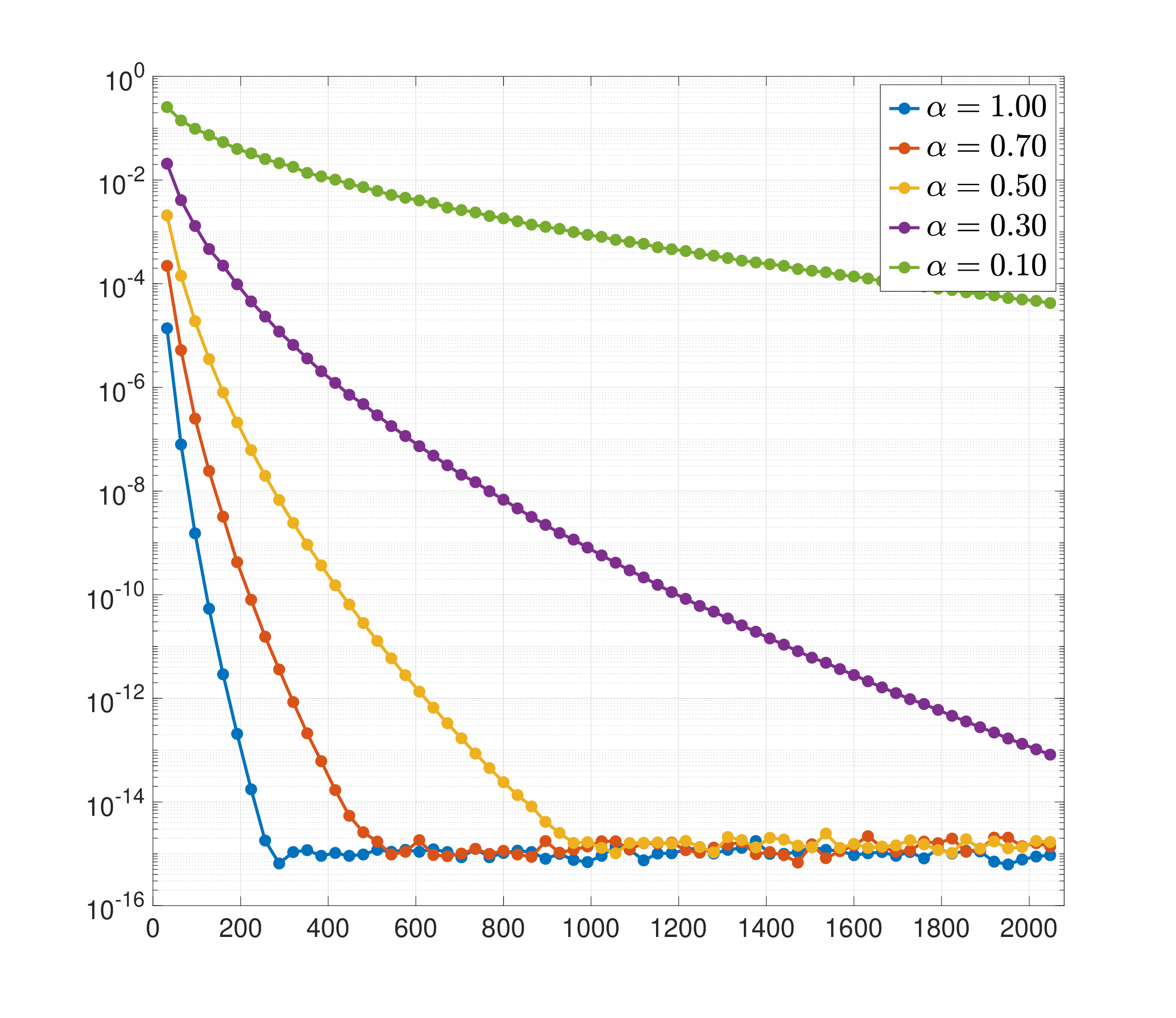}
			\includegraphics[width=0.98\lentwosubfig, viewport=80 20 1100 960, clip=true]%
			{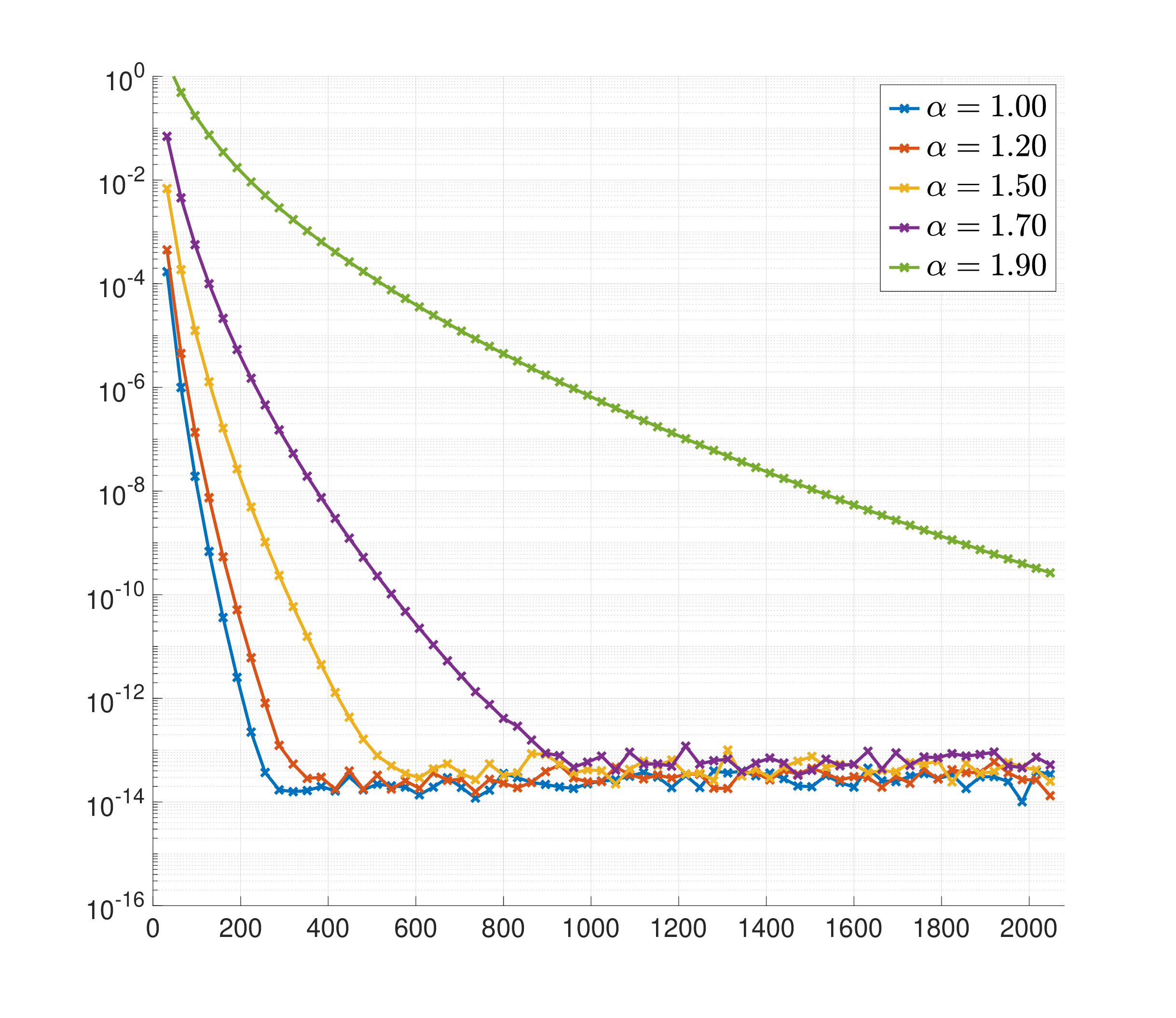}
		\else
			\hspace*{0.1em}
			\begin{overpic}[width=0.98\lentwosubfig, viewport=80 20 1100 960, clip=true]%
				{Ex2/FCP_ihnom_error_vs_N_alpha_0.1_0.3_0.5_0.7_1_nf0_1_ndf_4_D_1}
				\put(42,84){\colorboxo{white}{\small $ t \in [0, 1]$} }
				\put(49,1){\small $N$}
				\put(-2,45){\small \rotatebox{90}{$\mathrm{err}_{\mathrm{ih}}$}}
				\put(8,86){\scriptsize ({\bf a})}
			\end{overpic}
			\hspace*{0.5em}
			\begin{overpic}[width=0.98\lentwosubfig, viewport=80 20 1100 960, clip=true]%
				{Ex2/FCP_ihnom_error_vs_N_alpha_1_1.2_1.5_1.7_1.9_nf0_1_ndf_4_D_1}
				\put(42,84){\colorboxo{white}{\small $ t \in [0, 5]$} }
				\put(48,1){\small $N$}
				\put(-2,45){\small \rotatebox{90}{$\mathrm{err}_{\mathrm{ih}}$}}
				\put(8,86){\scriptsize ({\bf b})}
			\end{overpic}
		\fi
		\caption[Example 2: Error versus N]{%
			Sup-norm error of the approximate solution to problem \eqref{eq:FCP_DE}, \eqref{eq:FCP_BC} with $f(t) = \sin{\pi x} + t \sin{4\pi x}$, $u_0 = u_1 = 0$, and the operator $A$ defined by \eqref {eq:FCP_ex1_hom_A}; $L =1$,  $a = 1$. Errors are plotted for  $N=32, 64, 96, \ldots, 2048$ and %
			(\textbf{a}) $t \in [0,1]$, $\alpha = 0.1, 0.3, 0.5, 0.7, 1$; %
			(\textbf{b}) $t \in [0,5]$, $\alpha = 1, 1.2, 1.5, 1.7, 1.9$. %
		}
		\label{fig:FCP_Ex2_ex_err_vs_N}
	\end{figure}
\end{example}
\FloatBarrier
At this point, we presented enough experimental data to conduct  a meaningful comparisons with existing numerical methods.
We choose review \cite{Garrappa2015a} as a main comparison source because it contains error data for several modern time-stepping numerical methods applied to the linear scalar problem with the same range $\alpha \in [0.1, 1.9]$ as in the \mbox{\cref{fig:FCP_Ex2_ex_err_vs_t,fig:FCP_Ex2_ex_err_vs_N}.}
The error plots in Figure 4 from \cite{Garrappa2015a} are generated using 1024 grid points in time.
Thus, their evaluation is computationally comparable to setting $N=1023$  in the sequential version of \cref{alg:FCP_hom_sol_appr}.
With such $N$, our method gives approximately two times more significant digits then the mentioned second order time-stepping methods, provided that the fractional parameter is not too small ($\alpha > 0.2$).
This result improves drastically when the parallel evaluation is considered, because then the wall-time computational cost of our method is asymptotically equivalent to one step of the intrinsically sequential time-stepping numerical scheme.
As a result, our method is favored for the problems with initial data that satisfy \cref{thm:FCP_prop_appr}, especially when computational resources are plentiful.
On the other hand, sequential time-stepping methods \cite{Garrappa2015a,Baffet2017,Khristenko2021} may be a better choice in situations when the initial data $u_0, u_1$ are not sufficiently smooth \cite{jin2019numerical}, $\alpha$ is close to $0$ or if the computational resources are severely constrained.

In the final example, we consider a fully discretized numerical method for the given fractional Cauchy problem.
It is constructed by applying the solution scheme from \cref{sec:FCP_hom_sol_appr,sec:FCP_inhom_sol_appr}
to the fractional Cauchy problem \eqref{eq:FCP_DE},  \eqref{eq:FCP_BC}, where $A$ is substituted by the bounded linear operator $\wt{A}$ obtained via the finite-difference discretization of \eqref{eq:FCP_ex1_hom_A}.
The initial conditions and right-hand side for the problem are derived using the method of manufactured solutions.
\begin{example}\label{ex:FCP_ex3_inhom_R_FD}
	Let $A$ be defined as in \cref{ex:FCP_ex1_hom_R_eigenfunction}, with $L=1$.
	We postulate that the exact solution to problem \eqref{eq:FCP_DE}, \eqref{eq:FCP_BC} is defined as
	\begin{equation}\label{eq:Ex3_u}
		u(t,x) = x^2 (x-1) \left (x - t^\delta - b \right ), \quad \delta > 1, \ b \in \R.
	\end{equation}
	Then,
	\[
		\begin{aligned}
			A u(t)                 & = 2t^{\delta} (3 x - 1) -12 x^{2}+6x (b +1) - 2 b                                                                                                                                                                                           \\
			\partial_t^\alpha u(t) & =  -\frac{\delta  x^2 \left(x -1\right) }{\Gamma\left(1-\alpha \right)} \intl_{0}^{t}\left(t -s \right)^{-\alpha} s^{\delta -1} d\, s = -\frac{\delta\Gamma (\delta)t^{\delta-\alpha}}{\Gamma (1+\delta-\alpha )}  x^2 \left(x -1\right)  ,
		\end{aligned}
	\]
	so the right-hand side of \eqref{eq:FCP_DE} takes the form
	\begin{equation}\label{eq:Ex3_f}
		f(t)  = 6 t^{\delta} x -2 t^{\delta}  -\frac{ t^{\delta-\alpha} \delta !}{\Gamma (\delta+1-\alpha )}  x^2 \left(x -1\right) -12 x^{2}+6x (b + 1) -2 b.
	\end{equation}
	Such $f(t)$ permits us to study one important practical aspect of the developed solution method.
	Namely, what happens to the accuracy of a fully-discretized solution when $f'(t)$ does not formally belong to the domain of $A$, but the discretization $\wt{A}$ satisfies $\|\wt{A}\wt{f'}(t)\| < \infty$?

	Let $\wt{A}$ be $m \times m$ matrix obtained by a second-order finite-difference discretization of operator \eqref{eq:FCP_ex1_hom_A} on a grid $\Delta_d = \left \{(i-1)/(m-1) \right \}_{i=1}^m$.
	Then, the discretized right-hand side $\wt{f'}(t) \in \left(R^m, \|\ \|_\infty \right)$ is defined by the projection of $f'(t)$ onto $\Delta_d$: $\wt{f'}(t) = \left (f'(t,0), f'(t,x_2), \ldots, f'(t,L) \right )^T$.
	We set $\delta = 2$, $b=-1/2$, and visualize the graph of the derivative
	\[
		f'(t) = 6 \delta t^{\delta-1}x - 2\delta t^{\delta-1} - \frac{ t^{\delta-\alpha-1} (\delta-\alpha) \delta !}{\Gamma (\delta+1-\alpha )}  x^2 \left(x -1\right),
	\]
	for different values of $t, \alpha$, in \cref{fig:FCP_Ex3_ex_sol}a.

	\begin{figure}[h!tb]
		\newdimen\lentwosubfig
		\lentwosubfig=0.48\linewidth
		\iflatexml
			\includegraphics[width=0.98\lentwosubfig, viewport=80 60 1100 1000, clip=true]%
			{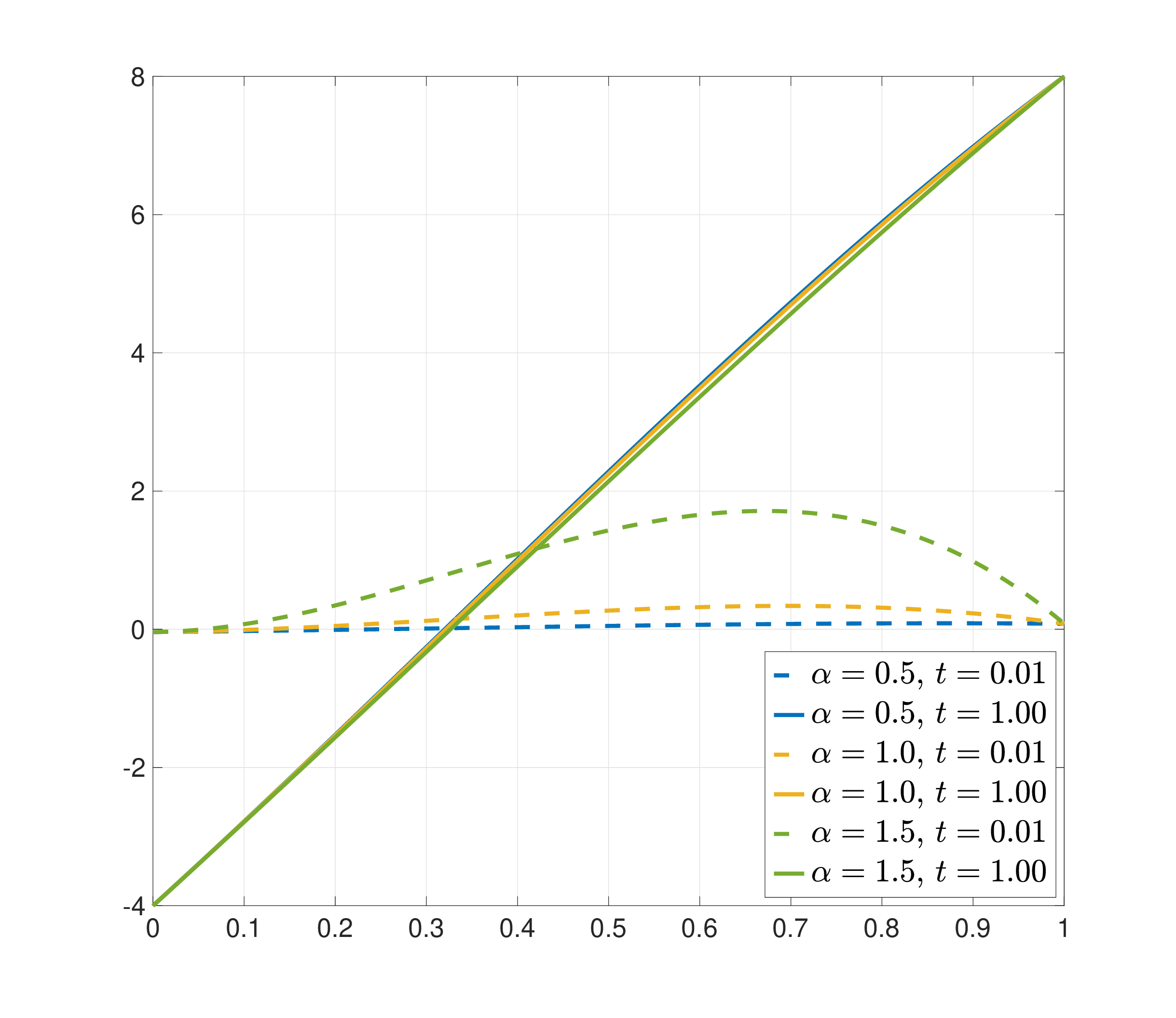}
			\includegraphics[width=0.98\lentwosubfig, viewport=80 60 1100 1000, clip=true]%
			{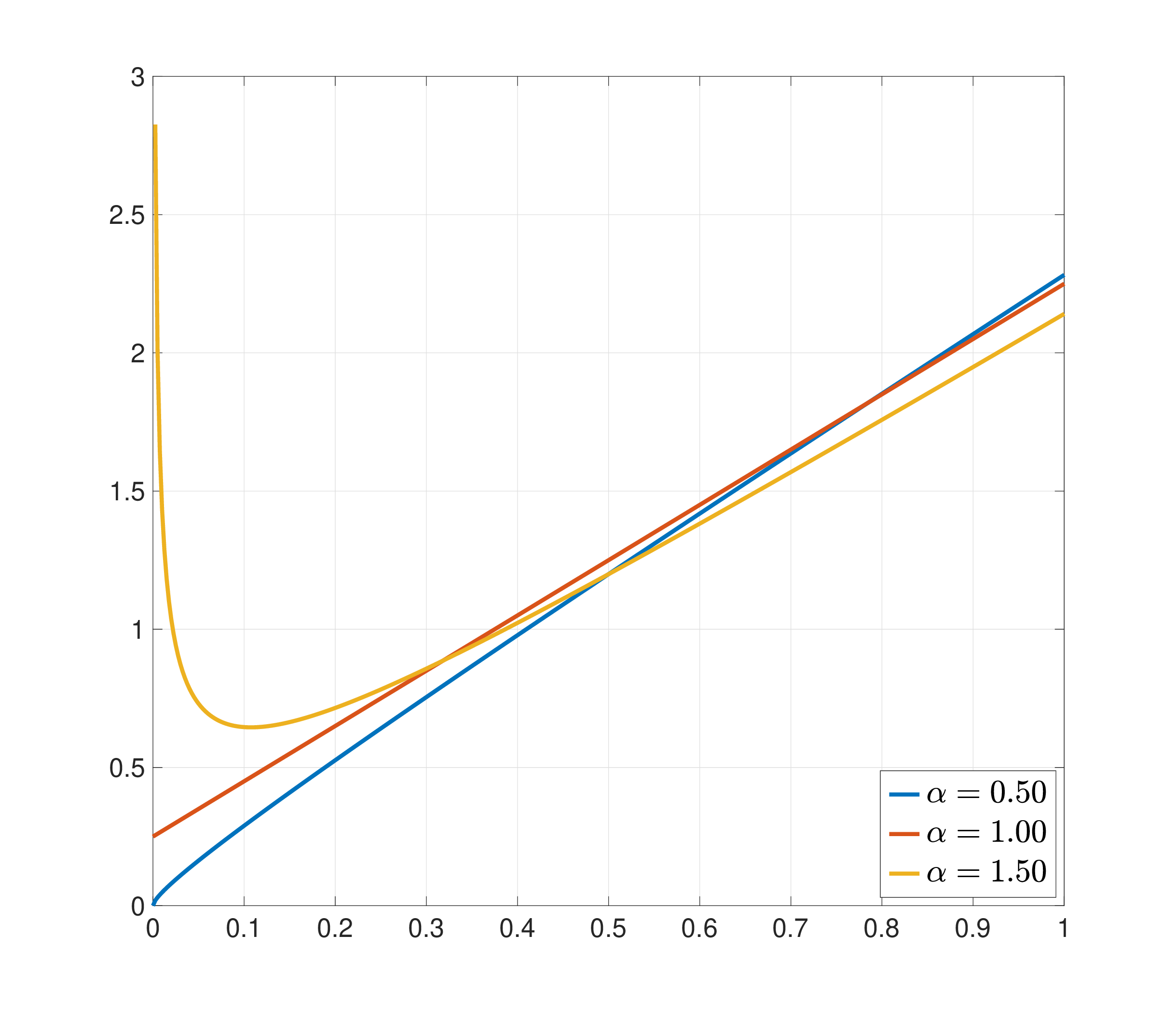}
		\else
			\hspace*{0.4em}
			\begin{overpic}[width=0.98\lentwosubfig, viewport=80 60 1100 1000, clip=true]%
				{Ex3/FCP_Ex3_x2_df_vs_x_alpha_0.5_1_1.5_d_2_s__0_5_t_0.01__1}
				\put(55,-1){\small $x$}
				\put(-3,41){\small \rotatebox{90}{$f'$}}
				\put(8,82){\scriptsize ({\bf a})}
			\end{overpic}\hfil
			\begin{overpic}[width=0.98\lentwosubfig, viewport=80 60 1100 1000, clip=true]%
				{Ex3/FCP_Ex3_x2_df_vs_t_alpha_0.5_1_1.5_d_2_s__0_5_x_0.5}
				\put(55,-1){\small $t$}
				\put(-3,41){\small \rotatebox{90}{$f'$}}
				\put(8,82){\scriptsize ({\bf b})}
			\end{overpic}
		\fi
		\caption[Example 3: Exact solution]{\small The graph of exact solution \eqref{eq:Ex3_u}: (\textbf{a}) plotted as a function of $x \in [0, 1]$ for $t=0.01, 1$, $\alpha = 0.5, 1, 1.5$; (\textbf{b}) plotted as a function of $t \in (0, 1]$ for $x=0.5$, $\alpha = 0.5, 1, 1.5$.} \label{fig:FCP_Ex3_ex_sol}
	\end{figure}
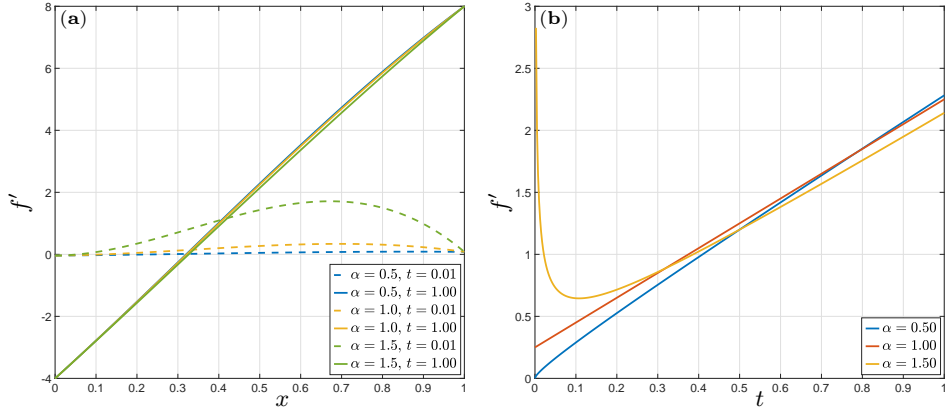
	As we can see, the function $f'(t)$, $t > 0$ does not satisfy the boundary conditions from \eqref{eq:FCP_ex1_hom_A}; hence, $f'(t) \notin D(A)$.
	Furthermore, when $\alpha > 1$, this function posses an integrable singularity at $t=0$; see \cref{fig:FCP_Ex3_ex_sol}~(\textbf{b}).
	This permits us to test \cref{rem:FCP_RLInt_appr_ext}, alluding that even for such $f'(t)$ the approximation $\wt{J}_\alpha^N f'(t)$ from \cref{prop:FCP_RLInt} remains exponentially convergent.
	To this end, we define the approximation errors
	\[
		\cE(t,x) = \left| u(t,x) - \wt{u}_m^N(t,x) \right|,
		\quad \mathrm{err} =  \sup\limits_{t \in [0, T]} 	\left\|  \cE(t)\right\|_\infty ,
	\]
	where
	$\wt{u}_m^N(t,x)$, $x \in \Delta_m$ is the numerical solution to \eqref{eq:FCP_DE}, \eqref{eq:FCP_BC}, with $A = \wt{A}$, $f(t) = \wt{f}(t)$ and $u_0 = \wt{u}(0,x)$, $u_1 = \wt{u'}(0,x)$, calculated by \cref{alg:FCP_hom_sol_appr,alg:FCP_inhom_sol_appr}.
	We set $a = 1$, $\varphi_s = \pi/60$, $\gamma = \chi =1$, as before, and consider the impact of the discretization parameters $N, m$ on the solution's~accuracy.

	In the first batch of experiments, we vary $N$, $\alpha$, while keeping the grid in space fixed with $m=100$.
	The resulting graphs of $\cE(t,0.1)$, $t \in [0,1]$ are displayed in \cref{fig:FCP_Ex3_ex_err_vs_t}.
	\begin{figure}[h!tb]
		\newdimen\lenthreesubfig
		\lenthreesubfig=0.325\linewidth
		\iflatexml
			\includegraphics[width=0.98\lenthreesubfig, viewport=00 00 1010 900, clip=true]%
			{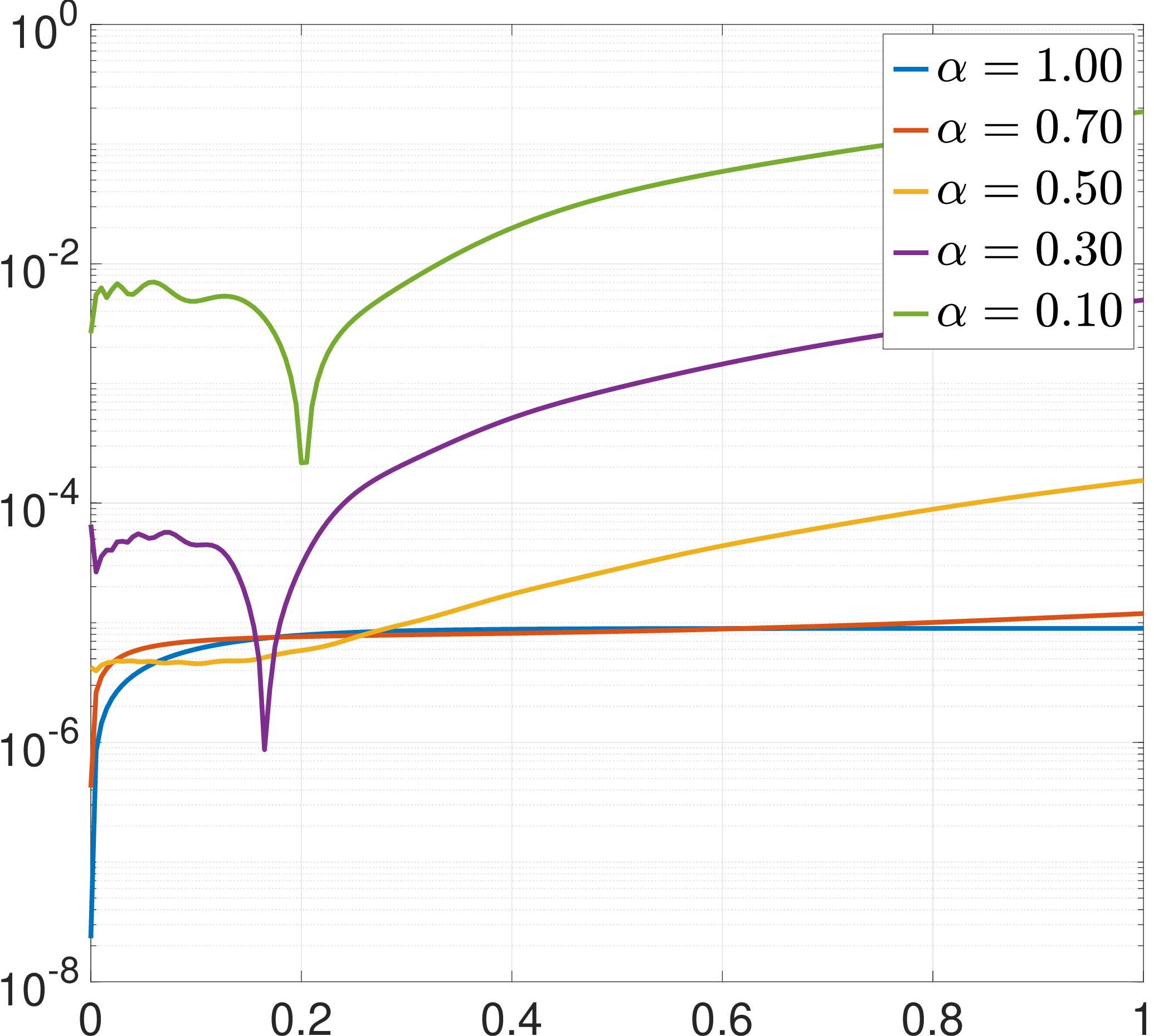}
			\includegraphics[width=0.98\lenthreesubfig, viewport=00 00 1010 900, clip=true]%
			{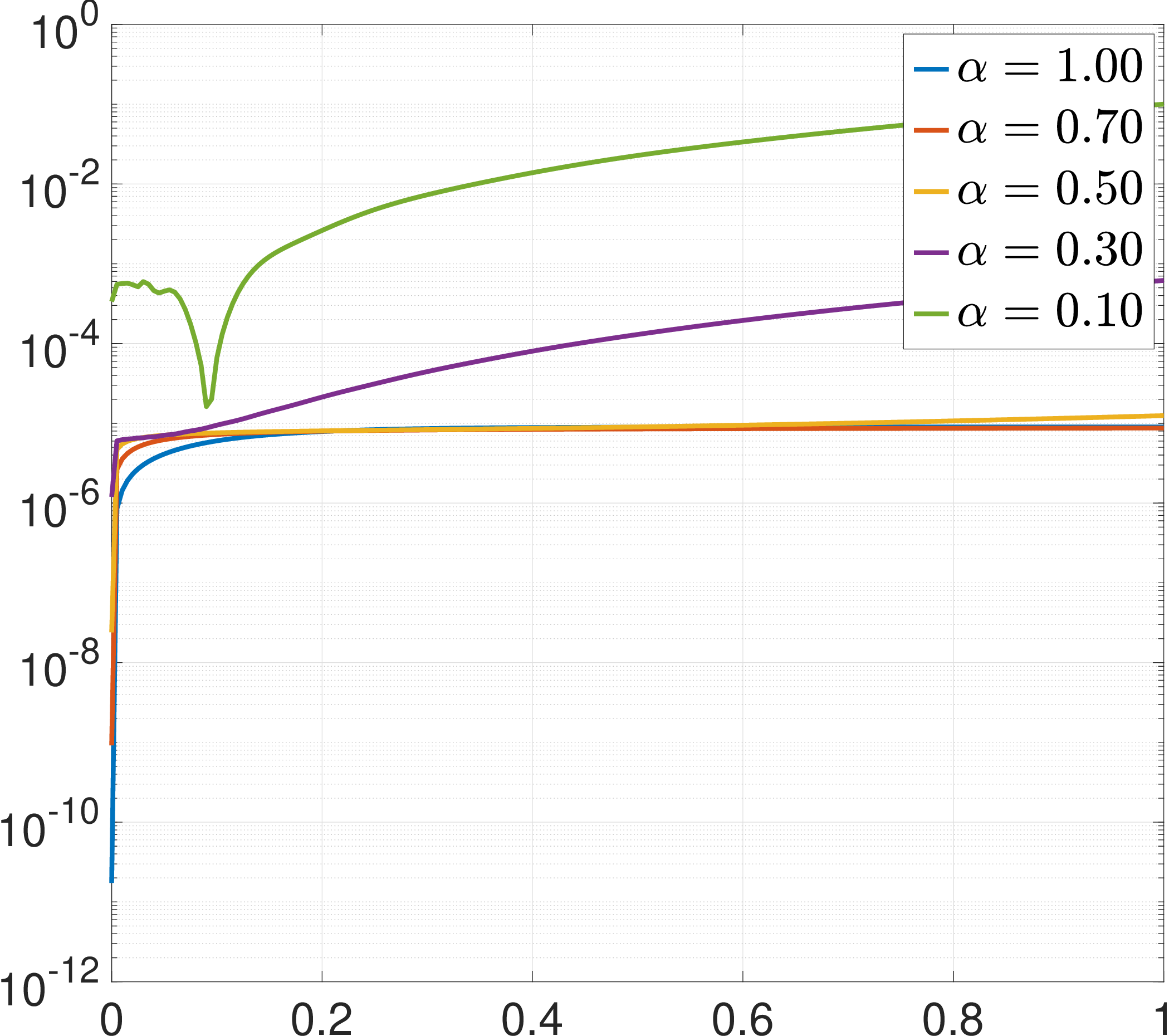}
			\includegraphics[width=0.98\lenthreesubfig, viewport=00 00 1010 900, clip=true]%
			{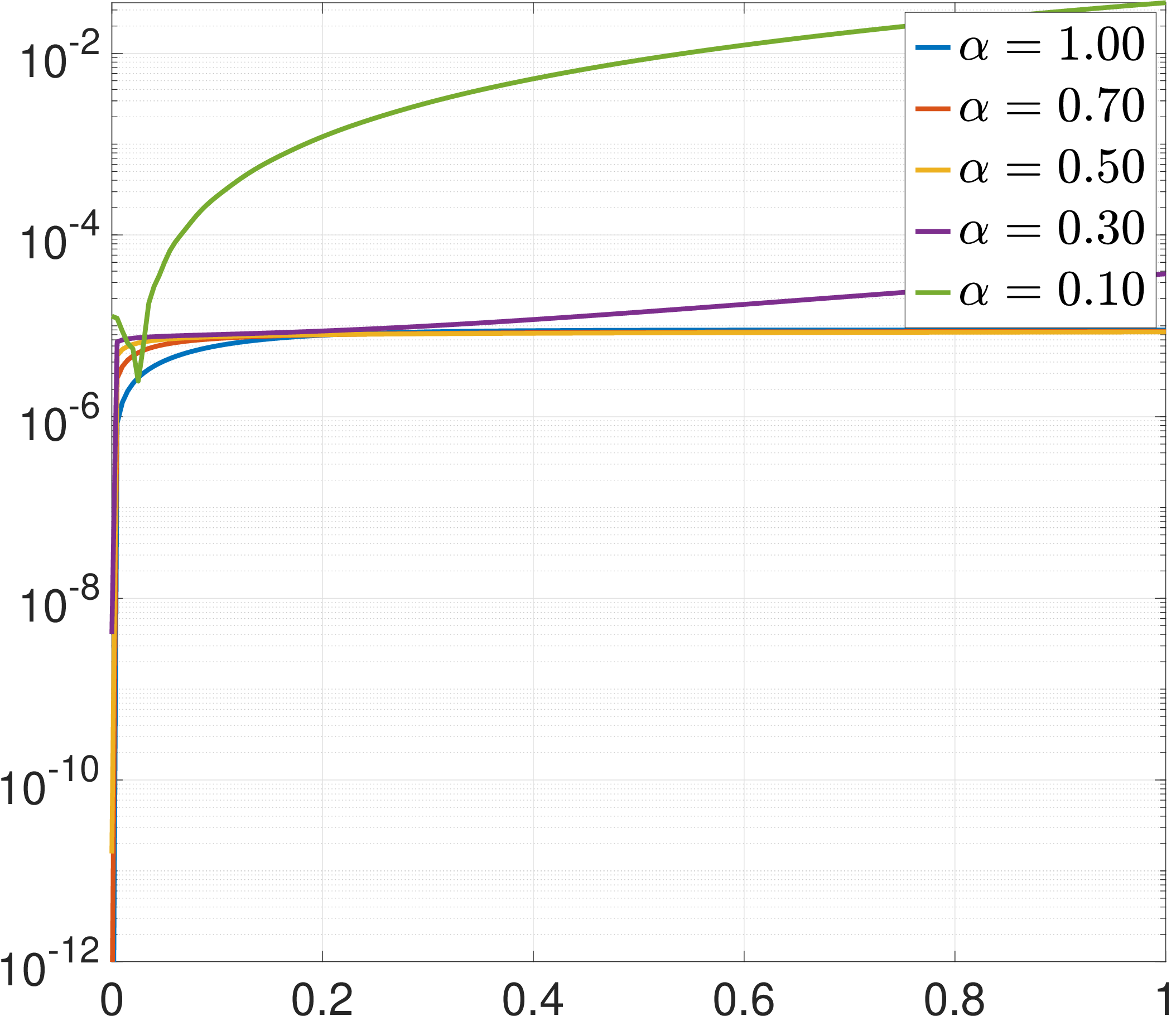}
		\else
			\hspace*{0.05em}
			\begin{overpic}[width=0.98\lenthreesubfig, viewport=00 00 1010 900, clip=true]%
				{Ex3/3x2/FCP_Ex3_x2_err_vs_t_N_64_alpha_0.1_0.3_0.5_0.7_1.0_d_2_s_0_5_t_0_1_Nt_200_Nx_100}
				\put(41,78){\colorboxo{white}{\small $N = 64$} }
				\put(-2,33){\scriptsize \rotatebox{90}{$\cE$}}
				\put(53,-3){\small $t$}
				\put(9,81){\scriptsize ({\bf a})}
			\end{overpic}
			\hspace*{-0.1em}
			\begin{overpic}[width=0.98\lenthreesubfig, viewport=00 00 1010 900, clip=true]%
				{Ex3/3x2/FCP_Ex3_x2_err_vs_t_N_128_alpha_0.1_0.3_0.5_0.7_1.0_d_2_s_0_5_t_0_1_Nt_200_Nx_100}
				\put(41,78){\colorboxo{white}{\small $N = 128$} }
				\put(-2,37){\scriptsize \rotatebox{90}{$\cE$}}
				\put(53,-3){\small $t$}
				\put(11,81){\scriptsize ({\bf b})}
			\end{overpic}
			\hspace*{-0.1em}
			\begin{overpic}[width=0.98\lenthreesubfig, viewport=00 00 1010 900, clip=true]%
				{Ex3/3x2/FCP_Ex3_x2_err_vs_t_N_256_alpha_0.1_0.3_0.5_0.7_1.0_d_2_s_0_5_t_0_1_Nt_200_Nx_100}
				\put(41,78){\colorboxo{white}{\small $N = 256$} }
				\put(-1,41){\scriptsize \rotatebox{90}{$\cE$}}
				\put(53,-3){\small $t$}
				\put(11,81){\scriptsize ({\bf c})}
			\end{overpic}
		\fi
		\\[6pt]
		\iflatexml
			\includegraphics[width=0.98\lenthreesubfig, viewport=00 00 1010 900, clip=true]%
			{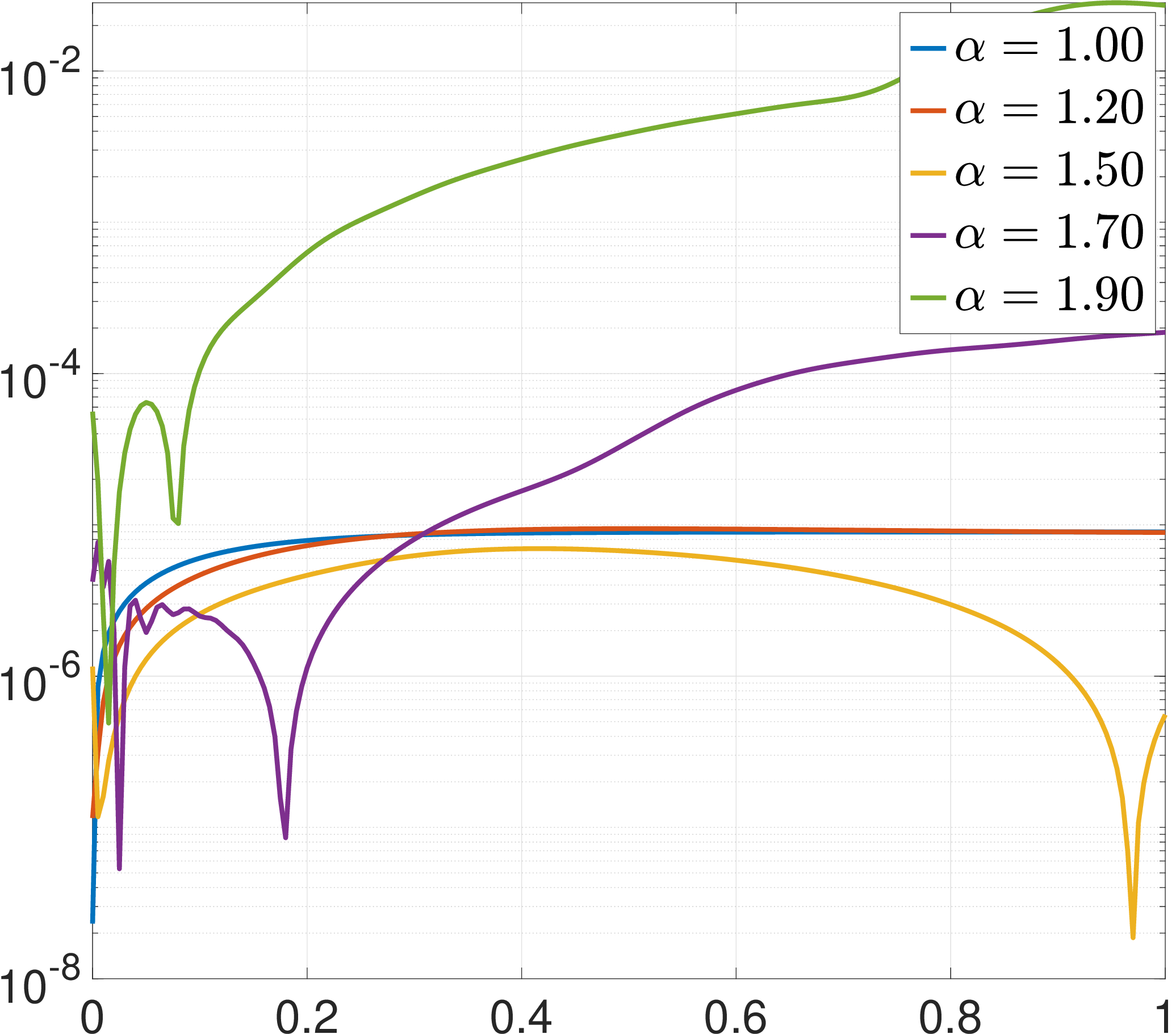}
			\includegraphics[width=0.98\lenthreesubfig, viewport=00 00 1010 900, clip=true]%
			{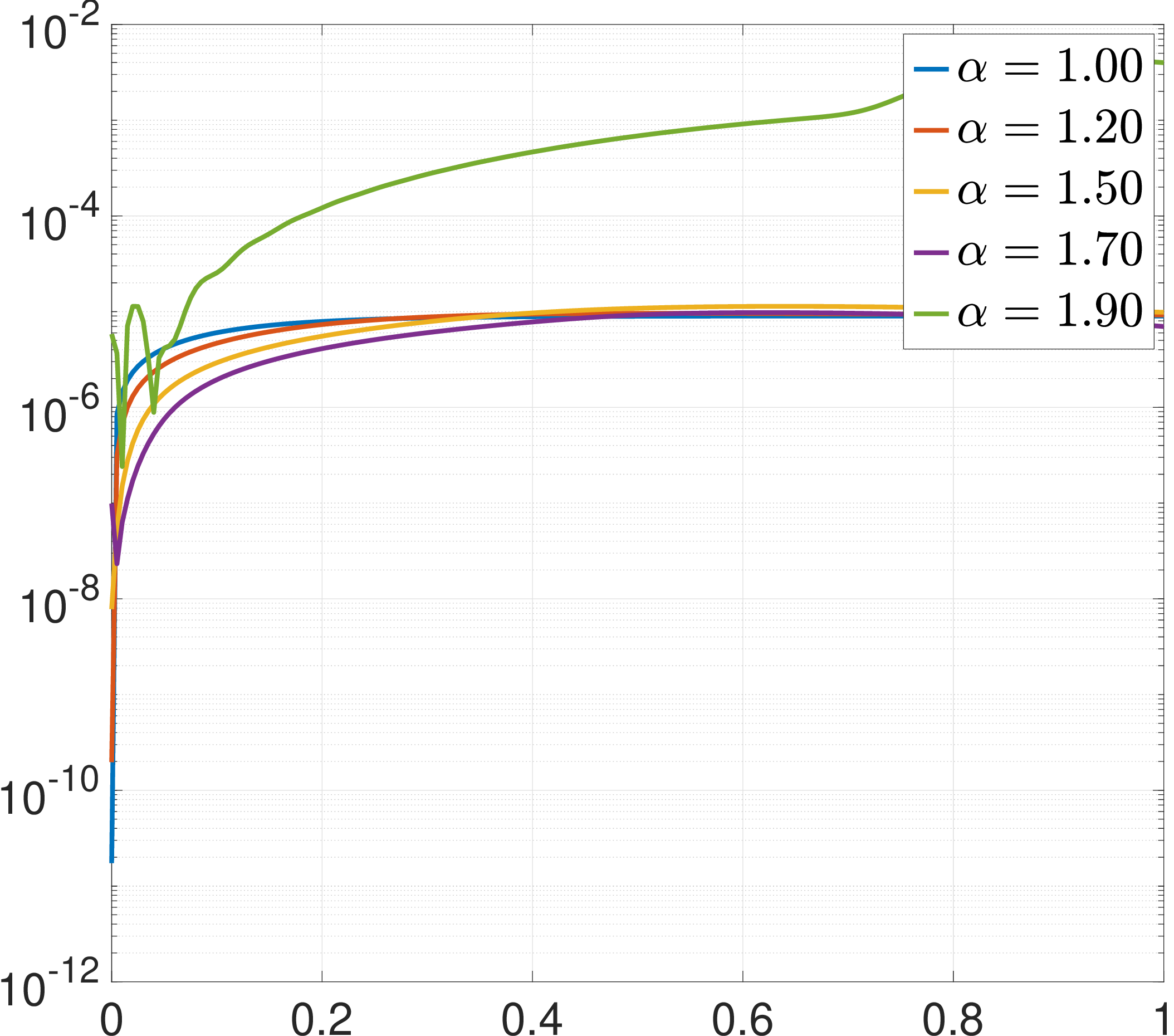}
			\includegraphics[width=0.98\lenthreesubfig, viewport=00 00 1010 900, clip=true]%
			{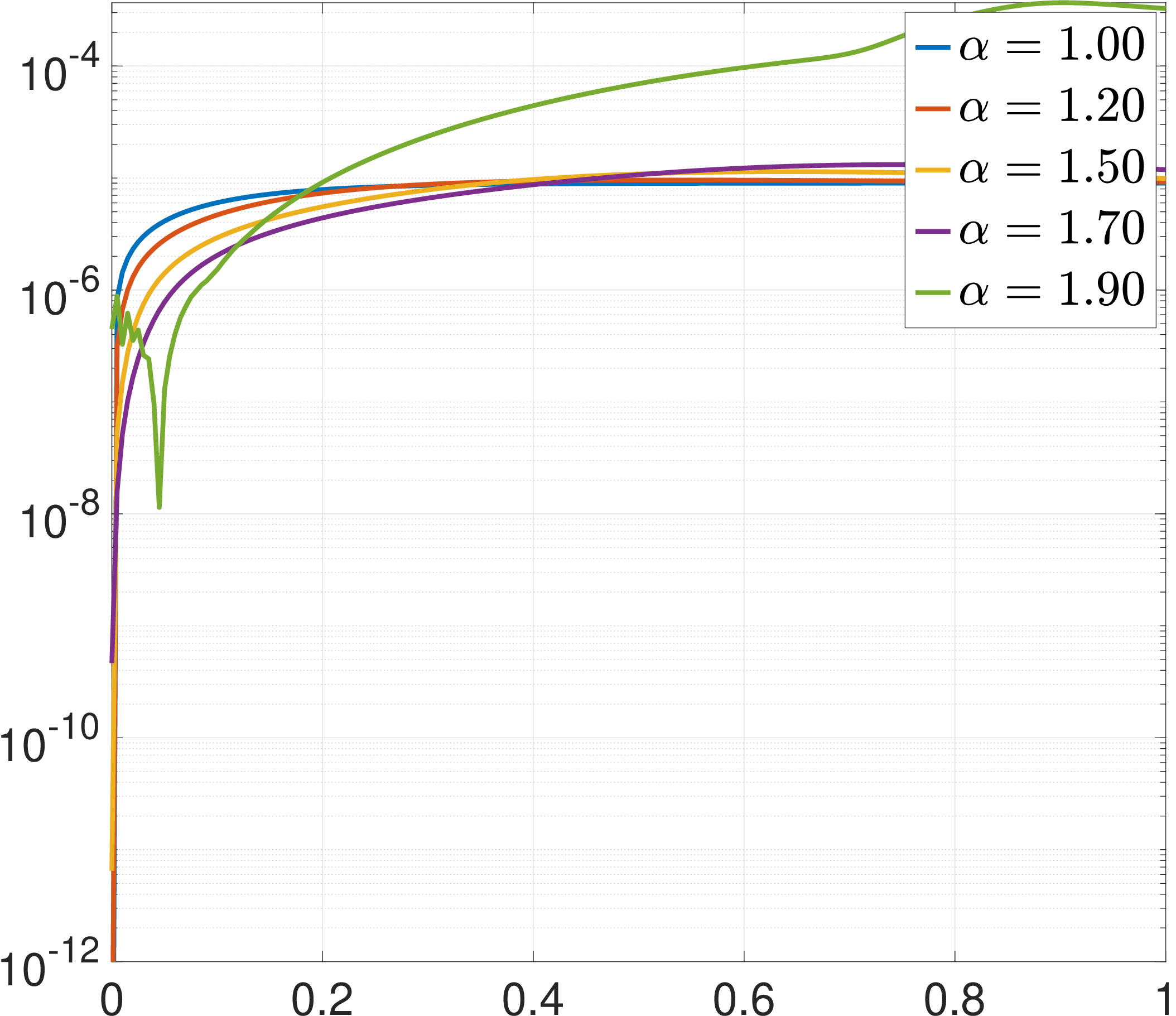}
		\else
			\hspace*{0.05em}
			\begin{overpic}[width=0.98\lenthreesubfig, viewport=00 00 1010 900, clip=true]%
				{Ex3/3x2/FCP_Ex3_x2_err_vs_t_N_64_alpha_1_1.2_1.5_1.7_1.9_d_2_s_0_5_t_0_1_Nt_200_Nx_100}
				\put(41,78){\colorboxo{white}{\small $N = 64$} }
				\put(-2,40){\scriptsize \rotatebox{90}{$\cE$}}
				\put(53,-3){\small $t$}
				\put(9,80){\scriptsize ({\bf d})}
			\end{overpic}
			\hspace*{-0.1em}
			\begin{overpic}[width=0.98\lenthreesubfig, viewport=00 00 1010 900, clip=true]%
				{Ex3/3x2/FCP_Ex3_x2_err_vs_t_N_128_alpha_1_1.2_1.5_1.7_1.9_d_2_s_0_5_t_0_1_Nt_200_Nx_100}
				\put(41,78){\colorboxo{white}{\small $N = 128$} }
				\put(-2,43){\scriptsize \rotatebox{90}{$\cE$}}
				\put(53,-3){\small $t$}
				\put(11,80){\scriptsize ({\bf e})}
			\end{overpic}
			\hspace*{-0.1em}
			\begin{overpic}[width=0.98\lenthreesubfig, viewport=00 00 1010 900, clip=true]%
				{Ex3/3x2/FCP_Ex3_x2_err_vs_t_N_256_alpha_1_1.2_1.5_1.7_1.9_d_2_s_0_5_t_0_1_Nt_200_Nx_100}
				\put(41,78){\colorboxo{white}{\small $N = 256$} }
				\put(-1,50){\scriptsize \rotatebox{90}{$\cE$}}
				\put(53,-3){\small $t$}
				\put(11,80){\scriptsize ({\bf f})}
			\end{overpic}
		\fi
		\caption[Example 3: Error versus t]{Error $\cE(t,0.1)$ of the fully discretized approximation $\wt{u}_{100}^N$  to the solution of \eqref{eq:FCP_DE}, \eqref{eq:FCP_BC} with $A$, $f(t)$ defined by \eqref{eq:FCP_ex1_hom_A} and \eqref{eq:Ex3_f}, correspondingly; $L = 1$, $a = 1$, $\varphi_s = \pi/60$, $\gamma = \chi =1$. Graphs from the top row of subplots correspond to $\alpha = 0.1, 0.3, 0.5, 0.7, 1$ and {\bf (a)}  $N = 64$; {\bf (b)} $N = 128$ {\bf (c)}; $N = 256$.  Graphs from the bottom row of plots correspond to $\alpha = 1, 1.2, 1.5, 1.7, 1.9$ and {\bf (d)}  $N = 64$; {\bf (e)} $N = 128$; {\bf (f)} $N = 256$.}
		\label{fig:FCP_Ex3_ex_err_vs_t}
	\end{figure}
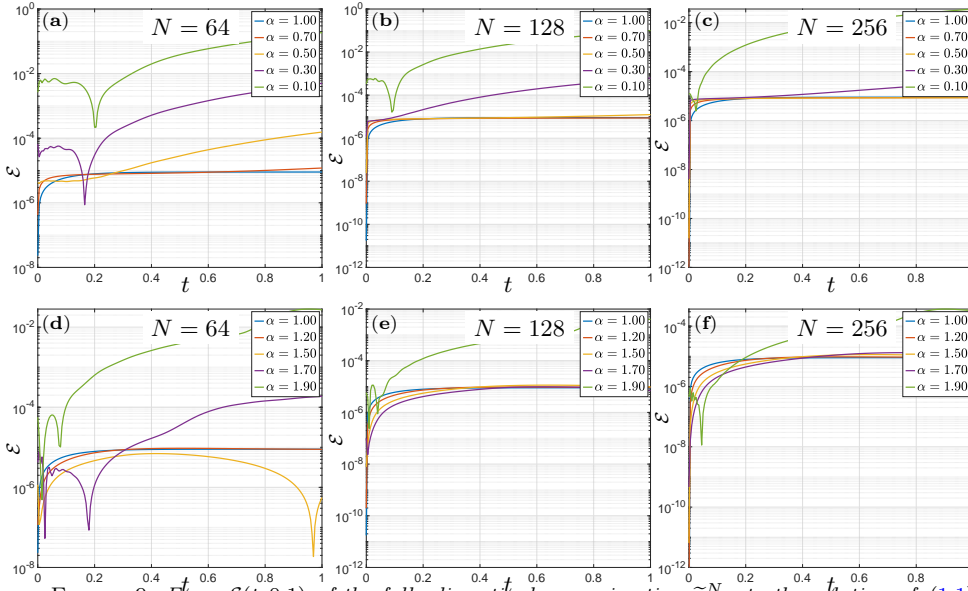
	The behavior of $\cE(t,0.1)$ in these graphs follows the pattern predicted by \cref{thm:FCP_prop_appr,thm:FCP_inhom_sol_err_est}, albeit this time the error saturation occurs at about $10^{-5} \leq m^{-2}$, when the effect of the second-order accuracy in space becomes dominant.
	Aside of that, we observe no accuracy degradation as compared to \cref{ex:FCP_ex1_hom_R_eigenfunction,ex:FCP_ex2_inhom_R_eigenfunction}, where the space-dependent component of the solution was evaluated explicitly.
	For $\alpha>1$, this demonstrates the aforementioned robustness of $\wt{J}_\alpha^N$ with respect to the integrable singularity of $f'(t)$ at $t = 0$.

	It is important to highlight that the a priori error estimates from \cref{sec:FCP_num_method} do not enforce any dependencies between the discretization parameters $N, m$ of the approximation $\wt{u}_m^N(t,x)$ or the chosen grids in time and space.
	To practically verify this proposition, we consider the sup-norm error $\mathrm{err}$ of the approximated solution.
	In the second batch of experiments, this error is evaluated for the increasing sequence of $m = 10^1, 10^2, 10^3, 10^4$, $N = 32,64, 96, \ldots 512$ and different $\alpha$.
	The resulting graphs, depicted in \cref{fig:FCP_Ex3_ex_err_vs_N}, reaffirm the pointwise error behavior observed  in \cref{fig:FCP_Ex3_ex_err_vs_t}.

	The sup-norm error is decaying exponentially with respect to $N$ until it plateaus near the certain value, which is roughly constant within each subfigure.
	For larger $m$, the mentioned plateauing occurs at a smaller value, consistently following the second-order decay rate with respect to the grid step-size in space.
	This is true all across the range of $m/N$ covered in \cref{fig:FCP_Ex3_ex_err_vs_N}~(\textbf{a})-(\textbf{d}).
	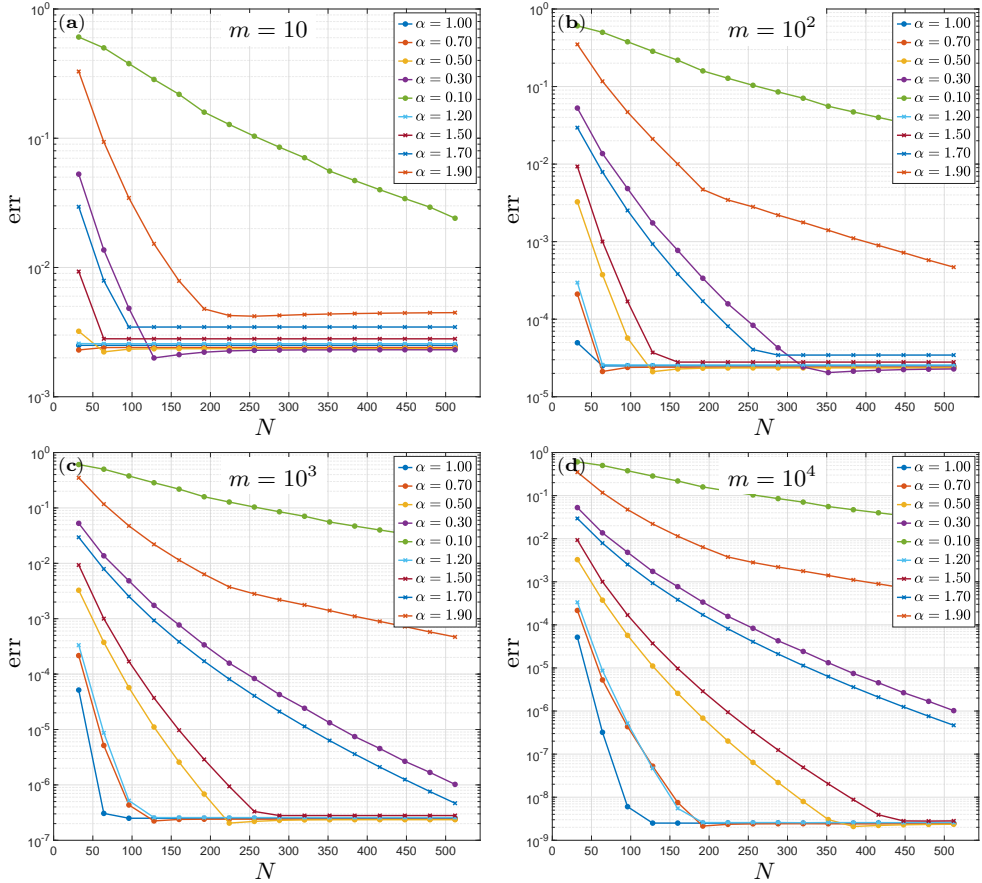
\begin{figure}[h!tb]
		\newdimen\lentwosubfig
		\lentwosubfig=0.48\linewidth
		\iflatexml
			\includegraphics[width=0.98\lentwosubfig, viewport=80 20 1100 960, clip=true]%
			{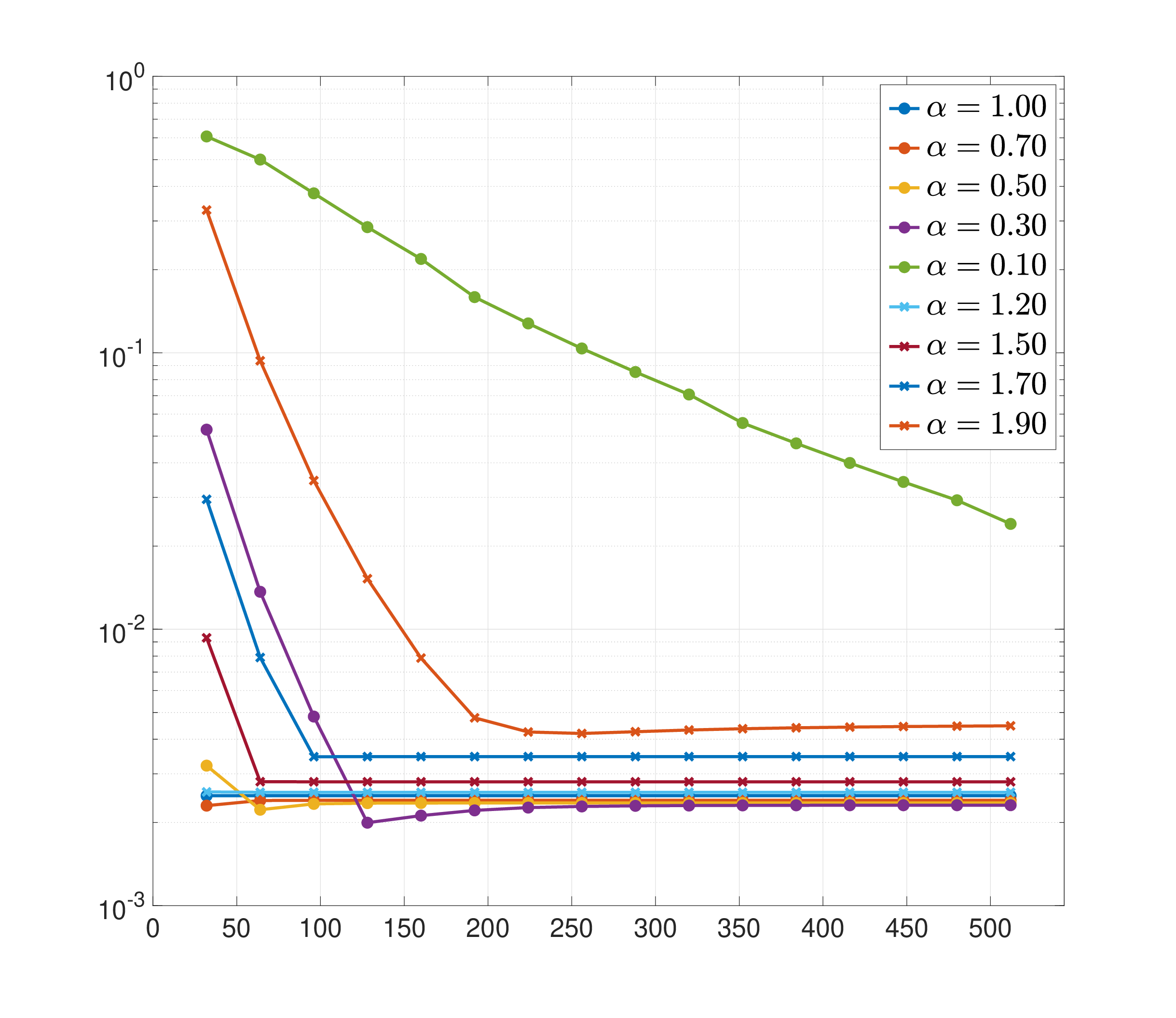}
			\includegraphics[width=0.98\lentwosubfig, viewport=80 20 1100 960, clip=true]%
			{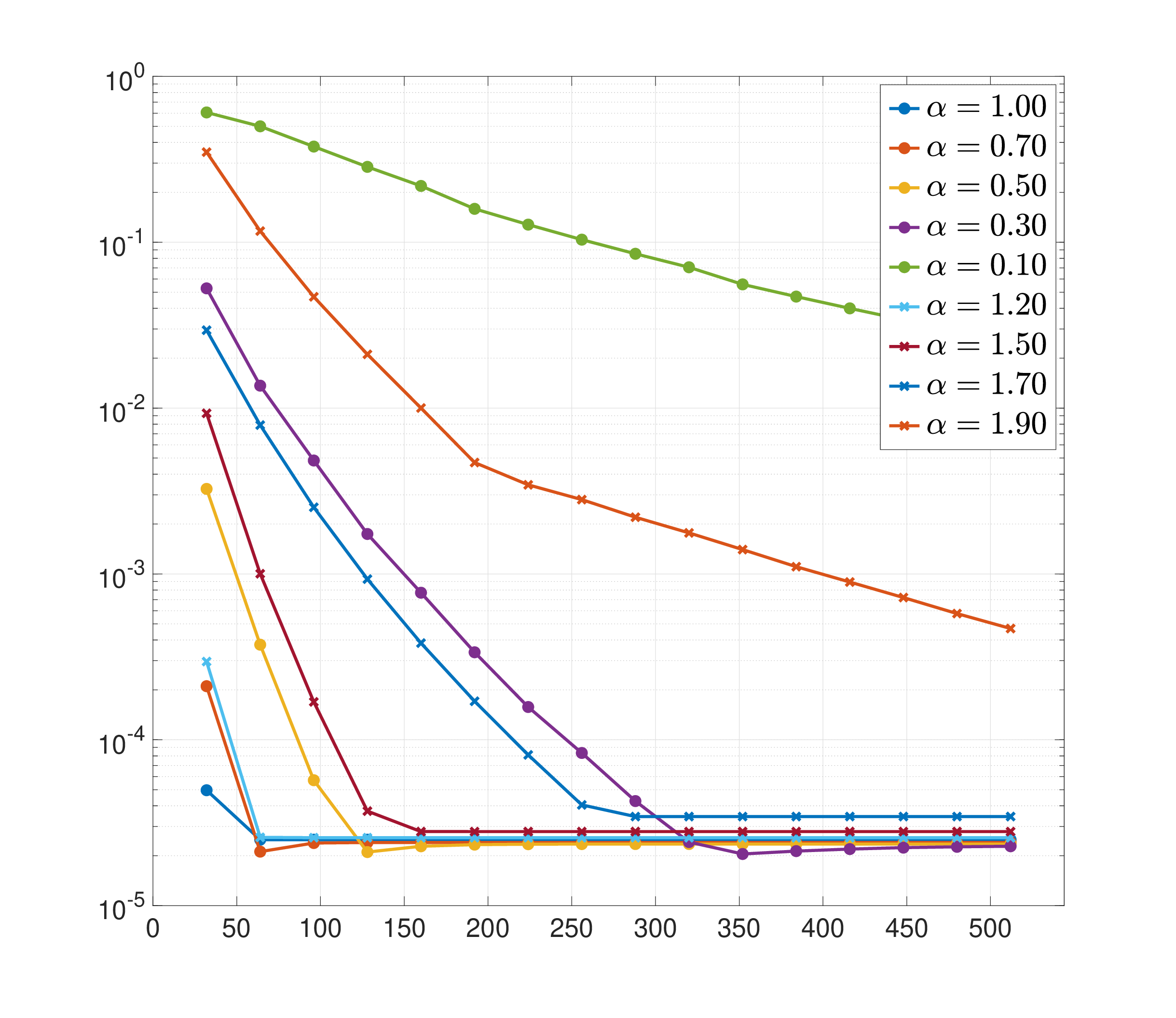}
		\else
			\hspace*{0.1em}
			\begin{overpic}[width=0.98\lentwosubfig, viewport=80 20 1100 960, clip=true]%
				{Ex3/FCP_Ex3_x2_err_vs_N_alpha_0.1_0.3_0.5_0.7_1.0_1.2_1.5_1.7_1.9_d_2_s_0_5_t_0_1_Nt_200_Nx_10}
				\put(42,84){\colorboxo{white}{\small $m = 10$} }
				\put(49,1){\small $N$}
				\put(-2,45){\small \rotatebox{90}{$\mathrm{err}$}}
				\put(8,86){\scriptsize ({\bf a})}
			\end{overpic}
			\hspace*{0.5em}
			\begin{overpic}[width=0.98\lentwosubfig, viewport=80 20 1100 960, clip=true]%
				{Ex3/FCP_Ex3_x2_err_vs_N_alpha_0.1_0.3_0.5_0.7_1.0_1.2_1.5_1.7_1.9_d_2_s_0_5_t_0_1_Nt_200_Nx_100}
				\put(42,84){\colorboxo{white}{\small $m = 10^2$} }
				\put(49,1){\small $N$}
				\put(-2,45){\small \rotatebox{90}{$\mathrm{err}$}}
				\put(8,86){\scriptsize ({\bf b})}
			\end{overpic}
		\fi
		\\
		\iflatexml
			\includegraphics[width=0.98\lentwosubfig, viewport=80 20 1100 960, clip=true]%
			{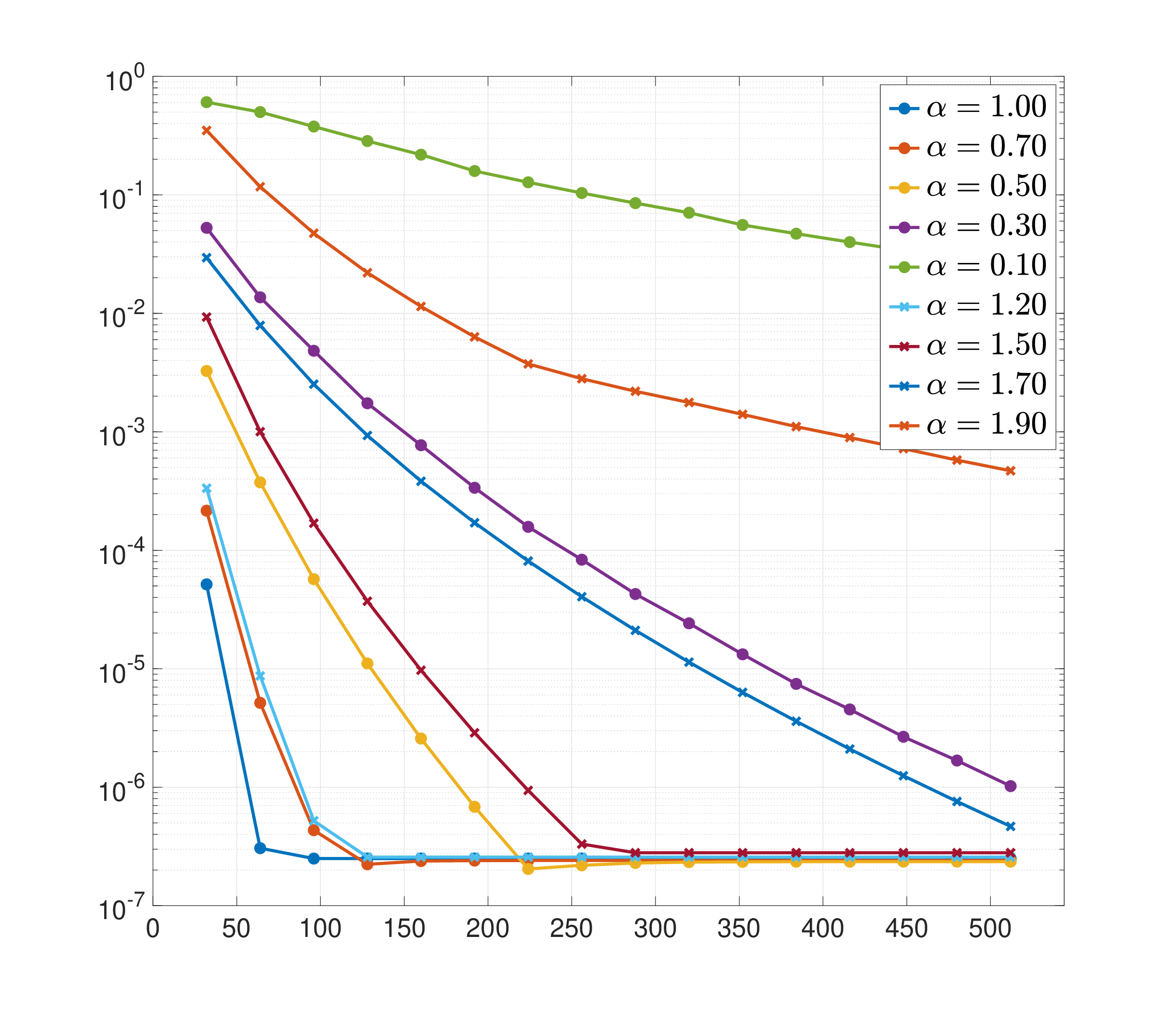}
			\includegraphics[width=0.98\lentwosubfig, viewport=80 20 1100 960, clip=true]%
			{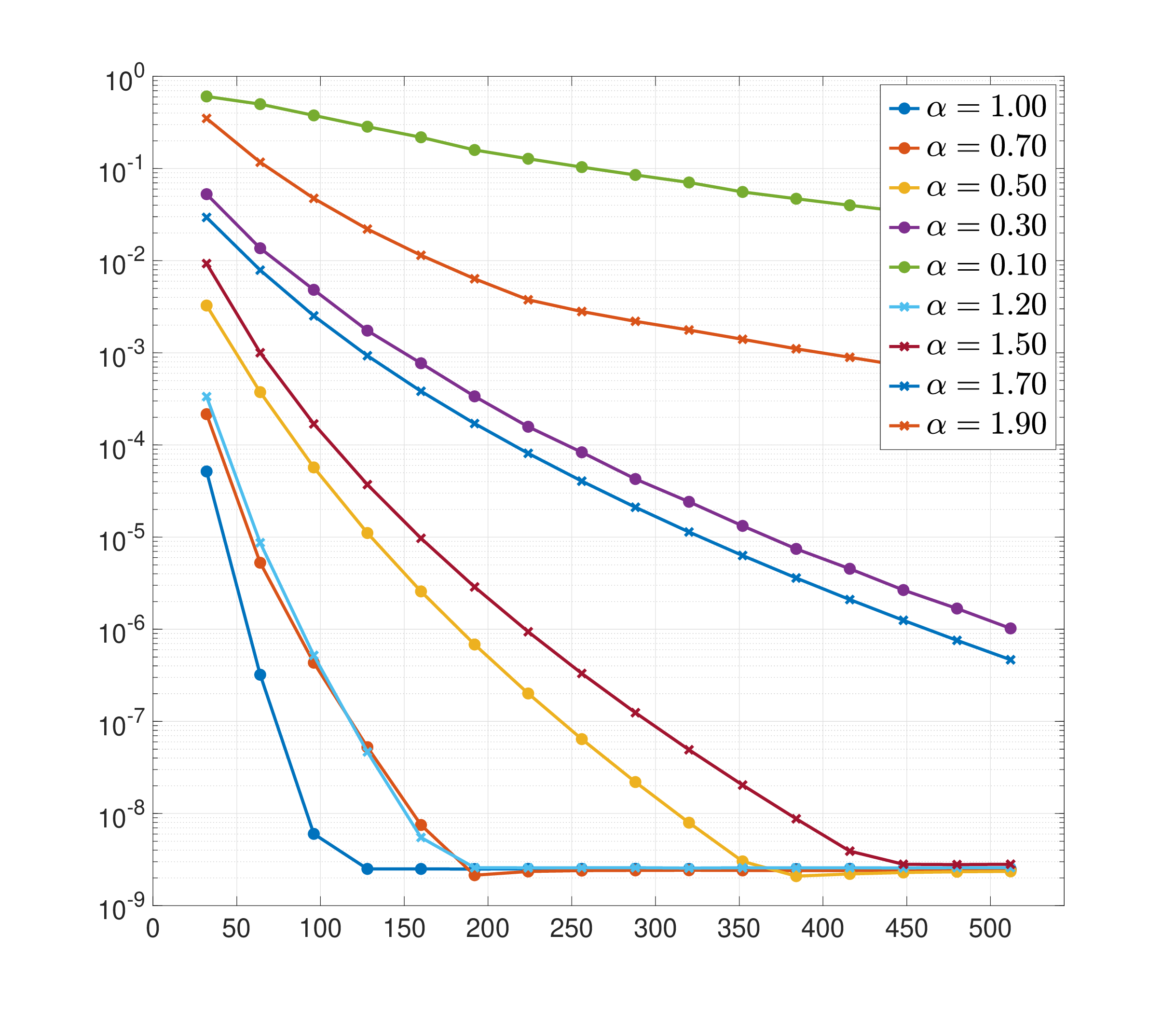}
		\else
			\hspace*{0.1em}
			\begin{overpic}[width=0.98\lentwosubfig, viewport=80 20 1100 960, clip=true]%
				{Ex3/FCP_Ex3_x2_err_vs_N_alpha_0.1_0.3_0.5_0.7_1.0_1.2_1.5_1.7_1.9_d_2_s_0_5_t_0_1_Nt_200_Nx_1000}
				\put(42,83){\colorboxo{white}{\small $m = 10^3$} }
				\put(49,1){\small $N$}
				\put(-2,45){\small \rotatebox{90}{$\mathrm{err}$}}
				\put(8,86){\scriptsize ({\bf c})}
			\end{overpic}
			\hspace*{0.5em}
			\begin{overpic}[width=0.98\lentwosubfig, viewport=80 20 1100 960, clip=true]%
				{Ex3/FCP_Ex3_x2_err_vs_N_alpha_0.1_0.3_0.5_0.7_1.0_1.2_1.5_1.7_1.9_d_2_s_0_5_t_0_1_Nt_200_Nx_10000}
				\put(42,83){\colorboxo{white}{\small $m = 10^4$} }
				\put(49,1){\small $N$}
				\put(-2,45){\small \rotatebox{90}{$\mathrm{err}$}}
				\put(8,86){\scriptsize ({\bf d})}
			\end{overpic}
		\fi
		\caption[Example 3: Error versus N for different space-grid sizes]{%
			Sup-norm error of the fully discretized approximation $\wt{u}_m^N$ to the solution of \eqref{eq:FCP_DE}, \eqref{eq:FCP_BC} with $A$, $f(t)$ defined by \eqref{eq:FCP_ex1_hom_A} and \eqref{eq:Ex3_f}, correspondingly; $L = 1$, $a = 1$, $\varphi_s = \pi/60$, $\gamma = \chi =1$, $T=1$.
			Errors are plotted for  $N=32, 64, 96, \ldots, 512$ and %
			(\textbf{a}) $m = 10$;
			(\textbf{b}) $m = 10^2$;
			(\textbf{c}) $m = 10^3$;
			(\textbf{d}) $m = 10^4$.
		}
		\label{fig:FCP_Ex3_ex_err_vs_N}
	\end{figure}
	\vspace*{-10pt}
\end{example}

The last example practically demonstrates that the only natural requirement imposed by our method on the spatial discretization scheme is for the discretized operator $\wt{A}$ to remain sectorial.
This suggests the possibility for the extension of the developed method to nonlinear problems, along the course discussed in the introduction.
In addition, the generalization of the method to Cauchy problems with the nonlocal-in-time condition~\cite{Gavrilyuk2010,nonloc_exMVS2014} also seems to be a promising direction of research, given its importance for the applications to final-value problems \cite{Wei2014,Jin2015}.

\section{Conclusions}
In this work, we proposed and justified the new exponentially convergent parallel numerical method for the fractional Cauchy problem \eqref{eq:FCP_DE}, \eqref{eq:FCP_BC}.
The constructed method is based on the approximation of mild solution representation \eqref{eq:FCP_InhomSol_rep} using the combination of efficient methods for the contour evaluation of the propagators $S_{\alpha,\beta}(t)$, $\beta = 1,2$ and tailored quadrature rules for the discretization of the Riemann--Liouiville and convolution integral operators from \eqref{eq:FCP_InhomSol_rep}.
As a result, the numerical evaluation of \eqref{eq:FCP_InhomSol_rep} is reduced to the solution of a sequence of independent linear stationary problems.
The accuracy estimates established by \cref{thm:FCP_prop_appr,thm:FCP_inhom_sol_err_est} remain valid uniformly in time for the entire range $\alpha \in (0,2)$, under the moderate smoothness assumptions $u_0 \in D(A^\gamma)$, $f'(z) \in D(A^{\chi}) $, with some $\gamma,\chi > 0$ and all $z \in D_d^2$, defined by \eqref{eq:FCP_D_d_2}.
These results recover the previously existing error estimates for parabolic problems \cite{gmv,gm5}, when  $\alpha$ is set to $1$ and $T< \infty$.
All the theoretical results are verified experimentally.
This includes the results from \cref{thm:FCP_hom_sol_err_est} and \cref{thm:FCP_inhom_sol_err_est} regarding the approximation of homogeneous and inhomogeneous parts of the solution, which are experimentally considered in \cref{ex:FCP_ex1_hom_R_eigenfunction,ex:FCP_ex2_inhom_R_eigenfunction}.
Here, we put extra effort to demonstrate that the constructed solution approximation is numerically stable for $\alpha \in [0.1,1.9]$ and practically capable of handling operators $A$ with a broad range of spectral characteristics.
It encompasses the class of so-called singularly perturbed operators, that are modeled in \cref{ex:FCP_ex1_hom_R_eigenfunction} by the Laplacian with a very small distance between the $\mathrm{Sp}(A)$ and the origin (see \cref{fig:FCP_Ex1_ex_err_vs_N}).
Additionally, in \cref{ex:FCP_ex3_inhom_R_FD}, we considered a fully discretized solution scheme for \eqref{eq:FCP_DE}, \eqref{eq:FCP_BC} to practically verify the robustness of the constructed approximation toward errors caused by the discretization in space.
Naturally, the mentioned benefits of the developed method come with some limitations.
Among such, we mention the required analyticity of $f(z)$ in the complex neighborhood $D_d^2$ of time interval $(0, T)$.
On the one hand, this is a considerably stronger assumption on the problem's right-hand side than the assumption $f \in W^{1,1}([0, T], X)$, imposed by the solution existence result from \cref{thm:FCP_sol_rep}.
On the other hand,  such analyticity assumptions are typical for the theory of exponentially convergent quadrature \cite{Davis1984,Stenger1993}.
Moreover, the quadrature rule chosen in \cref{thm:FCP_inhom_sol_err_est} permits for a practically realistic situation when $f'(t)$ has an integrable singularity at $t=0$.
The ability of the method to handle such class of $f$ was experimentally demonstrated in \cref{ex:FCP_ex3_inhom_R_FD}.
Another possible limitation of the current method is its practical viability only for moderate $T \leq 20$.
Nonetheless, existing numerical evidence suggests that the long-term stability of the method could be improved by some nonessential modifications.
Larger values of $T\approx 200 $ are necessary for certain parameter identification problems \cite{Zhokh2019}, which, along with the mentioned nonlinear and nonlocal extensions of the given problem, are going to be considered in the future works.

\iflatexml
{
	\footnotesize
	\ifx \undefined \Dbar \def \Dbar {\leavevmode\raise0.2ex\hbox{--}\kern-0.5emD}
  \fi\ifx \undefined \hckudot \def \hckudot#1{\ifmmode \setbox7
  \hbox{\accent20#1}\else \setbox7 \hbox{\accent20#1}\penalty 10000 \relax \fi
  \raise 1\ht7 \hbox{\raise.2ex \hbox to 1\wd7{\hss.\hss}}\penalty 10000
  \hskip-1\wd7 \penalty 10000\box7} \fi

}
\else
{
	\footnotesize
	\bibliographystyle{siamplain}
	\bibliography{references}

\begin{thebibliography}{10}

\bibitem{Ashyralyev2009}
{\sc A.~Ashyralyev}, {\em A note on fractional derivatives and fractional
  powers of operators}, Journal of Mathematical Analysis and Applications, 357
  (2009), pp.~232--236, \url{https://doi.org/10.1016/j.jmaa.2009.04.012}.

\bibitem{Baffet2017}
{\sc D.~Baffet and J.~S. Hesthaven}, {\em A kernel compression scheme for
  fractional differential equations}, SIAM Journal on Numerical Analysis, 55
  (2017), pp.~496--520, \url{https://doi.org/10.1137/15M1043960},
  \url{https://arxiv.org/abs/https://doi.org/10.1137/15M1043960}.

\bibitem{Bazhlekova1998}
{\sc E.~Bazhlekova}, {\em The abstract {Cauchy} problem for the fractional
  evolution equation}, Fract. Calc. Appl. Anal, 1 (1998), pp.~255--270.

\bibitem{Bazhlekova2001}
{\sc E.~Bazhlekova}, {\em Fractional evolution equations in Banach spaces}, PhD
  thesis, Department of Mathematics and Computer Science, 2001,
  \url{https://doi.org/10.6100/IR549476}.

\bibitem{Bohonova_2008}
{\sc T.~Y. Bohonova, I.~P. Gavrilyuk, V.~L. Makarov, and V.~Vasylyk}, {\em
  Exponentially convergent duhamel-like algorithms for differential equations
  with an operator coefficient possessing a variable domain in a banach space},
  SIAM J. Numer. Anal., 46 (2008), pp.~365--396,
  \url{https://doi.org/10.1137/06065252x}.

\bibitem{Colbrook2022a}
{\sc M.~J. Colbrook and L.~J. Ayton}, {\em A contour method for time-fractional
  {PDE}s and an application to fractional viscoelastic beam equations}, Journal
  of Computational Physics, 454 (2022), p.~110995,
  \url{https://doi.org/10.1016/j.jcp.2022.110995}.

\bibitem{Cuesta2006}
{\sc E.~Cuesta, C.~Lubich, and C.~Palencia}, {\em Convolution quadrature time
  discretization of fractional diffusion-wave equations}, Mathematics of
  Computation, 75 (2006), pp.~673--696,
  \url{https://doi.org/10.1090/s0025-5718-06-01788-1}.

\bibitem{Davis1984}
{\sc P.~J. Davis and P.~Rabinowitz}, {\em Methods of numerical integration},
  Computer science and applied mathematics, Academic Press, Boston, MA, 2~ed.,
  1984, \url{http://cds.cern.ch/record/278735}.
\newblock Includes examples.

\bibitem{Diethelm2020}
{\sc K.~Diethelm, R.~Garrappa, and M.~Stynes}, {\em Good (and not so good)
  practices in computational methods for fractional calculus}, Mathematics, 8
  (2020), p.~324, \url{https://doi.org/10.3390/math8030324}.

\bibitem{Diethelm2019}
{\sc K.~Diethelm and G.~Karniadakis}, {\em Fundamental approaches for the
  numerical handling of fractional operators and time-fractional differential
  equations}, Handbook of Fractional Calculus with Applications, 3 (2019),
  pp.~1--22, \url{https://doi.org/10.1515/9783110571684-001}.

\bibitem{Diethelm2022}
{\sc K.~Diethelm, V.~Kiryakova, Y.~Luchko, J.~Machado, and V.~E. Tarasov}, {\em
  Trends, directions for further research, and some open problems of fractional
  calculus}, Nonlinear Dynamics,  (2022), pp.~1--26,
  \url{https://doi.org/10.1007/s11071-021-07158-9}.

\bibitem{Dzherbashian2020}
{\sc M.~M. Dzherbashian and A.~B. Nersesian}, {\em Fractional derivatives and
  {Cauchy} problem for differential equations of fractional order}, Fractional
  Calculus and Applied Analysis, 23 (2020), pp.~1810--1836,
  \url{https://doi.org/10.1515/fca-2020-0090}.

\bibitem{Eidelman2004}
{\sc S.~D. Eidelman and A.~N. Kochubei}, {\em {Cauchy} problem for fractional
  diffusion equations}, Journal of differential equations, 199 (2004),
  pp.~211--255, \url{https://doi.org/10.1016/j.jde.2003.12.002}.

\bibitem{Fischer2019}
{\sc M.~Fischer}, {\em Fast and parallel {Runge--Kutta} approximation of
  fractional evolution equations}, {SIAM} Journal on Scientific Computing, 41
  (2019), pp.~A927--A947, \url{https://doi.org/10.1137/18m1175616}.

\bibitem{Fritz2022}
{\sc M.~Fritz, M.~L. Rajendran, and B.~Wohlmuth}, {\em Time-fractional
  {Cahn--Hilliard} equation: Well-posedness, degeneracy, and numerical
  solutions}, Computers \& Mathematics with Applications, 108 (2022),
  pp.~66--87.

\bibitem{fujita}
{\sc H.~Fujita, N.~Saito, and T.~Suzuki}, {\em Operator Theory and Numerical
  Methods}, Elsevier, Heidelberg, 2001.

\bibitem{Gander2013}
{\sc M.~J. Gander and S.~G{\"u}ttel}, {\em {PARAEXP}: A parallel integrator for
  linear initial-value problems}, {SIAM} Journal on Scientific Computing, 35
  (2013), pp.~C123--C142, \url{https://doi.org/10.1137/110856137},
  \url{https://arxiv.org/abs/https://doi.org/10.1137/110856137}.

\bibitem{Garrappa2015}
{\sc R.~Garrappa}, {\em Numerical evaluation of two and three parameter
  {Mittag-Leffler} functions}, SIAM Journal on Numerical Analysis, 53 (2015),
  pp.~1350--1369, \url{https://doi.org/10.1137/140971191}.

\bibitem{Garrappa2015a}
{\sc R.~Garrappa}, {\em Trapezoidal methods for fractional differential
  equations: Theoretical and computational aspects}, Mathematics and Computers
  in Simulation, 110 (2015), pp.~96--112,
  \url{https://doi.org/10.1016/j.matcom.2013.09.012}.

\bibitem{Garrappa2018}
{\sc R.~Garrappa}, {\em Numerical solution of fractional differential
  equations: A survey and a software tutorial}, Mathematics, 6 (2018), p.~16.

\bibitem{gm5}
{\sc I.~Gavrilyuk and V.~Makarov}, {\em Exponentially convergent algorithms for
  the operator exponential with applications to inhomogeneous problems in
  {B}anach spaces}, {S}{I}{A}{M} {J}ournal on Numerical Analysis, 43 (2005),
  pp.~2144--2171, \url{https://doi.org/10.1137/040611045}.

\bibitem{gmv}
{\sc I.~Gavrilyuk, V.~Makarov, and V.~Vasylyk}, {\em A new estimate of the
  {S}inc method for linear parabolic problems including the initial point},
  Computational Methods in Applied Mathematics (CMAM), 4 (2004), pp.~163--179,
  \url{https://doi.org/10.2478/cmam-2004-0009}.

\bibitem{bGavrilyuk2011}
{\sc I.~Gavrilyuk, V.~Makarov, and V.~Vasylyk}, {\em Exponentially convergent
  algorithms for abstract differential equations}, Frontiers in Mathematics,
  Birkh\"auser/Springer Basel AG, Basel, 2011,
  \url{https://doi.org/10.1007/978-3-0348-0119-5}.

\bibitem{Gavrilyuk2010}
{\sc I.~P. Gavrilyuk, V.~L. Makarov, D.~O. Sytnyk, and V.~B. Vasylyk}, {\em
  Exponentially convergent method for the m-point nonlocal problem for a first
  order differential equation in banach space}, {Numerical Functional Analysis
  and Optimization}, 31 (2010), pp.~1--21,
  \url{https://doi.org/10.1080/01630560903499019}.

\bibitem{Gonzalez2016}
{\sc C.~Gonz{\'a}lez and M.~Thalhammer}, {\em Higher-order exponential
  integrators for quasi-linear parabolic problems. part ii: Convergence}, SIAM
  Journal on Numerical Analysis, 54 (2016), pp.~2868--2888,
  \url{https://doi.org/10.1137/15m103384}.

\bibitem{Guo2019}
{\sc L.~Guo, F.~Zeng, I.~Turner, K.~Burrage, and G.~E. Karniadakis}, {\em
  Efficient multistep methods for tempered fractional calculus: Algorithms and
  simulations}, {SIAM} Journal on Scientific Computing, 41 (2019),
  pp.~A2510--A2535, \url{https://doi.org/10.1137/18m1230153}.

\bibitem{Hochbruck2005}
{\sc M.~Hochbruck and A.~Ostermann}, {\em Explicit exponential {Runge--Kutta}
  methods for semilinear parabolic problems}, SIAM Journal on Numerical
  Analysis, 43 (2005), pp.~1069--1090, \url{https://doi.org/10.1137/040611434},
  \url{https://arxiv.org/abs/https://doi.org/10.1137/040611434}.

\bibitem{jin2019numerical}
{\sc B.~Jin, R.~Lazarov, and Z.~Zhou}, {\em Numerical methods for
  time-fractional evolution equations with nonsmooth data: a concise overview},
  Computer Methods in Applied Mechanics and Engineering, 346 (2019),
  pp.~332--358, \url{https://doi.org/10.1016/j.cma.2018.12.011}.

\bibitem{Jin2015}
{\sc B.~Jin and W.~Rundell}, {\em A tutorial on inverse problems for anomalous
  diffusion processes}, Inverse problems, 31 (2015), p.~035003,
  \url{https://doi.org/10.1088/0266-5611/31/3/035003}.

\bibitem{Kadalbajoo2010}
{\sc M.~K. Kadalbajoo and V.~Gupta}, {\em A brief survey on numerical methods
  for solving singularly perturbed problems}, Applied mathematics and
  computation, 217 (2010), pp.~3641--3716,
  \url{https://doi.org/10.1016/j.amc.2010.09.059}.

\bibitem{Keyantuo2013}
{\sc V.~Keyantuo, C.~Lizama, and M.~Warma}, {\em Spectral criteria for
  solvability of boundary value problems and positivity of solutions of
  time-fractional differential equations}, in Abstract and Applied Analysis,
  vol.~2013, Hindawi, Hindawi Limited, 2013, pp.~1--11,
  \url{https://doi.org/10.1155/2013/614328}.

\bibitem{Khristenko2021}
{\sc U.~Khristenko and B.~Wohlmuth}, {\em Solving time-fractional differential
  equations via rational approximation}, {IMA} Journal of Numerical Analysis,
  (2021), \url{https://doi.org/10.1093/imanum/drac022}.

\bibitem{Kilbas2006}
{\sc A.~Kilbas, H.~Srivastava, and J.~J. Trujillo}, {\em Theory and
  applications of fractional differential equations}, vol.~204, elsevier, 2006.

\bibitem{Kochubei1989}
{\sc A.~N. Kochubei}, {\em The {Cauchy} problem for evolution equations of
  fractional order}, Differentsial'nye Uravneniya, 25 (1989), pp.~1359--1368.

\bibitem{Li2010}
{\sc M.~Li, C.~Chen, and F.-B. Li}, {\em On fractional powers of generators of
  fractional resolvent families}, Journal of Functional Analysis, 259 (2010),
  pp.~2702--2726, \url{https://doi.org/10.1016/j.jfa.2010.07.007}.

\bibitem{LopezFernandez2005}
{\sc M.~L{\'{o}}pez-Fern{\'{a}}ndez, C.~Lubich, C.~Palencia, and A.~Schädle},
  {\em Fast {Runge-Kutta} approximation of inhomogeneous parabolic equations},
  Numerische Mathematik, 102 (2005), pp.~277--291,
  \url{https://doi.org/10.1007/s00211-005-0624-3}.

\bibitem{lopez-fernandez1}
{\sc M.~L{\'o}pez-{F}ern{\'a}ndez, C.~Palencia, and A.~Sch{\"a}dle}, {\em A
  spectral order method for inverting sectorial {Laplace} transforms}, SIAM J.
  Numer. Anal., 44 (2006), pp.~1332--1350.

\bibitem{Lubich2004}
{\sc C.~Lubich}, {\em Convolution quadrature revisited}, BIT Numerical
  Mathematics, 44 (2004), pp.~503--514.

\bibitem{nonloc_exMVS2014}
{\sc V.~L. Makarov, D.~O. Sytnyk, and V.~B. Vasylyk}, {\em Existence of the
  solution to a nonlocal-in-time evolutional problem}, {Nonlinear Analysis:
  Modelling and Control}, 19 (2014), pp.~432--447,
  \url{https://arxiv.org/abs/1406.5417}.

\bibitem{Martinez2001}
{\sc C.~Martinez and M.~Sanz}, {\em The theory of fractional powers of
  operators}, Elsevier, 2001.

\bibitem{McLean2021}
{\sc W.~McLean}, {\em Numerical evaluation of {Mittag-Leffler} functions},
  Calcolo, 58 (2021), p.~7, \url{https://doi.org/10.1007/s10092-021-00398-6}.

\bibitem{McLTh}
{\sc W.~McLean and V.~Thomee}, {\em Time discretization of an evolution
  equation via {Laplace} transforms}, IMA Journal of Numerical Analysis, 24
  (2004), pp.~439--463, \url{https://doi.org/10.1093/imanum/24.3.439}.

\bibitem{McLean2010a}
{\sc W.~McLean and V.~Thom{\'e}e}, {\em Maximum-norm error analysis of a
  numerical solution via {Laplace} transformation and quadrature of a
  fractional-order evolution equation}, {IMA} Journal of Numerical Analysis, 30
  (2009), pp.~208--230, \url{https://doi.org/10.1093/imanum/drp004}.

\bibitem{McLean2010}
{\sc W.~McLean and V.~Thom{\'e}e}, {\em Numerical solution via {Laplace}
  transforms of a fractional order evolution equation}, The Journal of Integral
  Equations and Applications,  (2010), pp.~57--94,
  \url{https://doi.org/10.1216/jie-2010-22-1-57}.

\bibitem{Oldham1974}
{\sc K.~Oldham and J.~Spanier}, {\em The fractional calculus theory and
  applications of differentiation and integration to arbitrary order},
  Elsevier, 1974.

\bibitem{Pang2016}
{\sc H.-K. Pang and H.-W. Sun}, {\em Fast numerical contour integral method for
  fractional diffusion equations}, Journal of Scientific Computing, 66 (2016),
  pp.~41--66.

\bibitem{Rieder2023}
{\sc A.~Rieder}, {\em Double exponential quadrature for fractional diffusion},
  Numerische Mathematik, 153 (2023), pp.~1--52,
  \url{https://doi.org/10.1007/s00211-022-01342-8}.

\bibitem{Roos2008}
{\sc H.-G. Roos, M.~Stynes, and L.~Tobiska}, {\em Robust Numerical Methods for
  Singularly Perturbed Differential Equations: Convection-Diffusion-Reaction
  and Flow Problems}, Springer Series in Computational Mathematics 24,
  Springer-Verlag Berlin Heidelberg, 2~ed., 2008.

\bibitem{Schneider1989}
{\sc W.~R. Schneider and W.~Wyss}, {\em Fractional diffusion and wave
  equations}, Journal of Mathematical Physics, 30 (1989), pp.~134--144,
  \url{https://doi.org/10.1063/1.528578}.

\bibitem{Seybold2009}
{\sc H.~Seybold and R.~Hilfer}, {\em Numerical algorithm for calculating the
  generalized {Mittag-Leffler} function}, SIAM journal on numerical analysis,
  47 (2009), pp.~69--88.

\bibitem{Stenger1993}
{\sc F.~Stenger}, {\em Numerical methods based on sinc and analytic functions},
  vol.~20 of Springer Series in Computational Mathematics, Springer-Verlag, New
  York, 1993, \url{https://doi.org/10.1007/978-1-4612-2706-9}.

\bibitem{Stynes_2019}
{\sc M.~Stynes}, {\em Singularities}, in Handbook of Fractional Calculus with
  Applications, G.~E. Karniadakis, ed., vol.~3, De Gruyter, Apr. 2019,
  pp.~287--306, \url{https://doi.org/10.1515/9783110571684-011}.

\bibitem{Stynes2021}
{\sc M.~Stynes}, {\em A survey of the l1 scheme in the discretisation of
  time-fractional problems}, Numerical Mathematics: Theory, Methods and
  Applications, 15 (2021), pp.~1173--1192,
  \url{https://doi.org/10.4208/nmtma.oa-2022-0009s}.

\bibitem{SytnykWohlmuth2023}
{\sc D.~Sytnyk and B.~Wohlmuth}, {\em {Abstract Fractional Cauchy Problem:
  Existence of Propagators and Inhomogeneous Solution Representation}}, Fractal
  and Fractional, 7 (2023), p.~698,
  \url{https://doi.org/10.3390/fractalfract7100698}.

\bibitem{thomee1}
{\sc V.~Thom{\'e}e}, {\em A high order parallel method for time discretization
  of parabolic type equations based on {Laplace} transformation and
  quadrature}, Int. J. Numer. Anal. Model., 2 (2005), pp.~121--139.

\bibitem{Umarov2019}
{\sc S.~Umarov}, {\em Fractional {Duhamel} principle}, Handbook of Fractional
  Calculus with Applications, 2 (2019), pp.~383--410,
  \url{https://doi.org/10.1515/9783110571660-017}.

\bibitem{Vasylyk2022}
{\sc V.~Vasylyk, I.~Gavrilyuk, and V.~Makarov}, {\em Exponentially convergent
  method for the approximation of a differential equation with fractional
  derivative and unbounded operator coefficient in a banach space}, Ukrainian
  Mathematical Journal, 74 (2022), pp.~171--185,
  \url{https://doi.org/10.1007/s11253-022-02056-8}.

\bibitem{Wei2014}
{\sc T.~Wei and J.-G. Wang}, {\em A modified quasi-boundary value method for
  the backward time-fractional diffusion problem}, ESAIM: Mathematical
  modelling and numerical analysis, 48 (2014), pp.~603--621,
  \url{https://doi.org/10.1051/m2an/2013107}.

\bibitem{WeidTref1}
{\sc J.~Weideman and L.~Trefethen}, {\em Parabolic and hyperbolic contours for
  computing the bromwich integral}, Math. Comp., 76 (2007), pp.~1341--1356,
  \url{https://doi.org/10.1090/s0025-5718-07-01945-x}.

\bibitem{Weideman2006}
{\sc J.~A.~C. Weideman}, {\em Optimizing {Talbot's} contours for the inversion
  of the {Laplace} transform}, SIAM Journal on Numerical Analysis, 44 (2006),
  pp.~2342--2362, \url{https://doi.org/10.1137/050625837}.

\bibitem{Weideman2010}
{\sc J.~A.~C. Weideman}, {\em Improved contour integral methods for parabolic
  {PDE}s}, IMA J. Numer. Anal., 30 (2010), pp.~334--350.

\bibitem{Zhokh2019}
{\sc A.~Zhokh and P.~Strizhak}, {\em Macroscale modeling the methanol anomalous
  transport in the porous pellet using the time-fractional diffusion and
  fractional brownian motion: A model comparison}, Communications in Nonlinear
  Science and Numerical Simulation, 79 (2019), p.~104922,
  \url{https://doi.org/10.1016/j.cnsns.2019.104922}.

\end{thebibliography}
}
\fi
\clearpage
\appendix
\section{Corrigendum}
This section contains a list of the post-publication corrections that are necessary to make the publication:
\begin{description}
	\item[{[*]}] Sytnyk, D. and Wohlmuth, B. 2023. Exponentially Convergent Numerical Method for Abstract Cauchy Problem with Fractional Derivative of Caputo Type. Mathematics. 11, 10 (2023). DOI:https://doi.org/10.3390/math11102312
\end{description}
up-to-date with the current manuscript.
\begin{enumerate}
	\item In the formulation of Theorem 1 on p. 2 and in the text of second paragraph on p. 4 [*]: formula \textcolor{red}{$\varphi_s < \tfrac{\pi}{2}\min\left\{1,\alpha^{-1}\right \}$} should be read as \textcolor{blue}{$\varphi_s < \pi\min{\left\{\frac{1}{2}, \left(1 - \tfrac{\alpha}{2}\right)\right\}}$}.
	      The formula \textcolor{red}{$-\Sigma(\rho_s, \varphi_s)$} in the last sentence of the theorem's formulation should be read as \textcolor{blue}{$-\Sigma(\rho_s, \varphi_s) \cup \{0\}$}.

	\item Caption of Figure 1 at p.6 [*]: formula \textcolor{red}{$\rho_s = -\pi$} should be read as \textcolor{blue}{$\rho_s = \pi$}

	\item Last formula at p.7 [*] should be read as \textcolor{blue}{$\tan{(\phi_s - 2d)}$}.

	\item The text from the beginning of p. 8 [*] down to, and including, formula (14) should be read as:
	      \textcolor{blue}{which, after back-substitution, implies $\tan{(\phi_s - 2d)} = \tan{\phi_c}$.
		      Due to the constraints on $d$, $\phi_c$, $\phi_s$ we are interested only in the fol\-lowing solution
		      of the last equation:
		      \[
			      \hspace*{10em}d = \frac{1}{2}\left(\phi_s - \phi_c \right). \hspace*{14em} (14)
		      \]
	      }

	\item Lemma 1 on p. 9 [*]: resolvent part of formulas (18), (19) for $F_{\alpha,1} (\xi)$ and $F_{\alpha,2} (\xi)$ \textcolor{red}{$(z^{\alpha} I + A)^{-1}$} should be read as \textcolor{blue}{$(z^{\alpha}(\xi) I + A)^{-1}$}

	\item Paragraph around (22) on p. 11  [*]:\par
	      \textcolor{red}{
		      where
		      $
			      D_d(\epsilon)=\{z \in \mathbb{C}:\; | \textrm{Re}(z)| < 1/\epsilon, \
			      |\textrm{Im}(z)|<d(1-\epsilon)\}
		      $
		      and $\partial D_d(\epsilon)$ is the boundary of $ D_d(\epsilon)$.
		      The truncation errors of (21) satisfy the estimate [31,50]:
		      \begin{equation*}
			      \left \|S_{\alpha,\beta}(t)  - \wt{S}_{\alpha,\beta}^\infty(t) \right \|	\leq  \frac{e^{-\pi d/h}}{2 \sinh (\pi d/h)}\|{\cF_{\alpha,\beta}}(t,\cdot)\|_{{\bf H}^1(D_d)} .\hspace*{6em} (22)
		      \end{equation*}
		      Thus, in order to bound the truncation error, one needs to obtain estimates for the ${{\bf H}^1(D_d)}$ norms of the functions $\cF_{\alpha,\beta}(t,z)$, $\beta =1,2$.}\par
	      should be read as: \par
	      \textcolor{blue}{
		      where
		      $
			      D_d(\epsilon)=\{z \in \mathbb{C}:\; | \Re(z)| < 1/\epsilon, \
			      |\Im(z)|<d(1-\epsilon)\}
		      $
		      and $\partial D_d(\epsilon)$ is the boundary of $ D_d(\epsilon)$.
		      The discretization error of (21) satisfy the estimate [31,50]:
		      \begin{equation*}
			      \left \|S_{\alpha,\beta}(t) x_\beta - \wt{S}_{\alpha,\beta}^\infty(t) x_\beta \right \|	\leq  \frac{e^{-\pi d/h}}{2 \sinh (\pi d/h)}\|{\cF_{\alpha,\beta}}(t,\cdot)\|_{{\bf H}^1(D_d)} . \hspace*{4em} (22)
		      \end{equation*}
		      Thus, in order to bound this term, one needs to obtain estimates for the ${{\bf H}^1(D_d)}$ norms of the functions $\cF_{\alpha,\beta}(t,z)$, $\beta =1,2$.}

	\item Theorem 2 on p. 13 [*]:
	      The formula  \textcolor{red}{$\alpha \varphi_s < \tfrac{\pi}{2}$, $\gamma \in (0,1)$}  of second sentence should be read as
	      \textcolor{blue}{$\varphi_s < \pi\left(1 - \tfrac{\alpha}{2}\right)$,} \par
	      The sentence after eq. (29) should be read as:
	      \textcolor{blue}{Here, $d = \frac{\phi_\alpha - \phi_c}{2}$, $\phi_\alpha = \min\{\pi, \frac{\pi -\varphi_s}{\alpha}\}$   and $\phi_c \in \left[\tfrac{\pi}{2}, \phi_\alpha\right)$, $a_0>0$  are given.}

	\item The second paragraph on p. 16: formula \textcolor{red}{$z \in {\mathbb C} \setminus \mathrm{Sp}(A)$} should be read as \textcolor{blue}{$z \in {\mathbb C} \setminus \mathrm{Sp}(-A)$}.

	\item The end of the first paragraph on p. 17: formula \textcolor{red}{$d = \frac{\phi_\alpha + \varphi_c - \pi}{2}$} should be read as \textcolor{blue}{(14)}.

	\item Theorem 3 on p. 22 [*]: the part of the sentence after eq. (45) that reads as:
	      \textcolor{red}{such that $\alpha \varphi_s < \tfrac{\pi}{2}$, }
	      should be omitted. The sentence after eq. (48) should be read as:
	      \textcolor{blue}{Here, $d = \frac{\phi_\alpha - \phi_c}{2}$, $\phi_\alpha = \min\left\{\pi, \frac{\pi -\varphi_s}{\alpha}\right\}$  and $\phi_c \in \left[ \tfrac{\pi}{2}, \phi_\alpha \right)$, $a_0>0$  are given.}
	\item The last sentence of the first paragraph on p. 28 [*]: The word \textcolor{red}{type-dependent} should be read as \textcolor{blue}{time-dependent}.
	\item The bibliography reference 10 on p. 33 [*] should be read as:
	      \textcolor{blue}{Sytnyk, D. and Wohlmuth, B. 2023. Abstract Fractional Cauchy Problem: Existence of Propagators and Inhomogeneous Solution Representation. Fractal and Fractional. 7, 10 (Sep. 2023), 698. DOI:https://doi.org/10.3390/fractalfract7100698}
\end{enumerate}

\end{document}